\documentclass[11pt,oneside]{amsart}
\usepackage{amscd,amssymb}

\theoremstyle{plain}
\usepackage{color}

\newcommand{\bGS}{\bG_S}

\usepackage{epsfig}
\usepackage{graphicx}
\theoremstyle{plain}

\newtheorem{Thm}[equation]{Theorem}

\newtheorem{Ex}[equation]{Example}
\newtheorem{Cor}[equation]{Corollary}
\newtheorem{Prop}[equation]{Proposition}
\newtheorem{Lem}[equation]{Lemma}
\newtheorem{Def}[equation]{Definition}

\newtheorem{Rem}[equation]{Remark}
\numberwithin{equation}{section}

\newcommand{\mc}[1]{{}}

\newcommand{\q}{\mathbb{Q}}

\newcommand{\e}{\epsilon}
\newcommand{\z}{\mathbb{Z}}
\renewcommand{\q}{\mathbb{Q}}
\newcommand{\n}{\mathbb{N}}
\renewcommand{\c}{\mathbb{C}}
\newcommand{\br}{\mathbb{R}}
\newcommand{\qp}{\mathbb{Q}_p}

\newcommand{\zp}{\mathbb{Z}_p}
\newcommand{\A}{\mathbb{A}}
\newcommand{\ba}{\backslash}

\newcommand{\G}{\Gamma}
\newcommand{\bP}{\mathbf P}
\newcommand{\bU}{\mathbf U}
\newcommand{\bQ}{\mathbf Q}
\newcommand{\bH}{\mathbf H}
\newcommand{\bN}{\mathbf N}

\newcommand{\He}{\operatorname{H}}
\newcommand{\Cal}{\mathcal}

\newcommand{\bG}{\mathbf G}

\newcommand{\bM}{\mathbf M}
\newcommand{\bX}{\mathbf X}
\newcommand{\bY}{\mathbf Y}
\newcommand{\SL}{\operatorname{SL}}\newcommand{\GL}{\operatorname{GL}}
\newcommand{\PGL}{\operatorname{PGL}}
\newcommand{\SO}{\operatorname{SO}}
\newcommand{\Sp}{\operatorname{Sp}}
\newcommand{\MT}{\operatorname{MT}}

\newcommand{\Ad}{\operatorname{Ad}}
\newcommand{\vol}{\operatorname{vol}}
\newcommand{\s}{\sigma}

\renewcommand{\H}{\operatorname H}
\newcommand{\bi}{\begin{itemize}}

\newcommand{\ei}{\end{itemize}}
\newcommand{\be}{\begin{enumerate}}

\newcommand{\ee}{\end{enumerate}}

\newcommand{\op}{\operatorname}

\newcommand{\vs}{\vskip 5pt}
\newcommand{\bga}{\bG(\A)}
\newcommand{\bL}{\mathbf L}
\renewcommand{\s}{\operatorname{s}}

\newcommand{\kk}{{K}}
\newcommand{\kv}{{\kk_v}}
\newcommand{\kbar}{{\overline{\kk}}}
\newcommand{\kinf}{{\kk_\infty}}
\newcommand{\hh}{{\mathfrak{h}}}
\newcommand{\geff}{{\mathfrak{g}}}
\newcommand{\Gbar}{{\overline{G}}}
\newcommand{\Hbar}{{\overline{H}}}
\newcommand{\id}{{1}}
\renewcommand{\implies}{$\Rightarrow$}

\newcommand{\tors}{{_{\textup{tors}}}}

\renewcommand{\sc}{^{\textup{sc}}}
\newcommand{\xx}{{\textsf{\upshape X}_*}}

\newcommand{\ad}{^{\textup{ad}}}

\newcommand{\Spin}{\textup{Spin}}

\newcommand{\PSL}{\textup{PSL}}
\newcommand{\PSp}{\textup{PSp}}
\newcommand{\PSO}{\textup{PSO}}
\newcommand{\PSU}{{\rm PSU}}

\newcommand{\Lie}{\textup{Lie}}
\newcommand{\pp}{\mathfrak{p}}

\newcommand{\im}{\textup{im}}

\renewcommand{\AA}{{\mathbb{A}}}
\newcommand{\QQ}{{\mathbb{Q}}}
\newcommand{\ZZ}{{\mathbb{Z}}}
\newcommand{\CC}{{\mathbb{C}}}
\newcommand{\RR}{{\mathbb{R}}}
\newcommand{\Zz}{{\ZZ}}

\newcommand{\into}{\hookrightarrow}
\newcommand{\Hom}{\textup{Hom}}
\newcommand{\coker}{\textup{coker}}
\newcommand{\Gal}{\textup{Gal}}
\newcommand{\Aut}{\textup{Aut}}
\newcommand{\ab}{{\textup{ab}}}
\newcommand{\loc}{{\textup{loc}}}

\newcommand{\isoto}{\overset{\sim}{\to}}

\newcommand{\sV}{{R}}
\newcommand{\Vinf}{{\sV_\infty}}
\newcommand{\labelto}[1]{\xrightarrow{\makebox[1.5em]{\scriptsize ${#1}$}}}
\usepackage{eucal}
\usepackage{verbatim}
\usepackage[arrow,matrix]{xy}
\renewcommand{\gg}{{\mathfrak{g}}}
\newcommand{\gv}{{\gg_v}}
\DeclareMathOperator*{\fibreprod}{\times}
\DeclareFontEncoding{OT2}{}{} % to enable usage of cyrillic fonts
\DeclareTextFontCommand{\textcyr}{\fontencoding{OT2}\fontfamily{wncyr}\fontseries{m}\fontshape{n}\selectfont}
\newcommand{\Sha}{\textcyr{Sh}}

\def\noi{\par\noindent}
\newcommand{\GG}{{\mathbf{G}}}

\newcommand{\kabf}{{k^f_\ab}}

\theoremstyle{plain}
\newtheorem{theorem}[subsubsection]{Theorem}

\newtheorem{lemma}[subsubsection]{Lemma}
\newtheorem{corollary}[subsubsection]{Corollary}

\theoremstyle{definition}

\newtheorem*{definition*}{Definition}
\newtheorem*{example*}{Example}

\newtheorem{subsec}[subsubsection]{}

\begin{document}

\title[Rational points and Adelic periods]{Rational points on homogeneous varieties and
equidistribution of Adelic periods }
\author{Alex Gorodnik and Hee Oh \\ (with appendix by Mikhail Borovoi)}

\address{Gorodnik: School of Mathematics, University of Bristol, Bristol BS8 1TW, U.K.}
\email{a.gorodnik@bristol.ac.uk}

\address{Oh: Mathematics department, Brown university, Providence, RI
and Korea Institute for Advanced Study, Seoul, Korea}\email{heeoh@math.brown.edu}

\address{Borovoi: Raymond and Beverly Sackler School of Mathematical Sciences, Tel Aviv University, 69978 Tel Aviv, Israel}
\email{borovoi@post.tau.ac.il}

\begin{abstract}Let $\bU:=\bL\ba \bG$ be a homogeneous variety defined over a number
field $K$, where $\bG$ is a connected semisimple $K$-group and
$\bL$ is a connected maximal semisimple $K$-subgroup of $\bG$ with finite index in its normalizer.
 Assuming that $\bG(K_v)$ acts
transitively on $\bU(K_v)$ for almost all places $v$ of $K$, we
obtain the asymptotics as $T\to\infty$ of the number of rational
points in $\bU(K)$ with height bounded by $T$,
and settle new cases of Manin's conjecture for 
many wonderful varieties. The main ingredient of our approach is the
equidistribution of semisimple adelic periods, which is established using the theory of unipotent flows.
\end{abstract}

\thanks{A.G. is  partially supported by NSF grant 0654413 and  RCUK Fellowship,
Oh is  partially supported by NSF grant 0629322,
M.B. is partially supported by the Hermann  Minkowski Center for Geometry
and by the ISF grant 807/07}

\maketitle
\tableofcontents

\section{Introduction}
Let $K$ be a number field and $\bX$ a projective variety defined over $K$.
Understanding the set $\bX(K)$ of $K$-rational points in $\bX$ is a fundamental
problem in arithmetic geometry.
In this paper we study the asymptotic number (as $T\to \infty$) of the
points in $\bX(K)$ of height less than $T$ for
compactifications of affine homogeneous varieties $\bU=\bL\ba \bG$
of a connected semisimple algebraic $K$-group $\bG$ when $\bL$ is
a semisimple maximal connected $K$-subgroup.
%When $\bU$ admits a wonderful compactification,
Our results
solve new cases of Manin's conjecture \cite{BM} on rational points of Fano varieties.
%The tools introduced in this paper
% are completely different from those used in previous works, namely
%we relate this problem with the equdistribution of adelic periods, and
%use the theory of unipotent flows to solve the required equidistribution problem.

Manin's conjecture has been proved for
equivariant compactifications of homogeneous spaces:
flag varieties (\cite{FMT}, \cite{Pe1}), toric varieties (\cite{BT1}, \cite{BT2}), horospherical varieties \cite{ST},
equivariant compactifications of unipotent groups
 (see \cite{CT2}, \cite{ST}, \cite{ST1}), and for the wonderful  compactification of
a semisimple adjoint group defined over a number field
(\cite{STT2}, \cite{GMO}). We refer to survey papers
 by Tschinkel (\cite{T1}, \cite{T2}) for a more precise background on this conjecture.

\subsection{Counting rational points of bounded height}
We begin by recalling the notion of a height function
on  the $K$-rational points $\mathbb P^d(K)$
 of the projective $d$-space $\mathbb P^d$.
Denote by
 $R$ the set of all normalized
absolute values $x\mapsto |x|_v$ of $K$, and
by $K_v$ the completion of $K$ with respect to $|\cdot|_v$.

For each $v\in R$, choose a norm $\He_v$ on $K_v^{d+1}$
which is simply the max norm $\He_v(x_0,
\cdots, x_d)=\max_{i=0}^d|x_i|_v$
for almost all $v$.
Then
the height function $\He : \mathbb P^d(K)\to \mathbb R_{>0}$ associated to
$\mathcal O_{\mathbb P^d}(1)$ is given by
 $$\He (x):=\prod_{v\in R}\He_v(x_0, \cdots, x_d)$$
for $x=(x_0:\cdots : x_d)\in \mathbb P^d(K)$.
Since $\He_v(x_0, \cdots, x_d)=1$ for almost all $v\in R$,
we have $0<\He (x)<\infty$ and by the product formula,
$\He$ is well defined, i.e., independent of the choice of
representative for $x$.

%In other words, $\H$ is a height function associated
%to the tautological line bundle $\mathcal O_{\mathbb P^d}(1)$.

For instance, for $K=\q$,
if we choose $\H_p$ to be the maximum norm of
$\qp^{d+1}$ for each prime $p$ and
set $\H_\infty(x_0, \cdots, x_d)=\left(x_0^2+\cdots+x_d^2\right)^{1/2}$
for $x_i\in \br$,
we have
$$\H (x_0: \cdots :x_d)=\left(x_0^2+\cdots+x_d^2\right)^{1/2}
$$
where $x_0,\ldots,x_d\in\mathbb{Z}$ and $\gcd(x_0,\ldots,x_d)=1$.

Let $\bG$ be a linear algebraic group defined over $K$,
with a given $K$-representation
 $\iota: \bG \to \GL_{d+1}$.
Then $\bG$ acts on $\mathbb P^d$ via the canonical map $\GL_{d+1}\to \PGL_{d+1}$.
 Consider $\bU:=u_0\bG\subset \mathbb P^d$
for $u_0\in \mathbb P^d (K)$.
Fixing a height function $\H$ on
$\mathbb P^d(K)$,
we study the asymptotic
of the following number (as $T\to \infty$):
$$
N_T(\bU):=\#\{x\in \bU(K):\, \He(x)<T\}.
$$

Our main results are proved under the following assumption:
\begin{enumerate}
\item
[(i)] $\bG$ is a connected semisimple $K$-group.
\item[(ii)] $\bL=\hbox{Stab}_{\bG}(u_0)$ is a semisimple maximal connected
 $K$-subgroup of $\bG$.
\item[(iii)] For almost all $v\in R$, $\bG(K_v)$ acts
transitively on $\bU(K_v)$.
\end{enumerate}

If $\bL$ is the fixed points of an involution of $\bG$,
$\bU$ is called a symmetric space. A symmetric space $\bU=\bL\ba \bG$ satisfies
(ii) if $\bL$ is connected and semisimple, since $\bL$ is then a maximal
connected $K$-subgroup \cite{Bo}.

Borovoi gave a classification of  symmetric spaces $\bU=\bL\ba \bG$
satisfying (i)-(iii)
with   $\bG$  absolutely almost simple
(see  Appendix \ref{s:app}).

When both $\bG$ and $\bL$ are connected,
the property (iii) is equivalent to the finiteness of the set of $\bG(K)$-orbits
in $\bU(K)$
(Theorem \ref{thm:main-f-m-orbits}).
We remark that (iii) always holds for $\bL$ simply connected,
by Corollary \ref{cor:sc-finite}.
Note that there is a lot of examples of nonsymmetric  homogeneous
spaces $\bU=\bL\ba \bG$ satisfying (i)-(iii), see
%Addendum \ref{sec:6} and especially
\ref{subsec:non-spherical}.% in Appendix \ref{s:app}.

%More generally, $\bU$ is called spherical variety if
%a Borel subgroup of $\bG$ has a dense orbit in $\bU$.

Denote by $\bX\subset \mathbb P^d$  the Zariski closure of $\bU$.
Then $\bX$ is a $\bG$-equivariant compactification of $\bU$, and the
pull back $L$ to $\bX$ of the line bundle $\mathcal O_{\mathbb P^d}
(1)$ is a $\bG$-linearized very ample line bundle of $\bX$ defined
over $K$.
\begin{Thm}\label{maint}
Assume that there is a global section $s$ of $L$ such that
$\bU=\{s\ne 0\}$.
Then
there exist $a\in \q_{>0}$ and $b\in \n$ such that
 \footnote{$A(T)\asymp B(T)$ means
that for some $c>1$, $c^{-1}\, B(T)\le A(T)\le c\, B(T)$ holds for all sufficiently large $T>0$}
 $$
N_T(\bU) \asymp T^{a}(\log T)^{b-1}.
$$
Moreover, if $\bG$ is simply connected,
there exists $c>0$ such that
 $$
N_T(\bU)\sim c\cdot T^{a}(\log T)^{b-1}.
$$\end{Thm}

We note that the assumption in Theorem \ref{maint} is automatic in
many cases, for instance, if $L$ is in the interior of the cone of
effective divisors $\Lambda_{\text{eff}}(X)$; also see Example
\ref{ex-int} below.

The exponents $a$ and $b$ are given as follows:
First, we assume that $\bX$ is smooth and $\bX\backslash \bU$
is a divisor of normal crossings with smooth
 irreducible components $D_\alpha$, $\alpha\in \mathcal{A}$,
defined over a finite field extension of $K$.
Let $\omega$ be a differential form of $\bX$ of top degree, which is nowhere zero on $\bU$, and
choose a global section $\s$ of $L$
with $\bU=\{\s\ne 0\}$.
Then for $m_\alpha\in\mathbb{N}$ and $n_\alpha\in\mathbb{Z}$,
\begin{align*}
\hbox{div}(\s)=\sum_{\alpha\in\mathcal{A}} m_\alpha D_\alpha\quad\hbox{and}\quad
-\hbox{div}(\omega)=\sum_{\alpha\in\mathcal{A}} n_\alpha D_\alpha.
\end{align*}
The Galois group $\Gamma_K=\hbox{Gal}(\bar K/K)$ acts on $\mathcal{A}$.
We denote by $\mathcal{A}/\Gamma_K$ the set of $\G_K$-orbits.
Then
\begin{align}\label{eq:ab}
a=\max_{\alpha\in\mathcal{A}}\left\{\frac{n_\alpha}{m_\alpha}\right\}\quad\hbox{and}\quad
b=\#\left\{\alpha\in\mathcal{A}/\Gamma_K:\, \frac{n_\alpha}{m_\alpha}=a\right\}.
\end{align}
We note that $a$ and $b$ are independent of the choices
of $\s$ and $\omega$, since there are unique choices of
them up to multiplication by constants as a consequence of Rosenlicht theorem.

For a general projective variety $\bX$,
 we take an equivariant resolution of singularities $\pi:\tilde \bX\to \bX$
such that $\tilde \bX$ is smooth and $\pi^{-1}(\bX\backslash \bU)$
is a divisor with normal crossings. Then the constants $a$ and $b$ are defined
as above with respect to the  pull-backs $\pi^*(\s)$ and $\pi^*(\omega)$.

\begin{Ex}[Rational points on affine varieties]\label{ex-int}
{\rm Let $\iota_0:\bG\to \SL_d$ be a $\q$-rational representation, and $\mathbf V=v_0\bG\subset \mathbb A^d$
be Zariski closed for some non-zero $v_0\in \q^d$.
We write an element of ${\mathbf V}(\q)$ as $\left(\frac{x_1}{x_0}, \cdots, \frac{x_d}{x_0}\right)$
where $x_0, \cdots, x_d\in \z$, $x_0>0$
and $\text{g.c.d}(x_0, \cdots, x_d)=1$.
If $\bG$ and $\bL:=\text{stab}_{\bG}(v_0)$ satisfy the assumptions
(i)-(iii), Theorem \ref{maint} implies
$$\#\{\left(\frac{x_1}{x_0}, \cdots, \frac{x_d}{x_0}\right) \in \mathbf V(\q):
\sqrt{x_0^2+\cdots +x_d^2} <T\}
\asymp T^a (\log T)^{b-1} ;$$
$$\#\{\left(\frac{x_1}{x_0}, \cdots, \frac{x_d}{x_0}\right) \in \mathbf V(\q):
\max\{|x_0|, \cdots, |x_d|\} <T\}
\asymp T^a (\log T)^{b-1} .$$

To deduce this from Theorem \ref{maint},
consider the embedding of $\SL_d$ into $\PGL_{d+1}$
by $A\mapsto \text{diag}(A, 1)$, and  of $\A^d$ into $\mathbb P^d$ by
$(x_1, \cdots, x_d) \mapsto (x_1:\cdots: x_d :1)$.
This identifies $\mathbf V$ with the orbit $\bU:=(v_0:1)\bG$ in $\mathbb P^d$, and
 $\s=x_{d+1}$ is an invariant section of
 the line bundle $L$ obtained by pulling back
$\mathcal O_{\mathbb P^d} (1)$, satisfying $\bU=\{\s\ne 0\}$.
Finally, $\H\left(\frac{x_1}{x_0}: \cdots : \frac{x_d}{x_0}:1\right)=
\H(x_1: \cdots: x_d: x_0)$, and hence the claim follows.}
\end{Ex}
%Note that our theorems apply to any connected semisimple almost $K$-simple
%group variety $\bG$,
%since $\bG$ may be identified with the quotient of $\bG\times \bG$
%by the diagonal embedding of $\bG$, which is a maximal $K$-subgroup.
%More generally, it applies to any semisimple symmetric varieties, that
%is, $\bL$ is the group of fixed points of a $K$-involution of $\bG$.
%Borovoi recently classified all semisimple symmetric varieties with
%property (ii)-see the appendix.

Since $\bU=\{X\in \SL_{2n}: X^t=-X\}$ is a homogeneous variety $\Sp_{2n}\ba \SL_{2n}$
for the action $v. g=g^tvg$
and $\SL_{2n}(\qp)$ acts transitively on $\bU(\qp)$ for all $p$,
we have:
\begin{Ex} Let $n\ge 2$. For some $a\in \q^+$, $b\in \n$ and $c>0$, as $T\to \infty$,
$$\#\{X\in \SL_{2n}(\q): X^t=-X, \;\max_{1\le i,j\le 2n} \{|x_{ij}|, |x_0|\}<T\}
\sim c \cdot T^a (\log T)^{b-1} .$$
where $X=\left(\frac{x_{ij}}{x_0}\right)$, $x_{ij}\in \z$, $x_0\in \n$ and $\operatorname{g.c.d}\{x_{ij}, x_0: 1\le  i, j\le 2n \}=1$.
\end{Ex}

Theorem \ref{maint} settles new cases of Manin's conjecture on rational points
of some wonderful varieties,
which we recall.
Let $\bX$ be a Fano $K$-variety, i.e., a
smooth projective $K$-variety with its anticanonical class $-K_{\bX}$ being ample.
Let $\hbox{Pic}(\bX)$ denote the Picard group of $\bX$ and $\Lambda_{\hbox{\tiny eff}}(\bX)\subset
\hbox{Pic}(\bX)\otimes \mathbb{R}$ the cone of effective divisors.
Given a line bundle $L$ on $\bX$,
there exists an associated height function $\H_L$ on $\bX(K)$,
unique up to the multiplication by bounded functions,
via Weil's height function.
For instance if $L$ is very ample
with a $K$-embedding $\psi:\bX\to \mathbb P^d$,
then a height function $\H_L$ is simply the pull-back of a height function
of $\mathbb P^d(K)$ to $\bX(K)$ via $\psi$. Note this depends on the choice
of $\psi$.
For an ample line bundle $L$, $\H_L=\H_{L^k}^{1/k}$ for
$k\in \n$  such that $L^k$ is very ample.

The conjecture of Manin \cite{BM}, generalized by Batyrev and Manin,
predicts that there exist a Zariski open subset $\bU\subset \bX$ and a finite field extension $K'$ of $K$ such that
$$
\#\{x\in \bU(K'):\, \H_L(x)<T\}\sim c\cdot T^{a_L}(\log T)^{b_L-1},
$$
where $c>0$ and
\begin{align*}
a_L&:=\inf\{a:\, a[L]+[K_{\bX}]\in\Lambda_{\hbox{\tiny\rm  eff}}(\bX)\},\\
b_L&:=\hbox{the maximal codimension of the face of $\Lambda_{\hbox{\tiny\rm eff}}(\bX)$ containing $a_L[L]+[K_{\bX}]$}.
\end{align*}

A smooth connected projective $\bG$-variety $\bX$ defined over $K$ is said
 to be {\it wonderful} (of rank $l$), as introduced by Luna \cite{lu},
 if
\begin{enumerate}
\item $\bX$ contains $l$ irreducible $\bG$-invariant
divisors with strict normal crossings.
\item $\bG$ has exactly $2^l$ orbits in $\bX$.
\end{enumerate}

For a $\bG$-homogeneous variety $\bU$, a wonderful variety $\bX$ is called
 the wonderful compactification of $\bU$ if it
is a $\bG$-equivariant compactification of $\bU$.
Luna showed in \cite{lu} that every wonderful variety is spherical; in particular
a wonderful compactification of a homogeneous space $\bU=\bL\ba \bG$ exists
only when $\bL$ is a spherical subgroup, that is, a Borel subgroup of $\bG$ has an
open orbit in $\bU$.
%We remark that $[\operatorname{N}_{\bG}(\bL):\bL]<\infty$ for a spherical subgroup $\bL$.

The following can be deduced from Theorem \ref{maint}:
\begin{Cor} \label{wonderfulcm}
Let $\bU$ be as in Thm. \ref{maint}
and $\bX$ the wonderful compactification
of $\bU$.
Then for any ample line bundle $L$ on $\bX$ over $K$ and
an associated height function $\H_L$,
we have
$$
\#\{x\in \bU(K):\, \H_L(x)<T\}\asymp T^{a_L}(\log T)^{b_L-1}.
$$
Moreover, if $\bG$ is simply connected,
there exists $c=c(\H_L)>0$ such that $$
\#\{x\in \bU (K):\, \H_L(x)<T\}\sim c\cdot T^{a_L}(\log T)^{b_L-1}.
$$
\end{Cor}

Several examples for the above corollary are given in 3.4.
%There are many examples of homogeneous spaces admitting
%wonderful compactifications.
 De Concini and Procesi \cite{DP} constructed the wonderful compactification
of a symmetric variety $\bL\ba \bG$ for $\bG$ semisimple adjoint.
In these cases, $a_L$ and $b_L$ can also be interpreted
in terms of the representation theoretical data of $\bG$ (see \ref{rem:ab}).
%A concrete example for the above corollary holds unconditionally is the wonderful compactification of
%the space of sympletic forms; see \ref{sec_ex}.

Generalizing the work in \cite{DP},
 Brion and Pauer \cite{bp} established that a spherical variety $\bL\ba \bG$ possesses
an equivariant compactification with exactly one closed orbit if and only if
$[\operatorname{N}_{\bG} (\bL):\bL]<\infty$,
where $\operatorname{N}_{\bG}(\bL)$ denotes the normalizer of $\bL$ in $\bG$.
Knop \cite[Coro. 7.2]{kn} showed that
 the wonderful compactification of a spherical variety exists when $\operatorname{N}_{\bG}(\bL)=\bL$.
Complete classification of homogeneous spherical varieties were obtained; see \cite{Kr} and \cite{B4}.

\subsection{On the proofs}
To explain our strategy, let $\A$ denote the Adele ring over $K$. The first key observation is
that the global section $\s$ of $L$ with $\bU=\{\s\ne 0\}$ is in fact $\bG$-invariant. We use Luna's theorem
for this step.
And the extension of $\H$ to $\bU(\A)$ using $\s$ is uniformly continuous and proper for the action of
compact subsets of $\bG(\A)$.
Set $$B_T:=\{x\in \bU(\A): \H(x)<T\}$$ so that
$N_T(\bU):=\# B_T\cap \bU(K)$.
Under the assumption (iii), there are only finitely many $\bG(K)$-orbits
in $\bU(K)$, and hence the counting problem reduces to each $\bG(K)$-orbit.
In general, the naive heuristic $$\# (u_0\bG(K)\cap B_T)\sim \text{vol} (u_0\bG(\A)\cap B_T)$$
is false. The reason behind this is the existence of non-trivial automorphic characters
of $\bG(\A)$. From the dynamical point of view,
this means that the translates  $\bL(K)\ba \bL(\A)g_i$ of periods
do not get equidistributed in the whole space $\bG(K)\ba \bG(\A)$
as $g_i\to \infty$ in $\bL(\A)\ba \bG(\A)$.
This requires us to pass to a suitable finite index subgroup of $\bG(\A)$.
Denote by $\pi:\tilde \bG \to \bG$ a simply connected covering of
$\bG$ defined over $K$. For any compact open subgroup $W$ of the subgroup $\bG(\A_f)$
of finite adeles, we show that the product $G_{W}:=\bG(K)\pi(\tilde \bG(\A)) W$ is a normal
subgroup of finite index in $\bG(\A)$,
and the translates  $\bL(K)\ba (\bL(\A)\cap G_W) g_i$
become equidistributed in the space $\bG(K)\ba G_W$
relative to $W$-invariant functions.
The last statement is a special case of our main ergodic theorems in adelic setting, to be detailed
in the next subsection. We mention that our assumption $\bL$ is semisimple
is crucial.

In order to deduce
$$\# (u_0\bG(K)\cap B_T)\sim \text{vol} (u_0G_W\cap B_T),$$
we prove that for any compact open subgroup $W$
of $\bG(\A_f)$
by which $\H$ is invariant, the family $\{B_T\cap u_0G_{W}\}$
is well-rounded; roughly speaking, for any $\e>0$,
there is a neighborhood $U_\e$ of the identity in $\bG(\A)$ such that
the volume of $(B_T\cap u_0G_{W})U_\e$ is at most $(1+ \e)\text{vol}( B_T\cap u_0G_{W})$ for all large $T$.
Establishing this
is based on the work of Chambert-Loir and Tschinkel \cite{CT2} and of Benoist-Oh \cite{BO2}
(also \cite{GN} of Gorodnik-Nevo).
 Finally, we deduce the
asymptotics of $\text{vol} (u_0G_W\cap B_T)$  up to a bounded
constant from
 \cite{CT2,CT3}. When $\bG$ is simply connected,
we have $G_W=\bga$ and
 deduce the precise volume asymptotic for $u_0\bga\cap B_T$.
We remark that if $\bG(K_v)$ has no compact factors for some archimedean $v\in R$
and $\bG(\A)=\bG(K)\pi(\tilde \bG(\A))W_{\H}$ where
$W_{\H}$ is a compact open
 subgroup of $\bG(\A_f)$ under which
$\H$ is invariant,
we also have
$$\# (u_0\bG(K)\cap B_T)\sim \text{vol} (u_0\bga\cap B_T)\sim c \cdot T^a\log T^{b-1}.$$
In general,
replacing $\asymp$ with $\sim$ in Theorem \ref{maint} requires
regularizing the height integrals
$\int_{u_0\bG(\A)}\H^{-s}(u_0g)\cdot\chi(g)\, d\mu$
for $\bL(\A)$-invariant automorphic characters $\chi$ of $\bga$
as in \cite[Thm. 7.1]{STT2}.
We mention that the strategy of relating the counting problem with the
equidistribution of orbits
 is originated in the work of Duke-Rudnick-Sarnak \cite{DRS} and of Eskin-McMullen \cite{EM}
(see section \ref{s:c} for more details).
Perhaps the most unsatisfying assumption is (iii): the finiteness of $\bga$-orbits
in $\bU(\A)$. We believe this assumption should not be necessary
to deduce $\# (u_0\bG(K)\cap B_T)\sim \text{vol} (u_0G_W\cap B_T);$
however our proof of well-roundedness of $u_0G_W\cap B_T$ relies on the
finiteness assumption. With a proper use of motivic integration, it may be possible to deal with
a general case. Finally we mention that there are examples
where the orders of magnitude for  $\# (u_0\bG(K)\cap B_T)$
and $\#N_T(\bU)$ are not the same. This is already the case for $\bG=\hbox{SL}_2$ and
$\bU=\bG/Z(\bG)$.

%We relate the above counting problems to
%the equidistribution problem of certain adelic periods in
%the homogeneous space $\bG(K)\ba \bga$.

\subsection{Equidistribution of Adelic periods}
We now describe our main ergodic results on the equidistribution
of Adelic periods. Our results presented
in this section are much more general
than what is needed for the application on rational points.

Let $\bG$ be a connected semisimple group defined over a number field $K$.
%We denote by $\bG(\A)$ the adele group of $\bG$ over $K$.
%As is well known, $\bG(\A)$ is a
%locally compact topological group
%and the group $\bG(K)$ of $K$-rational points
%of $\bG$ is a lattice in $\bG(\A)$, after diagonally embedded.

Set $X:=\bG(K)\ba \bga$ and $x_0:=[\bG(K)]\in X$.
For a connected semisimple $K$-subgroup $\bL$ of $\bG$, we denote by $\pi:\tilde \bL \to
\bL$ a simply connected covering over $K$, which is unique up to $K$-isomorphism.
Then $\pi$ induces the map $\tilde \bL(\A)\to \bL(\A)$
and hence $\tilde \bL(\A)$ acts on $X$ via $\pi$, and
the orbit $x_0.\tilde \bL(\A)$ is closed and carries a unique
$\tilde \bL(\A)$-invariant probability measure supported in the orbit.
%Understanding the distribution of
% of the orbit $x_0 \tilde \bL(\A)$ in $X$ arises in many arithmetic
%problems.

Let $\{\bL_i\}$ be a sequence of connected semisimple $K$-subgroups
of $\bG$ and $\{g_i\in \bga\}$ be given.
Let $\mu_i$ denote the (unique) $\tilde \bL_i (\A)$-invariant
probability measure in $X$ supported on the orbit
$Y_i:=x_0.\tilde \bL_i (\A)$.
 The translate $g_i\mu_i$ of $\mu_i$ by $g_i$
is defined by $$g_i\mu_i(E):=\mu_i(E g_i^{-1})$$
for any Borel subset $E\subset X$.

Denoting by $\mathcal P(X)$ the space
of all Borel probability measures on $X$,
 that a sequence $\nu_i\in \mathcal P(X)$ weakly converges
to $\mu\in \mathcal P(X)$
means that for every $f\in C_c(X)$,
$$\lim_{i\to \infty} \int_X f(x)\, d\nu_i(x) = \int_X f\, d\mu .$$

We study the following question:
\begin{equation*}
\text{ Describe the weak-limits of $g_i\mu_i$ in $\mathcal P(X)$} .
\end{equation*}

We remark that the reason of considering
the translates of $x_0\tilde \bL_i(\A)$ rather than those
of the orbit $x_0\bL_i(\A)$,
 is essentially because
$X$ has more than one {\it connected} components and
the adele group of the simply connected cover plays exactly the role
of the identity component in a suitable sense.

\begin{Def} A valuation $v\in R$ is said to be {\it strongly isotropic} for $\bG$
if for every connected non-trivial normal $K_v$-subgroup $\bN$
of $\bG$, $\bN (K_v)$ is non-compact.
 We denote by $\mathcal {I}_{\bG}$ the set of all strongly isotropic
$v\in R$ for $\bG$ .\end{Def}

 For a compact open subgroup $W_f$ of the group
$\bG(\A_f)$ of finite adeles, we denote by
$C_c(X, W_f)$ the set of all continuous $W_f$-invariant functions
on $X$ whose support is compact and contained in the set
$x_0\pi(\tilde \bG(\A))W_f$.

%The following is our main theorem describing all possible
% weak-limits of $g_i\mu_i$:
\begin{Thm} \label{dmadele}
Suppose that $\cap_i \mathcal I_{\bL_i}\ne \emptyset$, and let $g_i\in \bG(K)\pi(\tilde \bG(\A))$. \be
\item  If the centralizer of $\bL_i$ is $K$-anisotropic for each $i$, then
 the sequence $\{g_i\mu_i\}$ does not escape
to infinity, that is, for any $\e>0$,
there exists a compact subset $\Omega\subset X$
such that $$g_i\mu_i(\Omega)>1-\e\quad\text{
for all large $i$.}$$

\item Let $\mu\in \mathcal P(X)$ be a weak limit
of $g_i\mu_i$.
Then there exists a connected $K$-subgroup $\bM$ of $\bG$ such that
\begin{itemize}
\item for some $\delta_i\in \bG(K)$,
$$\delta_i \bL_i\delta_i^{-1}\subset \bM \quad\text{for all sufficiently large $i$}; $$
\item for any compact open subgroup $W_f$ of
$\bG(\A_f)$,
%\begin{itemize}\item[(i)]
%$S$ is $\bG$-isotropic and $R_\infty (\bG)\subset S$,
% \item[(ii)] $S \cap (\cap_i \mathcal I_{\bL_i})\ne \emptyset$
% \item [(iii)] $S$ is $\bM$-strongly isotropic,
 %then %for any  compact open subgroup $W_S$ of $\bG(\A_S)$,
 there exist
 a finite index normal
subgroup $M_0$ of $\bM(\A)$ containing
 $\bM(K)\pi(\tilde \bM(\A))$ and $g\in \pi(\tilde \bG(\A))$
  such that
 $$\mu(f)=g\mu_{M_0} (f)
   \quad\text{for all $f\in C_c(X, W_f)$} $$
where $g\mu_{M_0}$ is the invariant probability measure
supported on $x_0M_0g$,
and  there exists $h_i\in \pi(\tilde \bL_i(\A))$ such that
$\delta_i h_i g_i$ converges to $g$ as $i\to \infty$.
\item If the centralizers of $\bL_i$ are $K$-anisotropic, $\bM$ is semisimple.
\end{itemize}\ee
\end{Thm}

See Corollaries \ref{cm} and \ref{cmgeneral} where we discuss special cases of the above theorem
for $\bL_i$ maximal semisimple.

We mention that  Theorem \ref{dmadele} solves a stronger
version of the conjecture of Clozel and Ullmo in a greater generality (see
\cite[p.~1258]{CL}, also \cite[Con. 6.4]{BRe}). We also refer to \cite{O3} for discussions on
Theorem \ref{dmadele} regarding Hecke operators.

The analogous theorems in the case of a homogeneous space
of a connected semisimple {\it real Lie} group
have been studied previously in \cite{DM1}, \cite{DRS},
\cite{EM}, \cite{EMS1}, \cite{MS}, \cite{EO} and
\cite{EMV}, etc.
Via the strong approximation properties of simply connected semisimple
groups, our proof of Theorem \ref{dmadele} is reduced to the generalizations of the aforementioned results,
especially of Dani-Margulis \cite{DM1} and Mozes-Shah \cite{MS},
in the $S$-algebraic setting (see Theorem \ref{rc}). We make a crucial use
 the classification theorem on ergodic
measures invariant under unipotent flows in this set-up obtained by Ratner \cite{R},
 Margulis-Tomanov \cite{MT1}, and also refined by Tomanov \cite{T} in the arithmetic situation.
Our approach is based on the linearization methods developed by Dani-Margulis \cite{DM2}.

In the case of $\bG=\PGL_2$,
 and $\bL_i$ a $K$-anisotropic torus,
the analogue of the above theorem can be deduced from a theorem
of Venkatesh (Theorem 6.1 in \cite{Ve}) using Waldspurger's formula (cf. \cite[2.5]{MV})
which relates the integral over a period with special values of $L$-functions.
For $\bG=\PGL_3$ and $\bL_i$
a $\q$-anisotropic maximal
torus, it was obtained by Einsiedler, Lindenstrauss,
Michel, and Venkatesh \cite{ELMV}.
$\bL_i$'s being tori, the methods in \cite{Ve} and \cite{ELMV} are very different
from ours. The powerful theorems on unipotent flows (\cite{R} and \cite{MT1})
are essentially what makes our theorem \ref{dmadele} so general.

%An effective Adelic mixing was obtained in [GMO]
%using the results of Oh [Oh] for $K$-rank at least $2$.
%For the low $K$-rank case, this required many deep theorems
%in automorphic theory, as explained in [Cl] and [COU].
\subsection{Other applications}
Theorem \ref{dmadele} should be useful in many future
arithmetic applications. For instance,
an application of Theorem \ref{dmadele} in a problem of Linnik, considered
in \cite{EO} and \cite{EV}, is discussed in \cite{O3}.
 We state only two below, which are most relevant to the
subject of this paper.
One application of Theorem \ref{dmadele} is
an ergodic theoretic proof of the adelic mixing theorem obtained in \cite{GMO}
though given only in a non-effective form.
\begin{Thm}[Adelic mixing]\label{am} Let $\bG$ be simply connected
and almost $K$-simple.
For any $f_1, f_2\in L^2(\bG(K)\ba \bG(\A))$ and
any sequence $g_i\in \bga$,
$$\int_{\bG(K)\ba \bga} f_1(xg_i) f_2(x) d\mu_G\to \int f_1 d\mu_G \cdot \int f_2 d\mu_G
\quad \text {as $g_i\to \infty$,}$$
where $\mu_G$ is the invariant probability measure on $\bG(K)\ba \bG(\A)$.
\end{Thm}

The adelic mixing theorem in particular implies the equdistribution of Hecke points
studied in \cite{COU} and \cite{EO2} (see \cite{O3} for details).
The proof in \cite{GMO} is based on the information on local harmonic analysis
of the groups $\bG(K_v)$ \cite{O1} as well as the automorphic theory of $\bG$ \cite{Cl},
and gives a rate of convergence.
In the methods of this paper, it suffices
to know the mixing property of $\bG_S:=\prod_{v\in S}\bG(K_v)$ for some finite $S$
containing all archimedean valuations and containing at least one strongly isotropic $v$.
 This property can either
be deduced from the classical Howe-Moore theorem
\cite{HM}, or from the property of unipotent flows in $\bG_S$ modulo lattices.

%We say an invaraint measure $\mu$ on the quotient space
%$H\ba G$ for (unimodular) locally compact groups $H<G$
%is compatible with Haar measures $\mu_H$ and $\mu_G$
%if for any $f\in C_c(G)$,
%$$\int f d\mu_G =\int_{H\ba G} \int_H f (h g)d\mu_H (h)d\mu(Hg) .$$

In the following corollary, let $\bU$ be an affine variety defined over $\z$ such that
$\bU=v_0\bG$ where
$\bG\subset \GL_N$ is a connected simply connected semisimple $\q$-group and $v_0\in \q^N\setminus \{0\}$.
Suppose that
$\bL:=\text{stab}_\bG (v_0)$ is
 a semisimple maximal connected
$\q$-subgroup of $\bG$. We let $\mu_p$, $v\in R$
be invariant
measures on $v_0\bG(\qp)$ such that $\mu=\prod \mu_p$ is a measure on $v_0\bga$
 compatible with the probability invariant measures
$\mu_{G}$ and $\mu_{L}$ on $\bG(\q)\ba \bga$
and $\bL(\q)\ba \bL(\A)$ respectively.

As another corollary,
we obtain the following local-global principle, which can be seen
as a higher dimensional analogue of the classical Hasse principle:

\begin{Cor}\label{linnik}
\be
\item For all sufficiently large $m\in \n$,
 $$\bU\left( m^{-1}\z\right) \ne \emptyset \quad\text{iff}\quad
\bU\left( m^{-1}\z_p\right)\ne \emptyset \quad\text{for all
primes $p$.}$$
\item If $\bL$ is simply connected, then
for any compact subset $\Omega\subset v_0\bG(\br)$
of boundary of measure zero,
$$\# \bU\left( m^{-1}\z\right)\cap \Omega
\sim  \mu_\infty (\Omega) \prod_{p} \mu_p\left(\bU\left(m^{-1}\zp\right)\right)$$
provided the right hand side is not zero as $m\to \infty$.
\ee
\end{Cor}

We remark that the assumption that both $\bG$ and $\bL$
are simply connected imply that the group $\bG(\A_f)$ of finite adeles acts
transitively on $\bU(\A_f)$ [BR], and hence $\mu_p$'s are invariant
measures
on $v_0\bG(\qp)=\bU(\qp)$ for each finite $p$.

When $\bU=\bG$, i.e., a group variety, Corollary \ref{linnik} was
observed in
\cite{Gu}, as an application of the Adelic mixing theorem.
(2) of the above corollary was previously obtained
in \cite{EO} and \cite{O2} assuming that both $\bG(\br)$ and $\bL(\br)$
has no compact factors and that $\bU(m^{-1}\z)\ne \emptyset$.
See also \cite{BO1} and \cite{ELMV} for the case when $\bL$ is a torus.

% Letting $\nu$ be the invariant measure supported on
%$\bL(\z)\ba \bL(\br)$ inside $\bG(\z)\ba \bG(\br)$, studying the behaviour
%of the translates $g_i\nu$ for $g_i\in \bG(\br)$
%is one of the main steps in understanding the equidistribution and counting
%problems of integral points on the homogeneous space $\bL\ba \bG$
%(see \cite{DRS}, \cite{EM}, \cite{EMS}).

%In the same way, studing the translates of the adelic orbits $\bL(\q)\ba \bL(\A) g_i$
%inside $\bG(\q)\ba \bG(\A)$ is one of the major steps in understanding
%the asymptotic distribution properties of rational points on $\bL\ba \bG$.

\vs
\noindent{\bf Organization:}
In section \ref{s:volume}, we discuss how to extend
a height function of $\bU(K)$ to $\bU(\A)$ so that the action of $\bga$ is
uniformly continuous and proper,
and obtain the asymptotic of the volume of the height balls
in each $M$-orbit of $\bU(\A)$ for a finite index subgroup $M$ of
$\bga$. The second part uses the work of Chambert-Loir and Tschinkel.
In section \ref{s:won} we discuss the wonderful varieties, introduced by Luna,
which are the generalization of the wonderful compactification
of symmetric varieties constructed by De Concini- Procesi. They provide
main examples of our theorem \ref{wonderfulcm}.
In section \ref{s:ea}, we deduce Theorem \ref{dmadele} from
the corresponding theorem \ref{ssm} in the $S$-arithmetic setting, which
is proved in the last 2 sections of this paper.
In section \ref{s:c}, we prove main theorems of the introduction.
In section \ref{sec:ms}, we prove one part of Theorem \ref{ssm},
and the other part is proved in section \ref{sec:dm}.

%Using the strong approximation properties for $\tilde \bG$,
%Theorem \ref{dmadele} is deduced from the analogous statements
%in the $S$-arithmetic cases, which are then $S$-arithmetic generalizations
%of a theorem of Dani and Margulis \cite{DM1}, and of a theorem of Mozes and Shah %\cite{MS}.
%Our
% methods are based on the linearization techinique deveoped in \cite{DM2} and
%Ratner's measure classification \cite{R}, as in the work of \cite{MS}.
%We have benefited from the work of Tomanov \cite{T} as well.

\vs
\noindent{\bf Acknowledgment} We thank Akshay Venkatesh
for generously sharing his insights.  We thank Mikhail Borovoi who kindly wrote up
the appendix on our request. Oh also wants to thank IAS
where some part of this work was done during her stay in Feb-Mar, 2006.
Gorodnik would like to thank
Princeton University
for hospitality.
\section{Heights and Volume estimates} \label{s:volume}
 Let $K$ be a number field and $R$ the set of all normalized absolute values of $K$.
By $R_\infty$, we mean the subset of $R$ consisting of all
archimedean ones and set $R_f:=R\setminus R_\infty$.
For each $v\in R$,
we denote by $K_v$ the completion of $K$ with respect to the absolute value $|\cdot |_v$,
by $k_v$ the residue field, and by $\mathcal{O}_v$ the ring of integers of $K_v$.
The cardinality of $k_v$ is denoted by $q_v$.
For a finite subset $S$ of $R$, the ring of $S$-integers
is the subring of $K$ defined by $\mathcal O_S:=\{x\in K: |x|_v\le 1\;\;
\text{for all non-archimedean
$v\notin S$}\}$.

Throughout section \ref{s:volume}, we
let $\bG$ be a connected semisimple algebraic $K$-group
with a given $K$-representation $\bG\into \GL_{d+1}$.
Fix $u_0 \in \mathbb P^d(K)$ such that the orbit $\bU:=u_0\bG$ is
a $K$-subvariety.
We fix integral models $\mathcal U$ and
$\mathcal G$ of $\bU$ and $\bG$, respectively,
 over the ring $\mathcal O_S$ for some $S$.

Then the adelic space $\bU(\A)$ is the restricted topological product
of $\bU(K_v)$'s with respect to $\mathcal U(\mathcal O_v)$'s.
As well-known, this is a locally compact space.

For finite $S\subset R$,
we set $\bU_S:=\prod_{v\in S}\bU(K_v)$
and denote by $\bU(\A_S)$ the
 the restricted topological product
of $\bU(K_v)$'s, $v\in R\setminus S$,
 with respect to $\mathcal U(\mathcal O_v)$'s.
Then $\bU(\A)$ is canonically identified with the
direct product $\bU_S\times \bU(\A_S)$.
We set $\bU_{\A_f}:=\bU_{\A_{R_\infty}}$ and $\bU_\infty:=\prod_{v\in R_\infty}\bU(K_v)$.
The notations $\bG(\A)$, $\bG_S$ and $\bG_\infty$ etc. are similarly defined.
Note that both $\bG_S$
and $\bG(\A_S)$ can be considered as subgroups of $\bG(\A)$ in a canonical way.

%We assume that
%$\bL:=\text{stab}_{\bG}(u_0)$
%is a connected semisimple subgroup.

Let $\bX\subset \mathbb P^d$
be the Zariski closure of $\bU$, which is then a
$\bG$-equivariant compactification
of $\bU$.
Consider the line bundle $L$ of $\bX$ given by the pull-back
of $\mathcal O_{\mathbb P^d}(1)$.
Then $L$ is very ample and $\bG$-linearized; in fact
 any $\bG$-linearized very ample line bundle
is of this form for some embedding.

%The line bundle $L$ is called {\it $\bG$-linearized}
%if it is equipped with $K$-rational action of $\bG$ such that the projection
%$\pi:L\to\bX$ is $\bG$-equivariant and the action of $\bG$ is linear on fibers.

Since $\bU(K)\ne \emptyset$,  Rosenlicht's theorem implies (cf. \cite[Lem. 1.5.1]{BR}):
\begin{Lem}\label{ros}
There is no non-constant invertible regular function on $\bU$.
\end{Lem}

Using a theorem of Luna \cite{lu}
and the above lemma, we obtain the following:
\begin{Thm}\label{ad-extension}
Suppose that $\bL:=\operatorname{stab}_\bG(u_0)$ is semisimple and
 $[\operatorname{N}_\bG(\bL):\bL]<\infty$.
 Then any global section $\s$ of $L$ such that $\bU=\{\s\ne 0\}$ is
$\bG$-invariant, and unique up to a scalar multiple.
\end{Thm}
\begin{proof}
Pick a point $y\in K^{d+1}\setminus \{0\}$ lying above $u_0$.
Let $\bH$ denote the stabilizer of $y$ in $\bG$.
Since $\bH$ is a normal co-abelian subgroup of $\bL$
and $\bL$ is semisimple, $\bH$ is also semisimple and
$\bH^\circ$ is a finite index in $\bL$.
Hence the finiteness of $[\operatorname{N}_{\bG}(\bL): \bL]$ implies
that $\bH$ has finite index in its normalizer.
Now a theorem of Luna \cite[Corollary 3]{lu}
says that the orbit of $y$ is closed.
By \cite[Ch 2, \S 1, Prop 2.2]{Mu}, there exists a global $\bG$-invariant section
$\s_1$ of $L^k$ for some $k$ such that
$\s_1(u_0)\ne 0$. Hence $\bU\subset \{\s_1\ne 0\}$.

Since $\bU=\{\s^k\ne 0\}$,
 the ration $\s_1/\s^k$ is an invertible regular function
on $\bU$, which is a constant by Lemma \ref{ros}.
Hence
$\s^k$ is $\bG$-invariant.
For any $g\in \bG$,
${\s^g}/{\s}$ is a constant function, say, $\alpha_g$,
 on $\bU$ by the above lemma.
Now $\alpha: g\mapsto \alpha_g$
defines a homomorphism from $\bG$ into the group of $k$-roots of unity.
Since $\bG$ is connected, $\alpha$ must be $1$.
Hence $\s$ is invariant. The uniqueness follows by a similar argument.
\end{proof}

%We set $\bL=\text{stab}_\bG(v_0)$, and assume $\bL$ is semisimple.

\subsection{Heights}\label{sec:height}
Let $\s_0, \cdots, \s_d$ be the global sections of $L$ obtained by pulling back
the coordinate functions $x_i$'s.
We assume that there is a $\bG$-invariant global section $\s$ of $L$ such that $\bU=\{\s\ne 0\}$.

\begin{Def}\label{def:am} An adelic metrization on the $\bG$-linearized line bundle
 $L$ on $\bX$ (with respect to $\s$)
 is a collection of $v$-adic metrics on $L$
for all $v\in R$
such that
\be
%\item for each $v\in R$, there exists $c>1$ such that for all $x\in \bU(K_v)$,
%$$ c^{-1} \cdot\max_{0\le i\le d}\left(\left| \frac{s_i(x)}{s(x)}\right|_v\right)\le  %\|s(x)\|_v^{-1} \le c\cdot
%\max_{0\le i\le d}\left(\left| \frac{s_i(x)}{s(x)}\right|_v\right)
%.$$
\item for each $v\in R_f$, $\|\cdot \|_v$ is locally constant in $\bU(K_v)$ in the $v$-adic topology.
\item for almost all $v\in R$,
$$ \|\s(x)\|_v= \left( \max_{0\le i\le d}\left| \frac{\s_i(x)}{\s(x)}\right|_v\right)^{-1}
\quad \text{for all $x\in \bU(K_v)$.}$$
\item  for each $v\in R_\infty$ and
for any $\e>0$, there exists a neighborhood $W_\e$ of $e$ in $\bG (K_v)$
such that for all $x\in \bU(K_v)$ and $g\in W_\e$,
$$(1-\e)\|\s(x) \|_v \le  \|\s(x g)\|_v\le (1+\e)\|\s(x)\|_v. $$
\ee
\end{Def}

Recall that a $v$-adic metric $\|\cdot \|_v$ on $L$ is a family $(\|\cdot\|_{x, v})_{x\in \bX(K_v)}$
of
$v$-adic Banach norms
on the fibers $L_x$ such that for every Zariski open $U\subset \bX$ and every section $\s\in H^0(U, L)$,
the map $U(K_v)\to \br$ given by $x\to \|\s\|_{x,v}$ is continuous in the
$v$-adic topology on $U(K_v)$.

We write $(\|\cdot\|_v)_{v\in R}$ for an adelic metric on $L$ and
call a pair $\mathcal L=(L, \|\cdot\|_v)$ an adelically metrized line bundle.
Note that an adelic metrization of $L$
extends naturally to tensor products $L^k$ for any $k\in \n$.

An adelically metrized line bundle $\mathcal L$ induces a family of local heights on $\bU(K_v)$:
$$\H_{\mathcal L,v}(x):=\|\s(x)\|_v^{-1}.$$

The following lemma can be proved in a standard way
 (see \cite[Ch.~2]{BG} for a detailed discussion of heights).
\begin{Lem}\label{height:p}
\begin{enumerate}
\item For each $v\in R$, $\inf_{x\in\bU(K_v)} \H_{{\mathcal L},v}(x)>0$.
\item For almost all $v$, $\inf_{x\in\bU(K_v)} \H_{{\mathcal L}, v}(x)=1$.
\item For almost all $v$, $\{x\in {\bU}(K_v):\, \H_{{\mathcal L},v}(x)=1\}=\mathcal{U}(\mathcal{O}_v)$.
\item Let $v\in R$.
If $x\to \infty$ in $\bU(K_v)$  (i.e., $x$ escapes every compact
subset), then $\H_{{\mathcal L}, v}(x)\to \infty$.
\end{enumerate}
\end{Lem}

\begin{Def} An adelic height function $\H_{\mathcal L} :\bU(\A)\to \br_{>0}$ associated to
$\mathcal L$ is defined by
\begin{equation}\label{def:h}
\H_{\mathcal L}(x):=\prod_{v\in R} \H_{\mathcal L, v}(x)\quad\text{for } x\in \bU(\A) .
\end{equation}
\end{Def}
The previous lemma implies that $\H_{\mathcal L}$ is a well-defined continuous proper function.
Moreover the following holds:
\begin{Lem}\label{l:proper}\label{p:well}\be
\item Set $$W_{\H_{\mathcal L}}:=\{g\in \bG(\A_f): \H_{\mathcal L}(xg)=\H_{\mathcal L}(x)\;\;
\text{for all $x\in \bU(\A)$}\}.$$
Then $W_{\H_{\mathcal L}}$ is an open subgroup of $\bG(\A_f)$.
\item For every compact subset $B\subset \bG(\A)$, there exists $c>0$ such that for
every $g\in B$ and $x\in \bU(\A)$,
$$
\H_{\mathcal L}(xg)< c\, \H_{\mathcal L}(x).
$$

\item
For every $\epsilon>0$, there exists a neighborhood $W$
of $e$ in $\bG(\A)$ such that for every $x\in\bU(\A)$ and $g\in W$,
$$
\H_{\mathcal L}(xg)< (1+\epsilon)\H_{\mathcal L}(x).$$
\ee
\end{Lem}
\begin{proof} Since $\|\cdot\|_v$ is locally constant for all $v\in R_f$,
$W_{\H_{\mathcal L}}\cap \bG(K_v)$ is an open subgroup of $\bG(K_v)$ for each $v\in R_f$.
Since $\bG$ acts on $\bU$ via the linear action of $\SL_{d+1}$ on
$\mathbb P^d$ and $s$ is invariant,
 $W_{\H_{\mathcal L}}\cap \bG(K_v)=\mathcal G ({\mathcal O}_v)$ for almost all $v\in R_f$ by
(3) of Def. \ref{def:am}.
It follows that $W_{\H_{\mathcal L}}$ is open. 
%By the properness of $\H_{\mathcal L}$,
% $W_{\H_{\mathcal L}}$ is compact.
Any compact subset $B$ of $\bga$ is contained in
$\prod_{v\in S}B_v\times \prod_{v\notin S}\mathcal G(\mathcal O_v)$ for some finite $S\subset R$
where $B_v$ is a compact subset in $\bG(K_v)$. By enlarging $S$,
we may assume $\prod_{v\notin S}\mathcal G(\mathcal O_v) \subset W_{\H_{\mathcal L}}$.
On the other hand, for each $v\in R$, there exists $c_v>1$
such that $$\H_{\mathcal L, v} (xg)\le
c_v \cdot \max_{i,j} |g_{ij}|_v \cdot \H_{\mathcal L, v}(x)$$
for all $g=(g_{ij})\in \bG(K_v)$ and $x\in \bU(K_v)$.
Hence it suffices to take $c=\prod_{v\in S} ( c_v \cdot  \max_{g\in B_v} |g_{ij}|_v)$
for the claim (2).

The claim  (3) follows from the claim (1) and (3) of Def. \ref{def:am}.
\end{proof}

%When $L$ is ample, we define
%\begin{equation}\label{eq:hv}
%\H_{L}=(\H_{L^{k}})^{1/k}
%\end{equation}
%where $k\in\mathbb{N}$ is such that $L^{k}$ is very ample.

%We fix an integral model  $\mathcal{U}$  for $\bU$ over the ring of $S$-integers.

%Note that  the restriction of the global height $H_L$ to $\bU(K)$ is independent of the %choice
%the section $s$.

We will call the height function $\H_{\mathcal L}$ {\it regular} if the function
 $\prod_{v\in R_\infty}\H_{\mathcal L,v}^2$ is regular on
$\bU_\infty$, considered as the real algebraic variety via
the restriction of scalars.
For instance, the following height function is given by a regular
 adelic metrization : \begin{equation}\label{eq:local}
\H_{\mathcal L, v}(x)=
\left\{
\begin{tabular}{ll}
$\frac{(\sum_i |\s_i(x)|^2_v)^{1/2}}{|\s(x)|_v}$ & \hbox{for archimedean $v$},\\
$\frac{\max_i |\s_i(x)|_v}{|\s(x)|_v}$\quad & \hbox{for non-archimedean $v$}.
\end{tabular}\right.
\end{equation}
This property will be used to deduce
that the volume is H\"older.

The following example shows that our settings apply
to any affine homogeneous varieties:
\begin{Ex}\label{classical}{\rm Denote by $\A^{d}$ the $d$-dimensional affine space.
Let $\bU=v_0\bG\subset \mathbb A^d$ be an affine
 homogeneous $K$-
variety for
a connected $K$-group $\bG\subset \GL_d$ and a non-zero $v_0\in \mathbb A^d(K)$.
Via the embedding $\iota:\mathbb A^d \hookrightarrow \mathbb P^d$ given by
$$\iota(x_0, \cdots, x_{d-1})\mapsto (x_0: \cdots :x_{d-1} :1)$$
and the embedding $\GL_d\to \PGL_{d+1}$ by $A\mapsto \text{diag}(A,1)$,
the Zariski closure
 $\bX\subset \mathbb P^d$ of $\iota(\bU)$ is
a $\bG$-equivariant compactification.

Consider the line bundle $L=\iota^*(\mathcal O_{\mathbb P^d}(1))$
and sections $\s_i=\iota^* (x_i)$ for $0\le i\le d$.
 Since $\iota(\bU)=\{ \s_d\ne 0\}$ for
the $\bG$-invariant section $\s_d$, we can choose an adelic metrization $\mathcal L$ of $L$ so that
the local height functions $\H_{\mathcal L, v}$ on $\bU(K_v)$ are given by
\begin{equation}\label{eq:sh1}
\left\{
\begin{tabular}{ll}
$\left(|x_0|^2_v+\cdots+|x_{d-1}|_v^2+1\right)^{1/2}$ & \hbox{for archimedean $v$},\\
$\max \left\{|x_0|_v,\ldots, |x_{d-1}|_v, 1 \right\}$ & \hbox{for non-archimedean $v$}.
\end{tabular}\right.
\end{equation}
}
\end{Ex}

\subsection{Tamagawa volumes of height balls}
%In this section, we assume that
%$L$ has a global section $\bU=\{s\ne 0\}$.
We assume that $\bL:=\text{stab}_\bG(u_0)$ is semisimple,
and $\s$ is an invariant global section of $L$ such that $\bU=\{\s\ne 0\}$.
Fix an adelic metrization $\mathcal L$
of $L$ and consider the height function $\H=\H_{\mathcal L}$ on $\bU(\A)$ defined
in \eqref{def:h}. For simplicity, we set $\H_v=\H_{\mathcal L,v}$.
We observe that $\bU$ is a geometrically irreducible nonsingular algebraic variety
and that $\bU$ supports a nowhere zero
differential form $\omega$ of top degree. We refer to \cite{We} for the following discussion on
the Tamagawa measure on $\bU(\A)$.
To the form $\omega$, we can associate measures $\mu_v$
on each $\bU(K_v)$. Then $\mu_v(\mathcal U(\mathcal O_v))=\frac{ \# \bU(k_v)}{q_v^{\text{dim} \bU}}$ for almost all
$v\in R$.
Since $\bU$ is a homogeneous space of a connected semisimple algebraic
group with the stabilizer subgroup being semisimple,
$\prod_v \mu_v (\mathcal U(\mathcal O_v))$ converges
absolutely,
and 
 $\omega$ defines the Tamagawa measure $$\mu= |\Delta_K|^{-\frac{1}{2} \text{dim} \bU} \prod_v \mu_v$$
on the space $\bU(\A)$ where $\Delta_K$ is the discriminant of $K$.

% This assumption is not essential
%for Theorem \ref{th:ct}, but simplifies
%the exposition.

For $t>0$,
set $$B_t:=\{y\in \bU(\A):\H(y)<t\} .$$

In this section, we review a theorem of Chambert-Loir and Tschinkel
on asymptotic properties (as $t\to \infty$) of the Tamagawa volume $$
V(t):=\mu(B_t).
$$
%We show that $V$ is H\"older and compute the asymptotics of $V(t)$ as $t\to\infty$.

First, we assume that $\bX$ is smooth and $\bX\backslash \bU$
is a divisor with normal crossings of irreducible
components $D_\alpha$, $\alpha\in \mathcal{A}$,
defined over finite field extensions $K_\alpha$ of $K$.
By extending $\omega$ to $\bX$, which we denote by $\omega$ by abuse of notation,
 we obtain a non-zero rational differential form on $\bX$ of top degree.
Since $\{\s= 0\}=\bX\backslash \bU$ and $\omega$ is nowhere zero on $\bU$,
\begin{align*}
\hbox{div}(\s)=\sum_{\alpha\in\mathcal{A}} m_\alpha D_\alpha\quad\hbox{and}\quad
-\hbox{div}(\omega)=\sum_{\alpha\in\mathcal{A}} n_\alpha D_\alpha
\end{align*}
for $m_\alpha\in\mathbb{N}$ and $n_\alpha\in\mathbb{Z}$.
The Galois group $\Gamma_K=\hbox{Gal}(\bar K/K)$ acts on $\mathcal{A}$.
We denote by $\mathcal{A}/\Gamma_K$ the set of $\G_K$-orbits.
Define
\begin{align}\label{defab2}
a(L)=\max_{\alpha\in\mathcal{A}}\left\{\frac{n_\alpha}{m_\alpha}\right\}\quad\hbox{and}\quad
b(L)=\#\left\{\alpha\in\mathcal{A}/\Gamma_K:\, \frac{n_\alpha}{m_\alpha}=a(L)\right\}.
\end{align}

\begin{Lem}
\be \item $D_\alpha$'s are not rationally equivalent;
\item $a(L)$ and $b(L)$ are independent of choices of $\s$ and $\omega$
\ee
\end{Lem}
\begin{proof}
 If $D_\alpha=\hbox{div}(f)+D_\beta$ for some $f\in K(X)^*$,
then the poles as well as the zeros of $f$ must lie outside $\bU$, and hence
$f$ is constant by Lemma \ref{ros}, proving (1).
Since $\s$ is unique up to constant, again by Lemma \ref{ros},
the independence of $m_\alpha$'s on $\s$ is clear.
Similarly, any non-zero differential form on $\bX$ of top degree,
which is nowhere zero on $\bU$, is a multiple of $\omega$ by a constant. Hence
 $n_\alpha$'s are determined independently on the choice of $\omega$.
\end{proof}

In general,
 we take an equivariant resolution of singularities $\pi:\tilde \bX\to \bX$
such that $\tilde \bX$ is smooth and the boundary
$\pi^{-1}(\bX\backslash \bU)$
is a divisor with  normal crossings. Then the constants $a(L)$ and $b(L)$ are defined
as above with respect to the  pull-backs $\pi^*(\s)$ and $\pi^*(\omega)$.
We refer to \cite{BMi} and \cite{V} for constructions of equivariant resolutions
of singularities.

We consider the Mellin transform of $V(t)$:
\begin{align*}
\eta(s)&:= \int_0^\infty t^{-s}\, dV(t)\\
&=\int_{\bU(\A)} \H(x)^{-s}\, d\mu (x).
\end{align*}
Hence
$$
\eta(s)=|\Delta_K|^{-\frac{1}{2} \text{dim} X}\prod_v \eta_v(s)
$$
where
$$
\eta_v(s):=\int_{\bU(K_v)} \H_{v}(x)^{-s}\, d\mu_v(x).
$$
Let
$$
\Omega_t=\{s\in\mathbb{C}:\,\, \hbox{\rm Re}(m_\alpha s-n_\alpha)>t, \alpha\in\mathcal{A}\}.
$$

\begin{Thm}[Chambert-Loir, Tschinkel]\label{th:ct}
\begin{enumerate}
\item For each $v\in R$, the integral $\eta_v(s)$ is absolutely convergent for $s\in\Omega_{-1}$.
\item The integral $\eta(s)$ converges absolutely for $s\in\Omega_0$,
and
$$
\eta(s)=\phi(s)\prod_{\alpha\in\mathcal{A}/\Gamma_K} \zeta_{K_\alpha}(m_\alpha s-n_\alpha+1)
$$
where $\zeta_{K_\alpha}$ is the Dedekind zeta function of $K_\alpha$,
and $\phi(s)$ is a bounded holomorphic function for $s\in\Omega_{-1/2+\epsilon}$, $\epsilon>0$.
\end{enumerate}
\end{Thm}

The first claim follows from \cite[Lemma 8.2]{CT2} for non-archimedean place.
For archimedean $v$, if $\H_{ v}^2$ is regular on $\bU$,
the same proof applies. Since any two norms on a finite dimensional vector spaces are
equivalent to each other, this implies the first claim for any local height $\H_{v}$.
The second claim is \cite[Corollary 11.4]{CT2}.
See also the recent preprint \cite{CT3}.

\begin{Cor}\label{c:volume}
If $a(L)>0$, then there exist a polynomial $P$ of degree $b(L)-1$ and
 $\delta>0$ such that
$$
V(t)=t^{a(L)}P(\log t)+O(t^{a(L)-\delta}) \quad \hbox{as $t\to\infty$}.
$$

\end{Cor}

\begin{proof}
It follows from Theorem \ref{th:ct} and the properties of the Dedekind zeta functions
that for some $\e>0$,
$\eta(s)$ has a meromorphic continuation to the region $\Omega_{-1/2+\epsilon}$
with a single pole at $s=a(L)$ of order $b(L)$. Moreover,
in this region, $\eta(s)$ satisfies the bound
$$
\left|\frac{(s-a(L))^{b(L)}}{s^{b(L)}}\eta(s)\right|\le
c \cdot |1+\op{Im}(s)|^N
$$
for some $c,N>0$. Hence, the claim follows from the Tauberian theorem (see
 \cite[Thm. 4.4]{GMO} and \cite[Appendix]{CT1}).
\end{proof}

\subsection{Volumes of homogeneous varieties}\label{sub:vol}
 We additionally assume
that the subgroup $\bL$ is connected.
We recall the properties of orbits of algebraic groups over local fields and over adeles.

\begin{Lem}\label{open}
\begin{enumerate}
\item For each $v\in R$, the space $\bU(K_v)$ consists of finitely many
$\bG(K_v)$-orbits, and each orbit is open and closed.
\item For almost all $v$, $\mathcal{G}(\mathcal{O}_v)$ acts transitively on $\mathcal{U}({\Cal O_v})$.
\item The orbits of $\bG(\A)$ in $\bU(\A)$ are open and closed.
%\item For any $\bga$-orbit $(x_v)\bga$ in $V(\A)$,
%we have $x_v \bG(K_v)=w_0 \bG(K_v)$ for almost all $v\in R$.
\end{enumerate}
\end{Lem}

\begin{proof}
The orbits in (1) are open by \cite[Ch.3,\S3.1]{PR}. This also implies that every orbit is closed.
The finiteness of $\bG(K_v)$-orbits follows from finiteness of Galois cohomology over local fields
(see \cite[Ch.3, \S6.4]{PR}).
(2) follows from Lang's theorem \cite{La} and Hensel's lemma (see \cite[Lemma 1.6.4]{BR}). (3) follows from (1) and (2).
%For (3),
%since $x_v\in V_{\Cal O_v}$ for almost all $v$,
%$x_v\in w_0\bG(\Cal O_v)$ by (2). Hence (3) follows.
\end{proof}

\begin{Thm}\label{t:volume}
Assume that there are finitely many $\bG(\A)$-orbits in $\bU(\A)$.
Let $x\in \bU(\A)$ and
$$
V(x,t):=\mu( x\bG(\A)\cap B_t).
$$
Then $a(L)>0$ and
there exist a nonzero polynomial $P_x$ of degree $b(L)-1$ and $\delta>0$
such that $$
V(x,t)=t^{a(L)}P_{x}(\log t)+O(t^{a(L)-\delta}) \quad \hbox{as $t\to\infty$}.
$$

\end{Thm}

\begin{proof}
As in the proof of Corollary \ref{c:volume}, we consider the Mellin transform
\begin{align*}
\eta(x, s) &:=\int_0^\infty t^{-s}\, dV(x,t)\\
&=\int_{x\bG(\A)} \H(y)^{-s}\, d\mu (y)=
|\Delta_K|^{-\frac{1}{2} \text{dim} X}\cdot \prod_v \eta_v(x_v,s)
\end{align*}
where
$$
\eta_v(x_v,s):=\int_{x_v\bG(K_v)} \H_{v}(y)^{-s}\, d\mu_v(y).
$$
By Theorem \ref{thm:main-f-m-orbits},
our assumption implies that for almost all $v$,
$x_v\bG(K_v)=\bU(K_v)$ and hence $\eta_v(x_v,s)=\eta_v(s)$.
Also, by Theorem \ref{th:ct}(1), $\eta_v(x_v,s)$ is absolutely convergent for $s\in\Omega_{-1+\epsilon}$,
$\epsilon>0$. Hence, it follows from Theorem \ref{th:ct}(2), that
\begin{equation}\label{eq:e}
\eta(x,s)=\phi(x,s)\prod_{\alpha\in\mathcal{A}/\Gamma_K} \zeta_{K_\alpha}(m_\alpha s-n_\alpha+1)
\end{equation}
where $\phi(x,s)$ is a bounded holomorphic function for $s\in \Omega_{-1/2+\epsilon}$, $\epsilon>0$.

Note that for almost all $v$, $\bG$ is quasi-split over $K_v$,
and hence
 there is a unipotent one-parameter subgroup of $\bG(K_v)$
acting nontrivially on $\bU(K_v)$. It was shown in \cite{BO2} that this property implies that for some $a'>0$,
$$
\mu_v(\{y\in x_v\bG(K_v):\, \H_{v}(y)<t\})\ge c\, t^{a'}
$$
for all large $t>0$. This implies that $\eta_v(x_v,s)$ has a pole in the region
$\hbox{Re}(s)\ge a'$. Hence, $a(L)>0$.
%Suppose that, in contrary, $a\le 0$. Then $\omega$ extends to a regular differential form on $\bX$
%which defines a Tamagawa measure $m_v$ on $\bX(K_v)$.
%This measure is finite because $\bX$ is projective.
%Since $\omega$ is $\bG$-invariant, the measure $m_v$ is also $\bG(K_v)$-invariant.
%By \cite[Theorem 1.1]{Sh}, there exist a normal $K_v$-subgroup $\bN$
%and $\bG$-invariant subvariety $\bY$ such that the measure $m_v$ is supported on $\bY(K_v)$,
%$\bN$ acts trivially on $\bY$, and $(\bG/\bN)(K_v)$ is compact.
%This contradicts assumption (2). Hence, $a>0$.

Now the claim follows from (\ref{eq:e}) using Tauberian theorem.
\end{proof}

 Denoting by $\bG_\infty^\circ$
the identity component of $\bG_\infty$, we set for $x\in \bU_\infty$,
$$\tilde V_\infty (x,t) :=\mu_\infty (\{ y\in x \bG_\infty^\circ: \H_\infty (y)<
t \})$$
where $\mu_\infty:=\prod_{v\in R_\infty}\mu_v$ and $\H_\infty:=\prod_{v\in R_\infty}\H_v$.
The following is proved in
\cite[Lemma 7.8]{BO2}.
\begin{Thm} \label{eq:v_111}
If $\H_\infty$ is regular and is not constant on $x\bG_\infty^\circ$ for $x\in \bU_\infty$, then
there exist $c_0,\kappa>0$ such that for all $t>0$ and $\epsilon\in (0,1)$,
\begin{equation*}
\tilde V_\infty(x, t(1+\e))-\tilde V_\infty (x, t)  \le c_0\epsilon^\kappa\, (\tilde V_\infty (x,t)+1 ).
\end{equation*}
\end{Thm}

\begin{Prop}\label{mwr}
Assume that there are only finitely many $\bG(\A)$-orbits in $\bU(\A)$,
and that $\H$ is regular.
Let $M=\bG_\infty^\circ M_f$
for a finite index closed subgroup $M_f$ of $\bG(\A_f)$, $x\in \bU(\A)$, and
$$
V^M(x,t):=\mu(xM\cap B_t).
$$
Then
\be\item $$
V^M(x,t)\asymp t^{a(L)}(\log t)^{b(L)-1}.
$$
\item If $\H_{ \infty}$ is not constant on $x_\infty \bG_\infty^\circ $ where $x=x_\infty x_f
\in \bU_\infty\bU_{\A_f}$,
there exist $c_0,\kappa, t_0>0$ such that for every $t>t_0$ and $\epsilon\in (0,1)$,
$$
V^M(x,(1+\epsilon)t)-V^M(x,t) \le c_0\,\epsilon^\kappa V^M(x,t).
$$\ee
\end{Prop}

\begin{proof}
Since $V^M(x,t)\le V(x,t)$, the upper estimate follows from Theorem \ref{t:volume}.
To prove the lower estimate, we write $\bG(\A)=\cup_{i=1}^n Mg_i$ for some
 $g_i\in \bga$.
Then by  \ref{p:well}(2) and invariance of $\mu$, we obtain that for some $c>1$,
$$
V(x,t)\le \sum_{i=1}^n V^{Mg_i}(x,t)\le n\, V^M(x,c\cdot t)\quad\text{for all $t>0$}.
$$
Hence the lower estimate in (1) follows from Theorem \ref{t:volume}.
To prove (2),
consider the decompositions
$$
\bU(\A)=\bU_\infty \bU_{\A_f}\quad\hbox{and}\quad \mu=\mu_\infty \otimes \mu_f
$$
where $\mu_f=\prod_{v\in R_f}
\mu_v$.

%Let $x=x_\infty x_f$ for $x_\infty\in \bU_\infty$ and $x_f\in \bU_{\A_f}$,
Set
\begin{align*}
\tilde V_\infty(t) &:=\mu_\infty(\{y\in x_\infty\bG_\infty^\circ :\, \H_{\infty}(y)<t\}),\\
\tilde V_f(t) &:=\mu_f(\{y\in x_fM_f :\, \H_f(y)<t\})
\end{align*}
where $\H_f=\prod_{v\in R_f}\H_{v}$.

We claim that there exist $\rho_1,\rho_2>0$ such that for every $t>0$,
\begin{equation}\label{eq:v_12}
\tilde V_f(t)\le \rho_1\, V^M(x,\rho_2t).
\end{equation}
Let $\Omega$ be a compact subset of $x_\infty \bG_\infty^\circ$ such that $\mu_\infty(\Omega)>0$
and $\rho_2=\max_{x\in \Omega} \H_{\infty}(x)$. Then
$$
\Omega\cdot \{y\in x_fM_f:\, \H_f(y)<t\}\subset \{y\in xM:\, \H(y)<\rho_2 t\},
$$
and hence
$$
\tilde V_f(t)\le \frac{V^M(\rho_2 t)}{\mu_\infty(\Omega)}.
$$

By Theorem \ref{eq:v_111},
there exist $c_0,\kappa>0$ such that for all $t>0$ and $\epsilon\in (0,1)$,
\begin{equation*}
\tilde V_\infty(t(1+\epsilon))-\tilde V_\infty(t)\le c_0\epsilon^\kappa\, (\tilde V_\infty(t)+1 ).
\end{equation*}

Let $\alpha=\inf_{y\in x_\infty \bG_\infty^\circ } \H_{\infty}(y)>0$. Then using (\ref{eq:v_111}) and (\ref{eq:v_12}),
\begin{align*}&
V^M((1+\epsilon)t)-V^M(t)\\&=\int_{y\in x_f M_f} \left(\tilde V_\infty(\tfrac{(1+\epsilon)t}{\H_f(y)})-\tilde V_\infty(\tfrac{t}{\H_f(y)})\right)d\mu_f(y)\\
&=\int_{y\in x_fM_f: \H_f(y)<\alpha^{-1}2t} \left(\tilde V_\infty(\tfrac{(1+\epsilon)t}{\H_f(y)})-\tilde V_\infty(\tfrac{t}{\H_f(y)})\right)
d\mu_f(y)\\
&\le c\epsilon^\kappa \int_{y\in x_fM_f: \H_f(y)<\alpha^{-1}2t} (\tilde V_\infty(\tfrac{t}{\H_f(y)}) +1 )\, d\mu_f(y)\\
&\le c\epsilon^\kappa \left(\int_{y\in x_fM_f} \tilde
V_\infty(\tfrac{t}{\H_f(y)})d\mu_f(y)+ \mu_f(\{y\in x_fM_f: \H_f(y)<\alpha^{-1}2t\})\right)\\
&= c\epsilon^\kappa \left(V^M(t)+\tilde V_f(\alpha^{-1}2t)\right)\\
&\le c\epsilon^\kappa \left(V^M(t)+\rho_1 V^M(\rho_2\alpha^{-1}2t)\right)
\\&\le c'\epsilon ^\kappa V^M(t) \end{align*}
for some $c'>1$, where the last inequality holds by the claim (1).

 This completes the proof.

\end{proof}

\begin{Thm}\label{mwrfinite}
Assume that there are only finitely many $\bG(\A)$-orbits in $\bU(\A)$.
Let $M$ be a finite index closed subgroup of $\bga$ and $W$
 a compact open subgroup of $\bG (\A_f)$ contained in $M\cap W_{\H}$.
Fixing $x\in \bU(\A)$, set
 $\tilde B_t:=B_t\cap xM$ for $t>0$.
\begin{enumerate}
\item We have $$\mu(\tilde B_t)
\asymp t^{a(L)}(\log t)^{b(L)-1}.$$

\item
Suppose that $\H$ is regular.
Then there exists $c>0$ such that for any
$\e>0$, there exists a neighborhood $U_\e$ of $e$ in $M$ such
that for all sufficiently large $t$,
 \begin{equation}\label{wr1} (1-c\cdot \e) \mu (\tilde B_t U_\e W)\le \mu (\tilde B_t) \le (1+c\cdot \e)
\mu (\cap_{u\in U_\e W}\tilde B_t u) .\end{equation}

\end{enumerate}\end{Thm}
\begin{proof}
Consider the subgroups $M_\infty=M\cap \bG_\infty$ and $M_f=M\cap \bG({\A_f})$, which are closed %finite index
subgroups of $\bG_\infty$ and $\bG({\A_f})$ respectively.
 Then $\bG_\infty^\circ$ is a finite index subgroup
of $\bG_\infty$ contained in $M_\infty$, and $M_0:=\bG^\circ_\infty M_f$ is a finite index subgroup of $M$.
Hence, $xM=\sqcup_{i=1}^n x m_i M_0 $ for some $m_i\in M$.
Therefore in proving the above claims (1) and (2), we may assume without loss of generality that
$M=\bG_\infty^\circ M_f$ for some finite index subgroup $M_f$ of $\bG(\A_f)$.

Note that
 any height function $\H=\H_{\mathcal L}$ on $\bU(\A)$ is equivalent to
a regular height function, i.e., there is an adelic metrization $\mathcal L'$ such that
for some $c\ge 1$,
$$c^{-1}\cdot \H_{\mathcal L}(y)\le \H_{\mathcal L'}(y) \le c\cdot \H_{\mathcal L}(y)$$
for all $y\in \bU(\A)$.
Hence the claim (1) follows from Proposition \ref{mwr}.

Let $x_\infty$ denote the $\bU_\infty$-component of $x$.
If $\H_{ \infty}$ is
 constant on $x_\infty \bG_\infty^\circ$,
then $\H$ is invariant under $\bG_\infty^\circ$.
Hence $\H$ is invariant under $\bG_\infty^\circ\times W$.
Therefore by taking $U_\e$ to be $\bG_\infty ^\circ \times W$,
$(B_t\cap xM ) u=B_t\cap x M$ for all $u\in U_\e$,
and hence $\tilde B_t$ satisfies \eqref{wr1}.

Now suppose that
 $\H_{ \infty}$ is non-constant on $x_\infty \bG_\infty^\circ$.
Let $\kappa$, $c_0$ and $t_0$ be as in Proposition \ref{mwr} (2).

For $\e>0$ small, take a neighborhood of $V_\e$ of $e$ in $G_\infty^\circ $
such that for all large $t$,
 $$B_t V_\e \subset B_{(1+\e^{1/\kappa} ) t}\quad\text{  and}\quad
 B_{(1-\e^{\kappa})t}\subset \cap_{v\in V_\e}  B_{t} v.$$
We may assume that this holds for all $t>t_0$
by replacing $t_0$ by a larger number if necessary.
Set $U_\e=V_\e \times W$. Since $W\subset W_{\H}\cap M_f$, for all $t>t_0$,
\begin{equation}\label{wh}
\tilde B_t U_\e  W\subset \tilde B_{(1+\e^{1/\kappa})t} \quad\text{
and}\quad
\tilde B_{(1-\e^{1/\kappa})t}\subset  \cap_{u\in U_\e W} \tilde B_t  . \end{equation}
By Proposition \ref{mwr} (2), there is $c>0$ such that for all $t>t_0$,
we have $$\mu(\tilde B_{(1+\epsilon^{1/\kappa})t} -\tilde B_{(1-\epsilon^{1/\kappa}) t }) \le c\, \e\,   \mu(\tilde B_{t}) .$$
Using \eqref{wh}, this proves (2).

\end{proof}

\section{Wonderful varieties}\label{s:won}

\subsection{Symmetric varieties }
We review some basic properties of symmetric varieties and their wonderful compactifications
due to De Concini and Procesi (see \cite{DP} for details).

Let $\bf{G}$ be a connected semisimple algebraic subgroup
and $\sigma:\bf{G}\to\bf{G}$ an involution of $\bf{G}$.
We denote by $\bf{L}$ the normalizer of the subgroup $\bf{G}^\sigma$
of invariants of $\sigma$. Then $\bf{L}\ba \bG $ is
 a {\it symmetric variety}.

A torus $\bf{T}\subset\bf{G}$ is called {\it $\sigma$-split} if $\sigma(t)=t^{-1}$
for every $t\in\bf{T}$. Let ${\bf T}_1$ be a $\sigma$-split torus of maximal dimension
and $\bf{T}$ a maximal torus containing ${\bf T}_1$.
Then $\bf{T}$ is invariant under $\sigma$ and it is an almost direct product
${\bf T}={\bf T}_1 {\bf T}_0$ where ${\bf T}_0$ is the subtorus of ${\bf T}$
on which $\sigma$ acts trivially.

Let $\Phi$ be the set of roots of $\bf{T}$. We set
$$
\Phi_0=\{\alpha\in\Phi:\, \alpha^\sigma=\alpha\}\quad\hbox{and}\quad \Phi_1=\Phi\backslash \Phi_0.
$$
One can choose a Borel subgroup $\bf B$ containing $\bf T$ such that
the corresponding set $\Phi^+$ of positive roots has the property that
\begin{equation}\label{eq:inv}
(\Phi_1\cap \Phi^+)^\sigma= -(\Phi_1\cap \Phi^+).
\end{equation}
For a root $\alpha\in \Phi$, we set $\tilde \alpha=\alpha-\alpha^\sigma$.
The set $\tilde\Phi=\{\tilde\alpha\}$ is a (possibly nonreduced) root system
of rank $\dim({\bf T}_1)$ with the set of simple roots $\Delta_\sigma:=\tilde\Delta\backslash\{0\}$.

Let $\Lambda$ be the weight lattice and $\Lambda^+\subset\Lambda$ the set of dominant integral weights.
For $\lambda\in\Lambda^+$, we denote by $\iota_\lambda:{\bf G}\to\hbox{GL}(V_\lambda)$
the corresponding irreducible representation with the highest weight $\lambda$.
The weight $\lambda$ is called {\it spherical} if there exists a nonzero vector  $v_0\in V_\lambda$
such that $\hbox{Lie}(\bL)\cdot v_0=0$. Let $\Omega^+\subset\Lambda^+$ be the subset of spherical
weights. Every spherical weight $\lambda$ satisfies $\lambda^\sigma=-\lambda$.
Since every dominant weight lies in the interior of the cone generated by positive roots,
this implies that every dominant spherical $\lambda$ can written as
$$
\lambda=\sum_{\alpha\in\Delta_\sigma} n_\alpha \alpha
$$
for some $n_\alpha\in\mathbb{Q}^+$.
A weight $\lambda$ is called {\it $\sigma$-regular} if $(\lambda,\alpha)\ne 0$ for
all $\alpha\in\Delta_\sigma$.

\subsection{Wonderful compactification of a symmetric variety}
Given a $\sigma$-regular spherical representation $\iota: {\bf G}\to\hbox{GL}(V)$
and a nonzero vector $v_0\in V$ fixed by $\bL$, one defines the {\it wonderful
compactification} ${\bf X}$ of $\bf L\ba G $ as the closure of $\bU:=[v_0] {\bf G}$ in
the projective space $\mathbb{P}(V)$.
It was proved in \cite{DP} that $\bX$ satisfies the following properties:
\begin{enumerate}
\item $\bf X$ is a Fano variety.
\item ${\bf X}\backslash \bU$ is a divisor with normal crossings and has smooth
irreducible components ${\bf X}_1$,\ldots,${\bf X}_l$ where $l=\dim({\bf T}_1)$.
\item The closures of $\bG$-orbits in $\bX$ are precisely the partial intersections of $\bX_i$'s.
\item $\bf X$ contains the unique closed $\bf G$-orbit ${\bf Y}:=\cap_{i=1}^l {\bf X}_i$
isomorphic to $\bf P\ba G$ where $\bf P$ is the parabolic subgroup with the unipotent radical
$\exp(\oplus_{\alpha\in \Phi_1\cap \Phi^+} \mathfrak{g}_\alpha)$.
\item $\bf X$ is independent of the representation $\iota$.
\end{enumerate}

We also recall a description of the Picard group of $\bX$.
The map $\hbox{Pic}({\bf X})\to \hbox{Pic}({\bf Y})$ induced by the inclusion $\bY\to \bX$ is injective.
The Picard group of ${\bf Y}\simeq {\bf P\ba G}$ can be identified with a sublattice
of the weight lattice $\Lambda$ and under this identification
\begin{align*}
\hbox{Pic}({\bf X})& \;\longrightarrow \hbox{ sublattice generated by $\Omega^+$},\\
[\bX_i] &\; \longrightarrow \; \alpha,\quad\quad \alpha\in\Delta_\sigma,\\
-K_{\bX} &\;\longrightarrow\; \sum_{\beta\in\Phi_1\cap \Phi^+} \beta+\sum_{\alpha\in\Delta_\sigma} \alpha.
\end{align*}

Now we assume that $\bG$ is defined over a number field $K$, the representation $\iota$
is $K$-rational and $v_0\in V(K)$. Then the action of the Galois group $\Gamma_K$
preserves the unique open $\bG$-orbit $\bU$ and permutes the boundary components
${\bf X}_1,\ldots,{\bf X}_l$. The identification of $\hbox{Pic}(\bX)$
with a sublattice of the weight lattice $\Lambda$ is $\Gamma_K$-equivariant with respect to
the twisted Galois action on $\Lambda$.

\subsection{Wonderful varieties}
A generalization of the wonderful compactification was introduced in \cite{lu1}.
A smooth connected projective $\bG$-variety $\bX$ is called {\it wonderful} of rank $l$ if
\begin{enumerate}
\item $\bX$ contains $l$ irreducible  $\bG$-invariant
divisors $\bX_1, \cdots, \bX_l$ with strict normal crossings.
\item $\bG$ has exactly $2^l$ orbits in $\bX$.
\end{enumerate}

It follows that $\bX$ contains unique open $\bG$-orbit, which we denote by $\bU$, and that
the irreducible components of the divisor $\bX\setminus \bU$ are $\bX_1,\cdots, \bX_l$.
Fix $u_0\in\bU$ and set $\bL=\hbox{Stab}_{\bG}(u_0)$.

In the following, we assume that the subgroup $\bL$ is semisimple.
A description of the $\hbox{Pic}({\bf X})$ was
 given by Brion \cite[Proposition~2.2.1]{b}. Since $\bL$ is semisimple, $\hbox{Pic}(\bL\ba \bG)$ is finite, and
it follows that $\hbox{Pic}({\bf X})$ is a finite extension of the free  abelian group generated by
$[{\bf X}_i]$, $i=1,\ldots,l$. Then by \cite[Lemma~2.3.1]{b},
the cone $\Lambda_{\hbox{\tiny eff}}(\bX)\subset
\hbox{Pic}(\bX)\otimes\mathbb{R}$ of effective divisors
 is generated by
$[{\bf X}_i]$, $i=1,\ldots,l$. The cone of ample divisors was computed in \cite[Section~2.6]{br}.
Combining this description with \cite[Lemma~2.1.2]{b}, it follows that
the ample cone is contained in the interior of effective cone.
The canonical class $K_{\bX}$ was computed in \cite{bi}. The formula
from \cite{bi} implies, in particular, that  $-K_{\bX}$ lies
in the interior of the effective cone $\Lambda_{\hbox{\tiny eff}}(\bX)$.

For an ample line bundle $L$ on $\bX$, we define
\begin{align*}
a_L&:=\inf\{a:\, aL+K_{\bX}\in\Lambda_{\hbox{\tiny eff}}(\bX)\},\\
b_L&:=\hbox{the maximal codimension of the face of $\Lambda_{\hbox{\tiny eff}}(\bX)$ containing $a_LL+K_{\bX}$}.
\end{align*}
Since $L$ and  $-K_{\bX}$ belong to the interior of $\Lambda_{\hbox{\tiny eff}}(\bX)$,
the parameter $a_L$ is well-defined and $a_L>0$.

\begin{Rem}\label{rem:ab}{\rm
In the case when $\bX$ is the wonderful compactification of a symmetric variety,
and $L$ is the restriction of $\mathcal{O}_{\mathbb{P}(V)}(1)$ to $\bX$,
the parameters $a_L$ and $b_L$ can be computed in
terms of the highest weight $\lambda_\iota$ of the representation $\iota$ as follows. Writting
\begin{align*}
\sum_{\beta\in \Phi_1\cap \Phi^+} \beta =\sum_{\alpha\in \Delta_\sigma} m_\alpha \alpha\quad\hbox{and}\quad
\lambda_\iota  =\sum_{\alpha\in\Delta_\sigma} n_\alpha \alpha,
\end{align*}
we have
\begin{align*}
a_L=\max\left\{\frac{m_\alpha+1}{n_\alpha}:\alpha\in\Delta_\sigma\right\}\quad\hbox{and}\quad
b_L=\#\left\{\alpha\in\Delta_\sigma/\Gamma_K:\, a_L=\frac{m_\alpha+1}{n_\alpha}\right\}.
\end{align*}
}
\end{Rem}

 Luna showed that any wonderful variety is spherical. It follows
that $\bL$ has finite index in its normalizer. Hence
using Theorem \ref{ad-extension},
we have an invariant global section $\s$ of any ample line bundle $L$ such that
$\bU\subset \{\s\ne 0\}$.
Since any ample line bundle is contained in the interior of
the cone of effective divisors, it follows that
$\bU=\{\s\ne 0\}$.

\begin{Cor}\label{c:wonder_vol}
For any adelic metrization $\mathcal L$ of
a very ample line bundle $L$ of a wonderful variety $\bX$,
there exist a polynomial $P_{\H_{\mathcal L}}$ of degree $b_L-1$ and $\delta>0$ such that
 $$
\mu(\{x\in \bU(\mathbb{A}):\, \H_{\mathcal L}(x)<t\})=t^{a_L}P_{\H_{\mathcal L}}(\log t)
 +O(t^{a_L-\delta})\quad \hbox{as $t\to\infty$}.
$$

\end{Cor}
\begin{proof} Since $\bX\setminus \bU$ is a divisor
whose irreducible components are given by $\bX_i$, $1\le i\le l$,
and $\Lambda_{\text{eff}}(X)$ is generated by $\bX_i$'s,
we have $a(L)=a_L$ and $b(L)=b_L$
 for $a(L)$ and $b(L)$ defined in \eqref{defab2}.
Hence the claim is a special case of Corollary \ref{c:volume}.
\end{proof}

In the same way,
the following is a special case of Theorem \ref{t:volume}:
\begin{Cor}\label{c:won_vol}
Assume that there are only finitely many $\bga$-orbits in $\bU(\A)$.
Then for $\mathcal L$ as in the above corollary
and for every $x\in \bU(\A)$,
there exist a polynomial $P_{\H_{\mathcal L},x}$ of degree $b_L-1$ and $\delta>0$ such that
$$
\mu(\{y\in x\bG(\A):\, \H_{\mathcal L}(y)<t\})=t^{a_L}P_{\H_{\mathcal L},x}(\log t)+O(t^{{a_L}-\delta}) \quad \hbox{as $t\to\infty$}.
$$
\end{Cor}

\subsection{Examples}\label{sec_ex}

\begin{enumerate}
\item (group varieties)
Let $\iota:{\bf L}\to \hbox{GL}(W)$ be an adjoint semisimple algebraic group defined over a number field $K$.
Then $\iota({\bf L})$ is a homogeneous variety of ${\bf G}={\bf L}\times {\bf L}$ with the action
$$
(l_1,l_2)\cdot x=\iota(l_1)^{-1}\cdot x\cdot \iota(l_2).
$$
The stabilizer of identity is the symmetric subgroup corresponding to the involution
$\sigma(l_1,l_2)=(l_2,l_1)$. Let $\bf S$ be a maximal torus of $\bf L$ with a root system $\Phi_L$
and set of simple roots $\Delta_{\bL}$. Then ${\bf T}={\bf S}\times {\bf S}$ is a maximal torus of ${\bf G}$ and
\begin{align*}
\Phi_1^+&=\{(\alpha,-\beta):\, \alpha,\beta\in \Phi_{\bL}^+\},\\
\Delta_\sigma&=\{(\alpha,-\alpha):\, \alpha\in \Delta_{\bL}\}.
\end{align*}
Let $\lambda_\iota$ be the highest weight of the representation $\iota$ and $\rho$
the sum of roots in $\Phi_{\bL}^+$. Then the highest weight for the corresponding
representation of $\bG$ is $(\lambda_\iota,-\lambda_\iota)$, and the sum of positive roots of
$\bG$ is $(2\rho,-2\rho)$. Writing
\begin{align*}
2\rho =\sum_{\alpha\in \Delta_{\bL}} m_\alpha \alpha\quad\hbox{and}\quad
\lambda_\iota  =\sum_{\alpha\in\Delta_{\bL}} n_\alpha \alpha,
\end{align*}
we have
\begin{align*}
a=\max\left\{\frac{m_\alpha+1}{n_\alpha}:\alpha\in\Delta_\sigma\right\}\quad\hbox{and}\quad
b=\#\left\{\alpha\in\Delta_{\bL}/\Gamma_K:\, a=\frac{m_\alpha+1}{n_\alpha}\right\}.
\end{align*}
This formulas agree with the ones obtained in \cite{GMO}.

\item (space of symplectic forms)
Consider the space $\bU$ of symplectic forms of dimension $2n$
modulo the equivalence defined by scaling.
It can be identified with the symmetric variety $\bU=\bL\ba \bG$
where $\bG=\hbox{PGL}_{2n}$ and $\bL=\hbox{PSp}_{2n}$. Note that $\bL$ is the
set of fixed points of the involution
$$
\sigma(g)=-J{}^tg^{-1}J
$$
where $J=\sum_{i=1}^n E_{i, 2n-i+1} -\sum_{i=1}^n E_{n+i, n-i+1}$.
Consider the maximal torus in $\hbox{Lie}(\bG)$ given by
$$
\mathfrak{t}=\{\hbox{diag}(u_1,\ldots,u_n,v_n,\ldots,v_1):\, \sum_{i=1}^n(u_i+v_i)=0\} .
$$
Then
$$
\sigma(s_1,\ldots,s_d,t_d,\ldots,t_1)=(-t_1,\ldots, -t_d, -s_d,\ldots, -s_1),
$$
and we have the decomposition $\mathfrak{t}=\mathfrak{t}_0+\mathfrak{t}_1$
where
$$
\mathfrak{t}_0=\{u_i=-v_i,\, 1\le i\le n\}\quad\hbox{and}\quad
\mathfrak{t}_1=\{u_i=v_i,\, 1\le i \le n\}.
$$
The root system $\Phi$ is given by
$$
\Phi=\{\alpha_{ij}:=s_i-s_j,\,\beta_{ij}:=t_i-t_j,\, \gamma_{kl}:=s_k-t_l:\;\; 1 \le i\ne j,k,l\le n\},
$$
and $\Phi_0=\{\gamma_{kk}:\;\; 1\le k\le n\}$. If we choose the set of positive roots as
$$
\Phi^+=\{\alpha_{ij},\beta_{ij},\gamma_{kl}:\;\; 1\le i<j, k,l\le n\},
$$
then (\ref{eq:inv}) holds, the set of simple roots is
$$
\Delta=\{\alpha_{i,i+1},\beta_{i,i+1},\gamma_{n,n}:\;\; 1\le i\le n-1 \},
$$
and
$$
\tilde \Delta=\{\alpha_i:=\alpha_{i,i+1}+\beta_{i,i+1}:\;\; 1\le i\le n-1\}\cup\{0\}.
$$
The sum of positive roots is given by
$$
2\rho=2\left(\sum_{i=1}^{n-1} i(2n-i) \alpha_i+n^2\gamma_{n,n}  \right).
$$
Since $\{\alpha_i\}$ forms a basis of $\mathfrak{t}_1^*$ and
$$
\left(\sum_{\beta\in\Phi_1\cap \Phi^+} \beta\right)|_{\mathfrak{t}_1}=2\rho|_{\mathfrak{t}_1},
$$
it follows that
$$
\sum_{\beta\in\Phi_1\cap \Phi^+} \beta= 2\sum_{i=1}^{n-1} i(2n-i) \alpha_i.
$$
Now take an irreducible spherical representation $\iota:\bG\to\hbox{GL}(V)$
with its regular highest weight given by
$$
\lambda_\iota=\sum_{i=1}^{n-1} n_i\alpha_i
$$
and $v_0\in V(K)$ such that $\bL=\hbox{Stab}_{\bG}([v_0])$.
We then have an embedding
$$
\iota:\bU\to \bX:= \overline{[v_0]\bG}:\, u\mapsto [v_0] u
$$
of the space of symplectic forms in its wonderful compactification $\bX$.
%We have the following result on the counting of $K$-rational points
%\begin{align*}
%\#\{u\in\bU(K):\, \H(\iota(u))<T\}\sim c\, T^a(\log T)^{b-1}
%\end{align*}
%where  $c>0$ and
The parameters $a$ and $b$ are
computed as follows

$$a=\max_{1\le i\le n-1} \left\{\frac{2 i(2n-i)+1}{n_i}\right\};
\quad \text{and}$$
$$
b=\#\left\{i=1,\ldots,n-1:\, a=\frac{2 i(2n-i)+1}{n_i}\right\}.
$$

\end{enumerate}

\section{Equidistributions of Adelic periods}\label{s:ea}
%\section{Theorems in $S$-arithmetic settings}
Let $\bG\subset \GL_N$ be a connected semisimple $K$-group.
Let $S$
be a finite subset of $R$ which contains all
archimedean valuations $v\in R$ such that
$\bG(K_v)$ is non-compact.
 This assumption is needed so that the diagonal
embedding of $\bG(\mathcal O_S)$
into $\bGS$ is a discrete
subgroup of $\bGS $.
Let $\Gamma\subset \bG(\mathcal O_S)$ be a finite index subgroup;
hence $\Gamma$ is a lattice in $\bGS$.

\begin{Def}
\begin{itemize}
\item $S$ is called isotropic for $\bG$ if for any connected non-trivial
normal $K$-subgroup $\bN$ of $\bG$,
$\bN_S=\prod_{v\in S} \bN(K_v)$
 is non-compact.
\item $S$ is called strongly
isotropic
for $\bG$ if $S$ contains $v$ such that
every $K_v$-normal subgroup $\bN$ of $\bG$ is isotropic over $K_v$, i.e.,
$\bN(K_v)$ is non-compact.
\end{itemize}\end{Def}
Clearly a strongly isotropic subset for $\bG$ is isotropic for $\bG$.

%Note that for a given $\bM$, there are only finitely many $v$'s
%such that $\{S_\infty, v\}$ is not strongly isotropic for $\bG$, since
%$\bG$ is quasi-split over $K_v$ for almost all $v$ (cf. [PR, Theorem 6.7]).

For any connected semisimple $K$-subgroup $\bL$ of $\bG$,
 $\pi:\tilde \bL\to \bL$ denotes the simply connected covering,
that is, $\tilde \bL$ is a connected simply connected semisimple $K$-group
and $\pi$ is a $K$-isogeny. Note that $\pi$ induces
a map $\tilde \bL(K_v)\to \bL(K_v)$
for each $v\in R$, which is no more surjective in general.

\begin{Def} $\bG$ satisfies strong approximation property
with respect to $S$ if the diagonal embedding of $\bG(K)$ into
$\bG(\A_S)$ is dense.
\end{Def}

For the property
of strong approximation theorems for algebraic groups,
we refer to \cite{PR}, for instance, see Proposition 7.2 and  Theorem 7.12 in \cite{PR}
for the following:
\begin{Thm}\label{sta}
\begin{itemize}
\item If $S$ is isotropic for $\bG$, then
$\tilde \bG$ satisfies the strong approximation property with respect to $S$.
\item If $v\in S$ is isotropic for $\bG$,
then $\tilde \bG(\mathcal O_S)$ is dense in $\tilde \bG_{S\setminus\{v\}}$.
\end{itemize}\end{Thm}

% Then $\tilde \bL(K_v)$ acts on $\G\ba \bGS$ via $\pi$.

 Following Tomanov \cite{T}, we define the following:
\begin{Def}\label{classf} A connected $K$-subgroup $\mathbf{P}$ of $\bG$ is in
  {\rm class
$\Cal F$} relative to $S$ if
 the radical of $\bP$ is
unipotent and every $K$-simple factor of $\bP$ is
  $K_v$-isotropic for some $v\in S$.
\end{Def}

The following is well-known, see
\cite[Lemma 5.1]{EMS1}, for example.
\begin{Lem}\label{aniso} Let $\bL\subset \bG$ be connected reductive algebraic $K$-subgroups
with no non-trivial $K$-character.
 The following are equivalent:
\be
\item the centralizer of $\bL$ is anisotropic over $K$.
\item $\bL$ is not contained in any proper $K$-parabolic subgroup
of $\bG$.
\item any $K$-subgroup of $\bG$ containing $\bL$ is reductive.
\ee
\end{Lem}

Set $X_S:=\G\ba \bG_S$, and let
$\mathcal P(X_S)$ denote the space of all Borel probability measures of
$X_S$.
Let $\{\bL_i\}$ be a sequence of connected semisimple $K$-subgroups of $\bG$.
We denote by $\nu_{i}\in \mathcal P(X_S)$ the unique invariant probability measure
supported on $Y_{i,S}:=\G\ba \G  \pi(\tilde \bL_{i,S})$.
For a given $g_i\in \bG_S$, $g_i\nu_i$
denotes the translated measure: $(g_i\nu_i)(E)=\nu_i(Eg_i^{-1})$ for
Borel subsets $E\subset X_S$
\begin{Thm}\label{rc}\label{ssm}
Let $S$ be strongly isotropic for all $\bL_i$.
 \be\item Suppose
that the centralizer of each $\bL_i$ is anisotropic over $K$.
Then $\{g_i\nu_{i}\}$
is relatively compact in $\mathcal P(X_S)$.

\item If $g_i\nu_i$ weakly converges to $\nu\in \mathcal P (X_S)$ as $i\to \infty$, then
the followings hold:
\be
\item There exists a connected $K$-subgroup $\bM$ in class $\mathcal F$
(with respect to $S$)
such that $\nu$ is the invariant
 measure supported on $\G\ba \G Mg$ for some closed subgroup $M$ of $\bM_S$ with finite index
and for some $g\in \bG_S$.

\item There exists a sequence $\{\gamma_i\in \Gamma\}$ such that for all sufficiently large $i$,
$$\gamma_i {\bL_i} \gamma_i^{-1}\subset
\mathbf M.$$
\item There exists $\{h_i\in \pi(\tilde \bL_{i,S})\}$ such that
 $\gamma_ih_i g_i$ converges to $g$ as $i\to \infty$.
\item If the centralizers of $\bL_i$'s are $K$-anisotropic, $\bM$ is semisimple.
\ee\ee
\end{Thm}

\begin{Def}\label{mt}\label{mumford} For a closed subgroup $L$ of $\bG_S$,
 the {\it Mumford-Tate subgroup}
of $L$, denoted by $\MT(L)$, is defined to be the smallest
 connected $K$-subgroup of $\bG$ such that
$$\overline L^{0} \subset \prod_{v\in S} \MT(L)$$
where $\overline L^\circ$ denotes the identity component of
the Zariski closure of $L$ in $\bG_S$.
\end{Def}
In this terminology, $\bM$ in the above theorem \ref{rc}
is the Mumford-Tate subgroup of $M$.

Theorem \ref{rc} will be proved in sections \ref{sec:ms}
and \ref{sec:dm} (see \eqref{ded:ssm} and \eqref{ded:ssm2}).
 We will deduce Theorem \ref{dmadele} from Theorem \ref{rc}
in the rest of this section.

\begin{Lem}\label{group} If $G_0$ is a subgroup of $\bga$ and
  $\bG(K)G_0$ contains $[\bga, \bga]$,
then $\bG(K)G_0$ is a normal subgroup of $\bga$.

\end{Lem}
\begin{proof}
Let $\gamma_i\in \bG(K)$ and $g_i\in G_0$.
Using the notation $[g, h]=ghg^{-1}h^{-1}$,
 $$\gamma_1g_1 \gamma_2 g_2=\gamma_1\gamma_2 [\gamma_2^{-1},
g_1]g_1g_2 \in \bG(K)G_0 ;$$
$$(\gamma_1 g_1)^{-1}=\gamma_1^{-1} [\gamma_1, g_1^{-1}] g_1^{-1}
\in \bG(K)G_0 ;$$
and for any $g\in \bga$,
$$g(\gamma_1 g_1^{-1})g^{-1}=\gamma_1[\gamma_1^{-1}, g^{-1}]
[g^{-1}, g_1] g_1 \in \bG(K)[\bga, \bga] G_0=\bG(K) G_0.$$
This proves the claim.
\end{proof}
Of course the same argument shows a more genereal lemma that
if $H_1$ and $H_2$ are subgroups of a group $G$
and $[G, G]\subset H_1H_2$, then $H_1H_2$ is a normal subgroup of $G$.
 
%\subsection{Strong approximation properties}

\begin{Prop}\label{finite} Let $S$ be isotropic for $\bG$.
For any compact open subgroup $W_S$ of $\bG(\A_S)$, the product
 $G_{W_S}:=\bG(K) \pi(\tilde \bG_S) W_S$ is a
co-abelian (normal) subgroup of finite index of $\bga$, which contains
$\pi (\tilde \bG(\A))$.
\end{Prop}
\begin{proof}
%Fix $v\in S$ which is strongly isotropic for $\bG$ and denote by
%$S_f$ the subset all non-archimedean valuations in $S$.
% Let $W_0\subset [\bG_{S_f-\{v\}}, \bG_{S_f-\{v\}}]$
%be a compact open subgroup, so that $W_1:=W_0W_S$ is a compact
%open subgroup of $\bG (\A_{\{v\}})$.

Consider the exact sequence $1\to F \to \tilde \bG\to \bG\to 1$.
This induces
the exact sequence $$ \tilde \bG(\A)\to_{\pi} \bG(\A) \to \prod_v H^1(K_v, F)$$
(see the proof of Proposition 8.2 in \cite{PR}).
Since $\prod_v H^1(K_v, F)$
is abelian, it follows that
$[\bga, \bga]\subset \pi (\tilde \bG(\A))$.

Since $\tilde \bG$ has
the strong approximation property with respect to $S$,
$$\tilde \bG(\A_S)=
\tilde \bG(K) \pi^{-1}(W_S).$$

%Because of the choice of $v$,
% $$[\bG_v, \bG_v]=\pi [\tilde \bG_v, \tilde \bG_v]=\bG _v^+$$

Therefore we have
\begin{align*}
[\bga,\bga] &\subset \pi (\tilde \bG(\A))=\pi (\tilde \bG(\A_S))
\pi(\tilde \bG_S)\\
 &\subset \bG(K)  \pi(\tilde \bG_S) W_S .
\end{align*}
Hence the claim follows from the above lemma.
\end{proof}

\begin{Cor}\label{c4.9}\label{wfinite}
\begin{enumerate}
 \item
If $S$ is strongly isotropic for $\bG$ and $G_0$ is a subgroup
of finite index in $\bG_S$ and $W_S$ is an open compact subgroup
of $\bG(\A_S)$, then $\bG(K)G_0W_S$ is a normal subgroup
of finite index of $\bga$.
\item
For any compact open subgroup $W$ of $\bG(\A_f)$,
the product $$G_W:=\{\gamma x w \in \bga:\gamma\in
\bG(K),\, x\in \pi(\tilde\bG(\A)), \; w\in W\}$$
is a normal subgroup of finite index in $\bga$ which contains
$\pi (\tilde \bG(\A))$.
\end{enumerate}
\end{Cor}
\begin{proof}
Let $v\in S$ be strongly isotropic for $\bG$.
Then
$\pi(\tilde \bG(K_v))$ coincides with the subgroup $\bG(K_v)^+$ generated by all unipotent one parameter subgroups
of $\bG(K_v)$, and hence
$G_0$ contains $\pi(\tilde \bG(K_v))$.
Choose any compact open subgroup $W_0$ of $\bG_{S-\{v\}}$.
Then $$\bG(K)G_0W_S\supset \bG(K) \pi(\tilde \bG(K_v)) (W_0W_S)$$
which is a normal subgroup of $\bga$ with finite index and
contains $[\bga, \bga]$
by the previous corollary.
Therefore by Lemma \ref{group}, the first claim follows.
For the second claim,
 let $S$ be a strongly isotropic subset of $\bG$.
Let $W_S<W$ be a compact open subgroup of $\bG(\A_S)$. Then
$G_{W_S}$ is a co-abelian normal subgroup of $\bga$ of finite index by
Proposition \ref{finite}. Since $G_W=G_{W_S}W$, the claim follows.
\end{proof}

%Recall from the introduction
%$G_{W_f}=\bG(K)\pi(\tilde \bG(\A)) W_f$.
%\begin{Cor}\label{small} For any small compact open subgroup $W_f$ of $\bG(\A_{f})$, there is an isotropic $S$ for $\bG$ and
%a compact open subgroup $W_S\subset W_f\cap \bG(\A_S)$ such that
%  $G_{W_f}=G_{W_S}. $
%\end{Cor}
%\begin{proof}
% By definition, $W_f$ is, modulo $\pi(\tilde \bG_S))$, of the form
%$ W_S$ for some isotropic $S$, so that $G_{W_f}=G_{W_S}$.
%. Hence
%$G_{W_f}$ is a normal subgroup of finite index of $\bga$ by the above lemma.
%\end{proof}

For an isotropic set $S$ for $\bG$,
 and a compact open subgroup $W_S$ of $\bG(\A_S)$,
 every element $g$ of $G_{W_S}$ can be written as
$$g=(\gamma_g,\gamma_g) (g_S, w)$$
where $\gamma_g\in \bG(K)$ and $g_S\in \pi(\tilde \bG_S)$ and $w \in W_S$
(here we are using the identification $\bG(\A)=\bG_S\times \bG(\A_S)$).
The choice of $g_S\in \bG(K)$ is unique up to the left multiplication by
the elements of the group
$$\G:=\{\gamma\in \bG(K): \gamma\in W_S, \;\;\gamma\in \pi(\tilde \bG_S)\}=(\bG(K)\cap W_S)\cap
\pi (\tilde \bG_S) .$$

\begin{Lem}\label{bij1}
Let $\bL$ be a connected semisimple $K$-subgroup of $\bG$
and assume that $S$ is isotropic both for $\bL$ and $\bG$.
Let $g\in G_{W_S}$.

\be
\item The map $g\mapsto g_S$ induces a $\pi(\tilde \bG_S)$-equivariant homeomorphism, say $\Phi$,
between $\bG(K)\ba G_{W_S}/W_S$ and $\G\ba  \pi(\tilde \bG_S) $
where $\G=\bG(K)\cap \pi(\tilde \bG_S)\cap W_S$.

\item The map $\Phi$ maps
 $\bG(K)\ba \bG(K)\pi(\tilde \bL(\A)) g W_S/W_S$
onto $\Gamma \ba \Gamma (\gamma_g ^{-1} \pi (\tilde \bL_S) \gamma_g) g_S$,
inducing a measurable isomorphism between them.

\item
If $\mu$ is the invariant probability measure supported on
$\bG(K)\ba \bG(K) \pi(\tilde  \bL(\A))$ (considered as a measure on $\bG(K)\ba G_{W_S}$),
then the measure $g.\mu$, considered as a functional on
$C_c(\bG(K)\ba G_{W_S})^{W_S}$,
is mapped by $\Phi$ to the invariant probability measure supported on
$\Gamma \ba \Gamma (\gamma_g ^{-1} \pi(\tilde \bL_S) \gamma_g) g_S $,
which will be denoted by $\Phi_*(g.\mu)$.
%In fact, the total volume of $\Phi(\g.\mu)$ is given by
% the inverse of
%$$\mu_0{ ( gW_f g^{-1}\cap \pi(\tilde \bL(\A_S)))$$
%where $\\pi(\tilde \bL(\A_S))
\ee
\end{Lem}

\begin{proof}
It is easy to check
that the map $g\mapsto (g_S, w)$ induces a $\pi(\tilde \bG_S)$-equivariant
 homeomorphism
between $\bG(K)\ba G_{W_S}$ and $\G\ba  (\pi(\tilde \bG_S)\times W_S) $.
The first claim then follows. For (2), let $h\in \pi(\tilde \bL(\A))$.
Since $S$ is isotropic for $\bL$,
we have
$$\pi(\tilde \bL(\A)) \subset \bL(K) \pi(\tilde \bL_S)  ( gW_S g^{-1}\cap
\pi(\tilde \bL(\A_S))) .$$
Without loss of generality, we may assume
that the $W_S$-component of $g$ is $e$, i.e., $w=e$.
We can write $h=(\delta, \delta)(h_S, gw'g^{-1})$
where $\delta\in \bL(K)$, $h_S\in \pi(\tilde \bL_S)$ and $w'\in W_S\cap  g^{-1}
\pi (\tilde \bL(\A_S)) g$. Note that $gw'g^{-1}=\gamma_g w'\gamma_g^{-1}$.
So $hg=(\delta \gamma_g, \delta \gamma_g) (\gamma_g^{-1} h_S\gamma_g g_S, w')$
and hence $$\Phi [hg]=\Gamma\ba \Gamma \gamma_g^{-1} h_S \gamma_g g_S .$$
This also explains the measurable isomorphism between
$$\bG(K)\ba \bG(K)\pi(\tilde \bL(\A)) g W_S/W_S\quad\text{and}$$
$$\Gamma \ba \Gamma (\gamma_g ^{-1} \pi(\tilde \bL_S) \gamma_g) g_S ,$$
%\times \left(gW_S g^{-1}\cap \pi(\tilde \bL(\A_S))\right)$$
and proves the third claim.
\end{proof}

\noindent{\bf Proof of Theorem \ref{dmadele}, assuming Theorem \ref{rc}}
Let $\{\bL_i\}$, $g_i$ and $\mu_i$ be as in the introduction.
Let $S$ be any isotropic subset for $\bG$
which intersects with $\cap_i\mathcal I_{\bL_i}$ non-trivially
and
 contains all archimedean valuations $v$ such that $\bG(K_v)$ is non-compact.
  Fixing some
compact open subgroup $W_S$ of $\bG(\A_S)$, let $\Phi$ be the map
defined in Lemma \ref{bij1} for this choice of $S$ and $W_S$.
If the first claim in Theorem \ref{dmadele} does not hold,
 then
$\Phi_*(g_i \mu_i)$ is not relatively compact in $\G\ba \G\pi(\tilde \bG_S)$,
which contradicts Theorem \ref{rc}.

%Let $S$ be a strongly isotropic subset for all $\bG$ and
% $\bL_i$, $i\in \n$.
To prove the second claim,
let $W_S$ be  a compact
open subgroup of $\bG(\A_S)$, and
set $X_{W_S}=\bG(K)\ba G_{W_S}$.
Letting $\Phi$ be the map
defined in Lemma \ref{bij1} for this choice of $S$ and $W_S$, we have
$\Phi_*(g_i\mu_i)$ weakly converges to $\Phi_*(\mu)$ in the space of Borel
measures of $X_S:=\G\ba \bG_S$ and $\G:=\bG(K)\cap W_S\cap \pi(\tilde G_S)$.

%We first claim that $c_i\to c$ and hence
% $\mu_i'\to \mu'$. To see this, for $d_i=c_ic^{-1}$, since
%$d_i\mu_i \to \mu'$, using the first  claim we obtain that
%for any $\e>0$, there exists a compact subset
%$\Omega\subset \G\ba \pi(\tilde \bG_S)$ such that
%for all sufficiently large $i$,
%$$1-\e< d_i\mu_i'(\Omega)<1+\e,\quad\text{and}\quad
%1-\e<\mu_i'(\Omega)<1+\e .$$
%This implies that $d_i\to 1$, and
%from this it is easy to deduce that $\mu_i'\to \mu'$.

Write $$g_i=(\gamma_{g_i}, \gamma_{g_i})( g_{i,S}, w_i)$$
where $\gamma_{g_i}\in \G, g_{i, S}\in \pi(\tilde \bG_S)$ and $w_i\in W_S$.
Then the measure $\Phi_*(g_i\mu_i)$ is precisely
same as $g_{i,S}\nu_i$ where $\nu_i$
is the invariant probability measure
supported on
$\Gamma \ba \Gamma (\gamma_{g_i} ^{-1} \pi (\tilde \bL_{i,S}) \gamma_{g_i})$.

Applying Theorem \ref{ssm} to $\bG_S$ and
the groups $\gamma_{g_i}^{-1}\bL_i \gamma_{g_i} $ and
$g_{i,S}$,
we obtain a connected $K$-group $\bM\in \mathcal F$
(with respect to $S$), $g\in \pi(\tilde \bG_S)$, $\gamma_i\in \G$
and $h_i\in \gamma_{g_i}^{-1}\pi(\tilde \bL_i)\gamma_{g_i}$, which depend on a priori
$S$ and $W_S$, such that
$$\gamma_i
\gamma_{g_i} ^{-1}  \bL_{i} \gamma_{g_i}\gamma_i^{-1}\subset \bM$$
and $\gamma_ih_i g_{i, S}\to g$,
and that for some finite index subgroup $M$ of $\bM_S$,
$\Phi_*(\mu)$ is the  invariant probability
 measure supported on $\G\ba \G M g$.

% $g\in [\bG_S, \bG_S]$
%and for some $\delta_i\in \Gamma$,
%$\delta_i^{-1} \gamma_i^{-1} L_i \gamma_i \delta_i\subset M$.

%Moreover, if $\bM$ is the Mumford-Tate group of $M$,
%then $\bM$ belongs to class $\Cal F$ relative to $S$
%and $M$ is a finite index subgroup of $\bM_S$.

We now claim that $\bM$ can be taken simultaneously for any $S$ as above and
all $W_S$. We denote $(M, g)$ by $i(S, W_S)$.
Let $W_S$ and $W_{S'}$ be open compact subgroups of $\bG(\A_S)$ and
$\bG(\A_{S'})$ respectively.
Without loss of generality, we may assume that $S\subset S'$
and $W_{S'}\subset W_{S}$.
Let $W_0<W_S$ be a compact open subgroup of $\pi(\tilde \bG_{S'-S})$
 and set $W':=W_{S'}W_0$. Then $G_{W_{S'}}=G_{W'}$.
%Hence if $W_f=\prod W_v$ and $W_f':=\prod W_v'$ then
%$W_v=W_v'$ for almost all $v$ and $W_v'\subset W_v$.
%Hence $W_f'$ is $S$-small as well.
%Then for $W_S:=\prod_{v\notin S}W_v$ and
% $W_S':=\prod_{v\notin S}W_v'$,
%$$G_{W_f'}=G_{W_S'}\subset G_{W_f}=G_{W_S}.$$
We set $\G:=\bG(K)\cap W_S\cap \pi(\tilde \bGS)$ and
$\G':=\bG(K)\cap W'\cap \pi(\tilde \bGS)$.
If $\Phi$ denotes the bijection of $\bG(K)\ba G_{W_S}/W_S$
and $\Gamma\ba \pi(\tilde \bG_S)$ and $\Phi'$ similarly for $W_S'$,
then $\Phi_*(\mu)$ and $\Phi'_*(\mu)$ are invariant measures
supported on $\G\ba \G Mg$ and $\G\ba \G M' g'$ for
 some $g, g'\in\pi(\tilde\bG_S)$.
Here $M$ and $M'$ are
finite index subgroups of $\bM _S$ and $\bM'_S$ respectively
where $\bM$ and $\bM'$ denote the Mumford-Tate subgroups of $M$ and $M'$ respectively.
Since both $\Phi'_*(\mu)$ and $\Phi_*(\mu)$ are the limits of images of $g_i\mu_i$,
it is clear that the canonical projection from
$\Gamma'\ba \pi(\tilde \bG_S)\to \Gamma\ba \pi(\tilde \bG_S)$
maps $\Phi'_*(\mu)$ to $\Phi_*(\mu)$.
Therefore $g'=\gamma m g$ for some $\gamma\in \Gamma$ and $m\in M$,
and $\Phi_*(\mu)$ is invariant under $g^{-1}Mg$
which implies that $g^{-1}\bM g =g'^{-1}\bM' g'$ (see Lemma \ref{zd} below).
Hence $m^{-1} \gamma^{-1}\bM' \gamma m =\bM$
or equivalently $\gamma^{-1}\bM' \gamma =\bM$.
Therefore by replacing $M'$ by $\gamma^{-1} M' \gamma$,
$\Phi'_*(\mu)$ is the invariant probability
 measure supported on $\G'\ba \G'M'g'$
where the Mumford-Tate subgroup of $M'$ is $\bM$.
Hence $\bM=\bM'$.

Therefore for a fixed $\bM$,
 we have associated to every $S$ and $W_S$
 a finite index subgroup
$M$ of $\bM_S$ and $g\in \bGS$ such that $i(S, W_S)=(M, g)$, proving claim.

%Since $C_c(X_{W_f})^{W_f}\subset C_c(X_S)^{W_S}$
%and similarly for $W_f'$ this proves the claim.

%Since $S$ is strongly isotropic for $\bM$, there is $v\in S$
%such that $\bM$ is strongly isotropic for $v$ and
%it follows that $M$ contains $\bM(K_v)^+=\pi(\tilde \bM(K_v))$.
%Since $\bM(K)M(\bM(\A_S)\cap W_S)$
%contains $\bM(K)\pi(\tilde \bM(K_v))W'$ for some compact open
%subgroup $W'$ of $\pi(\tilde \bM(\A_{\{v\}})$,
%$\bM(K)M(\bM(\A_S)\cap W_S)$ contains $\pi(\tilde \bM(\A))$
%and consequently $[\bM(\A), \bM(\A)]$.
%Therefore by Corollary \ref{c4.9},
%$\bM(K)M(\bM(\A_S)\cap W_S)$ is a normal subgroup of $\bM(\A)$.
%Since $\bM(K)\bM_S (\bM(\A_S)\cap W_S)$ is a normal
%subgroup of $\bM(\A)$ with finite index and $M$ has finite index
%in $\bM_S$,
%$M_0:=\bM(K)M(\bM(\A_S)\cap W_S)$ has finite index in $\bM(\A)$.

Now fix one $S$ which is also strongly isotropic for $\bM$,
and set $i(S, W_S)=(M, g)$.
Then $ M_0:=\bM (K) M  (\bM(\A_S)\cap W_S)$
is a finite index normal subgroup of $\bM(\A)$ by Corollary \ref{c4.9}.
Let $dm$ and $dw$ denote the Haar measures on $ M$
and $\bM(\A_S)\cap W_S$ respectively such that
$dm$ and $d(m\otimes w)$ induce probability measures on
$\G\ba \G M$ and
 $\G\ba \G (M \times (\bM(\A_S)\cap W_S))$ respectively.

For $f\in C_c(X_{W_S})^{W_S}$,
\begin{align*}
\mu(f)&= \Phi_*(\mu) (\Phi(f))=\int_{\G\ba \G M }\Phi( f)(mg)\;dm\\
&= \int_{\G\ba \G( M \times (\bM(\A_S)\cap W_S))}\Phi (f)(mg)\; dm \,dw\\
&=\int_{\bG(K)\ba \bG(K) M_0} f (m_0g) \;dm_0
\end{align*}
where $dm_0$ is the invariant probability measure on
$\bG(K)\ba \bG(K) M_0$.
Since
$C_c(X, W_f)$
and $C_c(X_{W_S})^{W_S}$ can be canonically identified,
 this finishes the proof.
\vs

If the sequence $g_i\mu_i$
weakly converges to $\mu\in \mathcal P(X)$,
we say that the orbits $Y_i g_i$ become equidistributed in $X$
with respect to the measure $\mu$.

%When $\bG$ is simply connected, the invariant probability measure
%on $X$ is precisely given by the Tamagawa measure.

\begin{Cor}\label{cm}
Let $\bG$ be simply connected and
 $\{\bL_i\}$ be a sequence of
semisimple simply connected maximal connected $K$-subgroups of $\bG$
 and  $\cap_i \mathcal I_{\bL_i}\ne \emptyset$.
Then for any sequence $g_i\in \bga$,
 either of the following holds:
\be
\item the sequence $x_0\bL_i(\A) g_i$
 is equidistributed in $\bG(K)\ba \bga$ with respect
to the invariant measure as $i\to \infty$,
\item there exist $i_0\in \n$, $\{\delta_i\in \bG(K)\}$
and $g\in \bga$ such that for infinitely many $i$,
$$\delta_i^{-1}\bL_i \delta_i
= \bL_{i_0};\quad\text{ and hence }
x_0\bL_i(\A) g_i
= x_0\bL_{i_0}(\A)\delta_i g_i$$ and $l_i \delta_i g_i$
converges to $g$ for some $l_i\in \bL_{i_0}(\A)$.
\ee
\end{Cor}
\begin{proof}
Since $[\operatorname{N}_\bG(\bL_i):\bL_i]<\infty$ and $\bL_i$ are semisimple,
their centralizers are $K$-anisotropic.
Hence by Theorem \ref{dmadele},
 $\{g_i\mu_i\}$ are weakly compact in the space
of probability measures on $\bG(K)\ba \bga$.
Let $\mu$ be a weak-limit and let $\bM$ be
as in Theorem \ref{dmadele}.
If $\bM\ne \bG$, by passing to a subsequence,
we have
 $\bL_i$'s are conjugate with each other by
elements of $\bG(K)$. Hence we may assume
$\bL_i=\delta_i^{-1}\bL_{i_0}\delta_i$
for some $\delta_i\in \bG(K)$ and
$\delta_ih_i g_i \to g$ for some $h_i=\delta_i^{-1} l_i \delta_i$
with $l_i\in \bL_{i_0}(\A)$.
Hence (2) happens.

Now suppose for every weak-limit $\mu$,
we have $\bM=\bG$. Fix a finite subset $S\subset R$
such that $R_\infty\subset S$ and
 $S\cap (\cap_i \mathcal I_{\bL_i}) \cap {\mathcal I}_{\bG} \ne \emptyset$.
Since $\bG$ is simply connected,
$G_{W_S}=\bga$ for any compact open subgroup $W_S$ of
$\bG(\A_S)$, and the restriction of $\mu$
to $C_c(\bG(K)\ba \bga)^{W_S}$
is the Tamagawa measure, since $M_0=\bga$ for any $W_S$.

Since $\cup_{W_S}  C_c(X)^{W_S}$ is dense
in $C_c(X)$, this implies $\mu$ is an invariant measure.
Therefore $g_i\mu_i$ converges to the invariant measure
and yields the equidistribution (1).

\vs
\noindent{\bf Proof of Theorem \ref{am}: }
If we set $\bG_0:=\bG\times \bG$
and  $\Delta (\bG)$ denotes the diagonal embedding of $\bG$ into
$\bG_0$,
it can be easily seen that the above Adelic mixing is equivalent to
the equidistribution of the translates $x_0 \Delta(\bG)(\A)(e,g_i)$
in the space $\bG_0(K)\ba  \bG_0(\A)$ for any $g_i\to\infty$;
for the function $f:=f_1\otimes f_2$, $f_i\in C_c(\bG(K)\ba \bG(\A))$,
 $$\int_{x_0 \Delta(\bG)(\A)} f(x, x g_i) dy
=\int_{X} f_1(x) f_2(xg_i) dx .$$

 Since $\bG$ is almost $K$-simple,
$\Delta (\bG)$ is a maximal connected
$K$-subgroup of $\bG_0$, we may apply Theorem \ref{am}.
If the second case happens, we have $\delta_i$ belongs to
the normalizer of $\Delta (\bG)$.
Since $\Delta(\bG)$ has finite index in its normalizer,
by passing to a subsequence, we may assume $\delta_i=e$.
Now since $g_i\to \infty$, we cannot have $l_i\in \Delta (\bG)(\A)$
such that $l_i g_i$ is convergent. Therefore (2) of Theorem \ref{am}
cannot happen, and consequently the claim is proved.
\end{proof}

\begin{Rem}
{\rm In the above and the next corollary,
the assumption on the maximality of $\bL$
 appears more than what we need, which is
that
$\bL$ is maximal as a semisimple $K$-group and $[\operatorname{N}_\bG(\bL_i):\bL_i]<\infty$.
However for $\bL$ semisimple, $[\operatorname{N}_\bG(\bL):\bL]<\infty$
is same as the centralizer of $\bL$ being finite, and any
connected $K$-group containing a semisimple group
with a finite centralizer is automatically semisimple (cf. \cite{EMV}).}
\end{Rem}

We now prove an analogue of Corollary \ref{cm} when $\bG$ and $\bL$ are not necessarily simply connected.

\begin{Cor}\label{cmgeneral} Let $\bL$ be a semisimple maximal
connected $K$-subgroup of $\bG$.
Let $W$ be a compact open subgroup of $\bG(\A_f)$, and let
$g_i\in G_{W}$ be a sequence going to infinity modulo $\bL(\A)$.
 Let $\nu$ be the invariant probability measure
supported on $\bL(K)\ba (\bL(\A)\cap G_{W})$
considered as a measure on $X_{W}:=\bG(K)\ba G_{W}$.
Then for any $f\in C_c(X_{W})^{W}$,
$$\lim_{i\to \infty} \int_{x\in X_{W}} f (xg_i) \, d\nu(x) =\int_{X_{W}} f d\mu$$
where $\mu$ is the probability Haar measure on $X_{W}$.
\end{Cor}
\begin{proof} Let $S$ be a strongly isotropic subset for $\bL$.
Since $G_{W}$ contains
$\bL(K)(\pi(\tilde \bG_S)\cap \bL_S)(W_S\cap \bL(\A_S))$,
by Corollary \ref{c4.9},
$G_{W}$ contains $\pi(\tilde \bL(\A))$.
Also, by the same corollary, for each $g_i\in G_{W}$,
$\bL(K)\pi(\tilde \bL(\A))(g_iWg_i^{-1}\cap \bL(\A_f))$
is a normal subgroup of $\bL(\A)\cap G_{W}$ with finite index.
Hence there exists a finite subset $\Delta_{g_i}\subset \bL(\A)\cap G_{W}$
such that
$$\bL(\A)\cap G_{W}=\cup_{x\in \Delta_{g_i}} \bL(K) \pi(\tilde \bL(\A))
x(g_iW g_i^{-1}\cap \bL(\A_f)) $$
where the union is a disjoint union.
Therefore for $f\in C_c(X_{W})^{W}$,
 the integral $(g_i\nu)(f)$ is equal to a finite linear
combination of integrals of $f$ against invariant measures on
$x_0 \pi(\tilde \bL(\A)) x g_i$, $x\in \Delta_{g_i}$.

Hence it suffices to show the following: for any $x_i\in \Delta_{g_i}$,
and $f\in C_c(X_{W})^{W}$,
$$\int_{x_0\pi(\tilde \bL(\A))x_ig_i} f \, d\mu_i \to \int f \,d\mu $$
where $\mu_i$ is the invariant probability measure supported
on $x_0\pi(\tilde \bL(\A))x_ig_i$.

We apply Theorem \ref{dmadele} for any weak-limit $\nu$ of $\mu_i$.
By (1), we have $\nu\in \mathcal P(X_{W})$.
We claim $\bM=\bG$. Suppose not. Since $\bL$ is maximal,
we have $\delta_i\in \bG(K)$ and $h_i\in \pi(\tilde \bL(\A))$, $g\in G_{W}$ such that
$\delta_i \bL \delta_i^{-1} =\delta_j \bL \delta_j^{-1}$ for all large $i$ and $\delta_i h_ix_ig_i\to g$.
Since $\bL$ has a finite index in the normalizer of $\bL$,
by passing to a subsequence, there exist $\delta_0\in \bG(K)$,
and $\delta_i\in \delta_0\bL(\A)$ such that $(\delta_0^{-1} \delta_i) h_i x_ig_i \to \delta_0^{-1} g$.
Since $\delta_0^{-1}\delta_i h_i\in \bL(\A)$ and $x_ig_i\to \infty$ modulo $\bL(\A)$, this is a contradiction.
Hence by Theorem \ref{dmadele},
$\nu$ is an invariant measure supported on $x_0 M_0 g$
where $M_0$ contains $\bG(K)\pi(\tilde \bG(\A))W$.
Since $G_{W}=\bG(K)\pi(\tilde \bG(\A))W$,
we conclude that $\nu=\mu$, proving the claim.
\end{proof}

%\end{proof}

\vs

\section{Counting rational points of bounded height}\label{s:c}
The basic strategy is due to Duke, Rudnick, Sarnak \cite{DRS}, and to Eskin-McMullen \cite{EM},
which can be summarized as follows.
Let $L\subset G$ be unimodular locally compact groups
and $Z:=L\ba G$. Let $\mu_G, \mu_L$ and $\mu$ be
invariant measures on $G$, $L$ and $Z$ respectively which are
compatible with each other, that is,
if for any $f\in C_c(G)$,
$$\int f d\mu_G =\int_{L\ba G} \int_L f (h g)d\mu_L (h)d\mu(Lg) .$$

\begin{Def} For a fixed compact subgroup $W$ of $G$,
a family $\{B_T\subset Z \}$ of compact subsets
 is called $W$-well-rounded if $B_TW=B_T$ for all large $T$ and
there exists $c>0 $ such that for every small
$\e>0$, there exists a neighborhood $U_\e$ of $e$ in $G$ such
that for all sufficiently large $T$,
 \begin{equation}\label{wr} (1-c \cdot \e) \mu (B_T U_\e W)\le \mu (B_T) \le (1+ c\cdot \e)
\mu (\cap_{u\in U_\e W}B_T u) .\end{equation}
\end{Def}
Note that this is a slight variant of the notion of well-roundedness introduced in \cite{EM}.

\begin{Prop}\label{rcounting} Let $\G\subset G$ be a lattice such that
$\G\cap L$ is a lattice in $L$. Let $W\subset G$ be a compact subgroup.
Suppose that for $Y:=[e]L\subset \G\ba G$,
the translates $Yg$ become equidistributed in $\G\ba G$ as $g\to \infty$
in $Z$ with respect to $C_c(\G\ba G)^W$,
that is, for any $f\in  C_c(\G\ba G)^W$,
$$\int_{Y} f(yg) d\mu_L(y)\to \int_{\G\ba G}f\; d\mu .$$

Then for any $W$-well-rounded sequence $\{B_T\subset Z\}$ of compact subsets
whose volume going to infinity,
we have
$$ \# z_0 \G\cap B_T \sim \frac{\mu_L(L\cap \G\ba L)}{\mu_G(\G\ba G)}\mu (B_T) .$$
\end{Prop}
\begin{proof}
Without loss of generality, we may assume that
$\mu_L(L\cap \G\ba L)=1=\mu_G(\G\ba G)$.
Let $U_\e$ be as in the definition \ref{wr}.
We may assume that $U_\e$ is symmetric and $U_\e  \cap \G=\{e\}.$
If we define a function on $\G\ba G$ by
$$F_{B_T}(g):=\sum_{\gamma\in \G\cap L\ba \G}\chi_{B_T} (z_0\gamma g)$$
where $\chi_{B_T}$ is the indicator function of $B_T$,
then $F_{B_T}(e)=\#(z_0\G\cap B_T)$.
Let $\psi_\e$ be a non-negative $W$-invariant continuous function on $\G\ba G$ with support in $\G\ba
\G U_\e W$
and with integral one.
Set $F_T^{+}=F_{B_TU_\e W}$ and $F_T^-=F_{\cap_{u\in U_\e W} B_T u }$.
Observe that for any $g\in U_\e W$,
$$ F_T^-(g) \le F_{B_T}(e)\le  F_T^+(g),$$
and hence
$$\langle F_T^-, \psi_\e\rangle \le F_{B_T}(e)\le \langle F_T^+, \psi_\e \rangle$$
where the inner product takes place in $L^2(\G\ba G)$.
One can easily see that
$$\langle F_T^+, \psi_\e\rangle =\int_{g\in B_T U_\e W}
\int_{y\in Y } \psi_\e (yg) d\mu_L(y) d\mu (g)$$

By the assumption,
$$\int_{y\in Y} \psi_\e (yg) dy\to 1$$
as $g\to \infty$ on $Z$ and hence
if the volume of  $B_T$
goes to infinity as $T\to \infty$, we have
$$\langle F_T^+, \psi_\e\rangle
\sim \vol(B_TU_\e W) .$$

Similarly, we have
$$ \langle F_T^-, \psi_\e\rangle
\sim \vol(\cap_{u\in U_\e W} B_Tu) .$$

Using the $W$-well-roundedness assumption on $B_T$,
it is easy deduce that $F_{B_T}(e)\sim \mu(B_T)$ (see \cite{BO2} for details).

\end{proof}

%In the following two corollaries, we let $\mu_v$, $v\in R$
%be measures on $v_0\bG(K_v)\simeq \bL(K_v)\ba \bG(K_v)$
%such that $\mu(v_0\bG(\mathcal O_v))=1$ for any non-archimedean
%valuation $v$ and $\mu:=\prod_{v\in R}\mu_v$
%is compatible with $\mu_{\bG}$ and $\mu_{\bL}$.

Let $\bG$ be a connected semisimple algebraic group defined over $K$,
with a given $K$-representation
 $\iota: \bG \to \GL_{d+1}$.
Let $\bU:=u_0\bG\subset \mathbb P^d$
for $u_0\in \mathbb P^d (K)$,
and fix a height function $\H_{\mathcal O(1)}$ on
$\mathbb P^d(K)$ as in the introduction.
That is, $\H_{\mathcal O(1)} =\prod_{v\in R}\H_v$ where $\H_v$ is a norm on $K_v^{d+1}$
and is a max norm for almost all $v$.

We set
$$
N_T(\bU):=\#\{x\in \bU(K):\, \He_{\mathcal O(1)}(x)<T\}.
$$

We assume that
\begin{enumerate}
\item[(i)] $\bL=\hbox{Stab}_{\bG}(u_0)$ is a semisimple maximal proper connected
$K$-subgroup of $\bG$.
\item[(ii)] There are only finitely many $\bG(\A)$-orbits on $\bU(\A)$.
\end{enumerate}
We note that (ii) is equivalent to saying that
for almost all $v\in R$, $\bG(K_v)$ acts transitively on $\bU(K_v)$ (see Thm. \ref{thm:main-f-m-orbits}).
Denote by $\bX\subset \mathbb P^d$  the Zariski closure of
$\bU$ and by $L$ the line bundle which is
the pull back of $\mathcal O_{\mathbb P^d} (1)$.
We assume that there is a global section $\s$ of $L$
such that $\bU=\{\s\ne 0\}$.
By Theorem \ref{ad-extension}, $\s$ is $\bG$-invariant.
 Let $\s_0,\cdots, \s_d$ be the global sections
of $L$ which are the pull-backs of the coordinates $x_i$'s.
Using the height function $\H_{\mathcal O(1)}=\prod_{v\in R}\H_v$,
we define
the adelic height function $\H_{\mathcal L}=\prod_v\H_{\mathcal L,v}:\bU(\A)\to \br_{>0}$
where
$$\H_{\mathcal L, v}(x)=\H_v\left(\frac{\s_0(x)}{\s(x)}, \cdots, \frac{\s_d(x)}{\s(x)}\right).$$

Set $$B_T:=\{x\in \bU(\A): \H_{\mathcal L}(x)<T\} .$$

The assumption (ii) implies
that the set $\bU(K)$ consists of finitely many $\bG(K)$-orbits
(Theorem \ref{thm:main-f-m-orbits}).
Choose a set
$u_1,\ldots,u_l\in\bU(K)$ of representatives of these orbits, and denote by
$\bL_1,\ldots ,\bL_l$ their stabilizers in $\bG$.
Then
$$N_T(\bU)=\sum_{i=1}^l \#(B_T\cap u_i\bG(K)).$$

% We set,
%for each $1\le i\le l$, $$
%V_i(T):=\mu(\{x\in u_i\bG(\A):\, \H(x)<T\}).
%$$

A naive heuristic $$\#B_T\cap u_i\bG(K)\sim_T\text{vol}(B_T\cap u_i\bG(\A))$$ does not hold
in general unless $\bG$ is simply connected. To correct this problem, we consider the following finite
index subgroup of $\bG(\A)$:

%\begin{Def}\label{ssw} A compact open subgroup $W$ of $\bG(\A_f)$
%is called {\it small} if there exists
%a strong isotropic $v\in R$ for both $\bG$ and $\bL$
%such that $W\cap \bG(K_v)\subset \pi(\tilde \bG(K_v))$.
%\end{Def}

Recall from Lemma \ref{l:proper}
that the following is an open
subgroup of $\bG(\A_f)$:
 $$W_{\H_{\mathcal L}}:=\{g\in \bG(\A_f): \H_{\mathcal L}(ug)=\H_{\mathcal L}(u)\;\; \text{for all
$u\in\bU(\A)$}\} .$$
%A small co-finite subgroup of $W_{\H_{\mathcal L}}$ means
%a small
%compact open subgroup $W$ of $\bG(\A_f)$
%contained in $ W_{\H}$.
Recall from Corollary \ref{wfinite} that
for any compact open subgroup $W$ of $\bG(\A_f)$, $$G_W:=\{\gamma x w \in \bga:\gamma\in
\bG(K),\, x\in \pi(\tilde\bG(\A)), \; w\in W\}$$
is a normal subgroup of $\bga$ with finite index.

Let $\mu$ be the Tamagawa measure on $\bU(\A)$,
and choose invariant measures $\mu_G$ and $\mu_{L_i}$
 on the adelic spaces $\bG(\A)$ and $\bL_i(\A)$ respectively
so that $\mu_G=\mu\times \mu_{L_i}$ locally.

The main theorem \ref{maint} in the introduction follows from the following:
\begin{Thm}\label{mcounting}
\begin{enumerate}
\item If the height function $\H_{\mathcal L}$ is regular, then for
any compact open subgroup $W$ of $W_{\H_{\mathcal L}}$
$$
N_T(\bU)\sim_T\sum_{i=1}^l
\frac{\mu_{L_i}(\bL_i(K))\ba G_W\cap \bL_i(\A))}{\mu_{G}(\bG(K)\ba G_W)}
\mu(u_iG_{W}\cap B_T).
$$
\item For $a=a(L)$ and $b=b(L)$ defined as in \eqref{defab2},
 $$
N_T(\bU)\asymp T^a(\log T)^{b-1}.
$$
\item Suppose that $\bG$ is simply connected,
or that $\bG(\A)=G_{W_{\H_{\mathcal L}}}$.
Then for some $c>0$,
$$
N_T(\bU)\sim c \cdot T^a(\log T)^{b-1}. $$
\end{enumerate}
\end{Thm}
\begin{proof}
%Let $S$ be a strongly isotropic subset for both $\bG$ and $\bL$
%such that $G_{W}=G_{W_S}$, and choose any
% compact open subgroup $W_S<W$ of $\bG_{\A_S}$.
Fixing $1\le i \le l$,
we apply the above proposition to $G=G_{W}$, $L=\bL_i(\A)\cap G_{W}$ and
$Y=\bG(K) \ba \bG(K)L \subset \bG(K)\ba G$.

By Corollary \ref{cmgeneral},
the translates $Yg$ become equidistributed in $\bG(K)\ba G_{W}$
with respect to $C_c(\bG(K)\ba G_{W})^{W}$.

%Let $\H_L=\prod \H_{L, v}$ be a regular $\bG(\A)$-admissible height function on $\bU(\A)$
%which extends the height $\H_L$ of $\bU(K)$ in discussion. This exists by
%Theorem \ref{ad-extension} and the example preceding it.

And by Theorem \ref{mwrfinite}, the family $\{B_T\cap u_iG_{W}\}$ is $W$-well-rounded.
Hence (1) follow from Proposition \ref{rcounting}.
(2) follows from (1) using Corollary \ref{mwrfinite}.
For (3), first note that
$G_W=\bga$ for $\bG$ simply connected.  Theorem \ref{t:volume} implies
   $B_T\cap u_i\bG(\A)$ is $W$-well-rounded, and hence (1) holds under the hypothesis of (3),
    without assuming that $\H_{\mathcal L}$ is regular.
It remains to apply the asymptotic given Theorem \ref{t:volume} once more.
\end{proof}

\vs

\noindent{\bf Proof of Corollary \ref{wonderfulcm}}
Since $\bX$ is smooth, $L^k$ is $\bG$-linearized for some $k$ (cf. \cite{kkl}).
Therefore, by replacing $L$ by
$L^k$ if necessary,
we are in the setup of Theorem \ref{mcounting}.
%By the argument in the proof of Corollary \ref {c:wonder_vol},
%we may also assume that $L$ has an invariant section $s$
%such that $\bU\subset \{s\ne 0\}$.
Since $a_L=a(L)$ and $b_L=b(L)$ (see
the proof of Corollary \ref {c:wonder_vol}),
the claim follows from Theorem \ref{mcounting}.

\vs

\noindent{\bf Proof of Corollary \ref{linnik}:}
Let $\|\cdot\|_p$ denote the max norm on $\qp^N$ for each $p$.
Fix any compact subset $\Omega\subset v_0\bG(\br)$ with boundary
of measure zero and $\vol(\Omega)>0$.
If $m=\prod_{p:prime} p^{m_p}$ (of course, $m_p=0$ for almost all $p$),
set $$B_m:=\{(x_p)\in  v_0\bG(\A):
x_\infty\in \Omega, \|x_p\|_p\le p^{m_p}\text{ for each $p$} \}.$$
That is, for $B_m':= v_0\bG(\A_f)\cap \prod_p \bU(m^{-1}\zp)$, we have
$B_m:=\Omega \times B_m'.$
 Since $B_m'$ is invariant under
the subgroup $\prod_p\bG(\zp)$,
the family $\{B_m\}$ is clearly well-rounded.
Moreover since $\bG$ is simply connected,
$G_{W_S}=\bG(\A)$ for any strongly isotropic $S$ for $\bG$.

%Let $\mu=\prod\mu_p$ be an invariant measure on $v_0\bG(\A)$
%which is compatible with Haar measures on $\bG(\A)$ and
%$\bL(\A)$ such that $\vol(\bG(\q)\ba \bG(\A))=1$
%and $\vol (\bL(\q)\ba \bL(\A))=1$ respectively.

By the computation in \cite[Corollary 7.7]{BO2},
$$\mu(B_m):=\mu_\infty(\Omega)\prod_p\mu_p(\bU(m^{-1}\zp)\cap v_0\bG(\qp))
\to \infty$$ if $m \to \infty$, subject to
$B_m\ne \emptyset$.

Therefore by Proposition \ref{rcounting},
we have, as $m \to \infty$, subject to
$B_m\ne \emptyset$,
$$\# v_0\bG(\q)\cap B_m\sim \mu (B_m)$$

Observe that if $x\in \bU(\q)\cap B_m$, then $x\in \bU(m^{-1}\z)\cap\Omega$,
and hence $$\# v_0\bG(\q)\cap B_m \le \#\bU(m^{-1}\z)\cap \Omega .$$
Consequently, for all sufficiently large $m$,
 $B_m \ne \emptyset$
implies  $\bU(m^{-1}\z)\ne \emptyset$.

In the case when $\bL$ is simply connected,
there is exactly one $\bG(\q)$-orbit in each
$\bG(\br)$-orbit and hence
for $\Omega\subset v_0\bG(\br)$
$$\# v_0\bG(\q)\cap B_m =\# \bU(\q)\cap B_m =\# \bU(m^{-1}\z)\cap \Omega .$$
  Hence the above argument shows
(2).

\section{Limits of invariant measures for unipotent flows}\label{sec:ms}

\subsection{Statements of Main Theorem}
Let $K$ be a number field, and
 $\bG$ be a connected $K$-group with no non-trivial $K$-character.
Let $S$ be a finite set of (normalized)
valuations of $K$ including all the archimedean
$v\in R$ such that $\bG(K_v)$ is non-compact.
For each valuation $v\in S$,
we denote by $|\cdot|_v$ the normalized absolute value on the completion field $K_v$,
and by $\theta_v$ the normalized Haar measure on $K_v$.

Let $G$ be a finite index subgroup of
$$\bGS:=\prod_{v\in S} \bG(K_v) ,$$
and $\Gamma$ be an $S$-arithmetic subgroup of $G$,
that is, $\G \subset \bG(K)$ is
commensurable with $\bG (\mathcal O_S) $,
 where $\Cal O_S$ denotes
 the ring of $S$-integers in $K$.
Then $\Gamma$ is a lattice
in $G$ by a theorem of Borel \cite{B}.

%For simplicity, we sometimes omit the subscript $S$ from $\G_S$
%in this section.

% which shows that
%the limits of ergodic invariant measures under
%unipotent one parameter subgroups
%are ergodic in
%$\mathcal{P}(X)$.
%\change{changed the last sentence}
 Recall the definition of {\it class} $\Cal F$-subgroups of $\bG$ from \ref{classf}.
Equivalently, a connected $K$-subgroup $\mathbf{P}$ of $\bG$ is in
  {\rm class
$\Cal F$} relative to $S$ if for each proper normal $K$-subgroup $\mathbf Q$ of $\mathbf P$
there exists $v\in S$ such that $(\bP/\bQ) (K_v)$ contains a non-trivial
unipotent element.

Note that for every subgroup $L$ of finite index
in $\bP_S$ with $\bP\in\mathcal{F}$, the orbit $\G\ba \G L$ is
closed and supports a finite $L$-invariant measure.

For a closed subgroup $L$ of $\bGS$, we denote by $L_u$
the closed subgroup of $L$ generated by all unipotent one-parameter subgroups
of $L$. We note that since $G$ has a finite index in $\bG_S$,
every one-parameter unipotent subgroup of $\bG_S$ is contained in $G$.

\begin{Def}
We say that a closed subgroup $L$ of $G$
is in class $\mathcal H$ if
there exists a connected $K$-subgroup $\bP$ in {\rm class $\Cal F$}
relative to $S$ such that
$L$ has a finite index in $\bP_S$ and
$L_u$ acts ergodically on $\G\ba\G L$ with respect to the
$L$-invariant probability measure.
% and $P_u$ acts ergodically
% on $\G\ba \G P$ where $P_u$ denotes the closed subgroup
% of $P$ generated by all unipotent elements of $P$.
\end{Def}

%\change{Added equivalent definition of class F}

%\change{P is changed to H. It is confusing to have two letters P}

%Note that $P_u=(\bP_S)_u$ since $[\bP_S: P]<\infty$, and that $\G\ba \G P$ is closed.

%\todo{Add remark about ergodicity in (2)}

%\change{Thm \ref{tom} is stated in terms of $\mathcal{H}$.}

%\change{Topological version added}

Set $X=\Gamma\ba G$. We denote by $\mathcal{P}(X)$ the space of
probability measures on $X$ equipped with the weak$^*$ topology.
For $\mu\in \mathcal P(X)$ and $d\in G$,
the translate $d\mu $ is defined by
$d\mu (E)=\mu(Ed^{-1})$ for any Borel subset $E$
of $X$.
We also define the invariance subgroup for $\mu\in \mathcal P(X)$ by
$$\Lambda (\mu)=\{d\in G : d\mu  =\mu\}.$$
For a unipotent one parameter subgroup $u: K_v \to \bG(K_v)$, $x\in X$
and $\mu\in \mathcal P(X)$, the trajectory
$xU$ is said to be uniformly distributed relative to $\mu$
if for every $f\in C_c(X)$,
$$\lim_{T\to \infty}
\frac{1}{\theta_v(I_T)} \int_{t\in I_T} f(xu(t))\;d\theta(t) =
\int_X f(x)\; d\mu(x)$$
where $I_T=\{t\in K_v: |t|_v<T\}$.

We present a generalization of
the theorem of Mozes and Shah in \cite{MS} in  the $S$-arithmetic setting,
which is the main result of this section:
\begin{Thm}\label{ms}
Let $v\in S$ and $\{U_i\}$
be a sequence of one-parameter unipotent
 subgroups of $\bG(K_v)$.
Let $\{ \mu_i: i\in \n\}$ be a sequence of $U_i$-invariant ergodic measures in
$\Cal P (X)$. Suppose that $\mu_i \to \mu$ in
$\Cal P (X)$ and let $x=\G\ba \G g\in \op{supp}(\mu)$. Then the following holds:
\begin{itemize}
\item[(1)] There exists
a closed subgroup $L \in \mathcal H$ such that $\mu$ is
an invariant measure supported on $\G\ba \G L g$.
In particular
$$\operatorname{supp}(\mu)=x\Lambda(\mu).$$
\item[(2)] Let $z_i \to e$ be a sequence in $G$ such that
 $xz_i\in \op{supp}(\mu_i)$ and the trajectory
$\{xz_iU_i\}$ is uniformly distributed with respect to
 $\mu_i$. Then there exists $i_0$ such that for all $i\ge i_0$,
$$\operatorname{supp}(\mu_i)\subset
\operatorname{supp}(\mu)z_i \quad\text{and}\quad
\Lambda(\mu_i)\subset z_i^{-1}\Lambda(\mu)z_i.$$

\item[(3)] Denote by $H$
the closed subgroup generated by the set $\{z_iU_i{z_i}^{-1}:i\geq
i_0\}$.
Then $H\subset g^{-1}L g$ and $\mu$ is
$H$-ergodic.
\end{itemize}
\end{Thm}

%\change{unipotent ``$K_{p_0}$-algebraic'' is removed. I think
%  unipotent is automatically algebraic}
We state some corollaries of the above theorem \ref{ms}, as in \cite{MS}.
Let
$Q(X)$ denote the set $\Cal P(X)$ of probability measures $\mu$ on $X$
such that the group generated by all unipotent one-parameter
subgroups of $G$ contained in $\Lambda(\mu)$ acts ergodically on
$X$ with respect to $\mu$.
The following is an immediate consequence of the above theorem:
\begin{Cor}
\begin{enumerate}
\item $Q(X)$ is a closed subset of $\Cal P(X)$.
\item For $x\in X$,
$Q(x):=\{\mu\in Q(X): x\in \text{supp}(\mu)\}$ is a
closed subset of $\Cal P(X)$.

\end{enumerate}
\end{Cor}

Let $X\cup\{\infty\}$ denote the one-point compactification of $X$.
As well known, $\Cal P(X\cup\{\infty\})$
is compact with respect to the weak$^*$-topology.

Combined with a theorem proved by Kleinbock and Tomanov
(see Theorem \ref{kt}),
we can also deduce:
\begin{Cor}
\begin{enumerate}
\item Let $\{\mu_i\}\in Q(X)$ be a sequence of measures
converging weakly to a measure $\mu\in \Cal P(X\cup\{\infty\})$.
Then either $\mu\in Q(X)$ or $\mu(\{\infty\})=1$.
\item For $x\in X$,
$Q(x)$ is compact with respect to the
weak$^*$-topology.

\end{enumerate}
\end{Cor}

\subsection{ Deduction of Theorem
\ref{ssm} (2) from Theorem \ref{ms}}\label{ded:ssm}
 We will now deduce Theorem \ref{ssm} (2) from Theorem \ref{ms}.
\begin{Lem}\label{l:ergodic}
Let $\bL$ be a connected semisimple $K$-subgroup of $\bG$.
If $S$ is strongly isotropic for $\bL$, then
%\be
%\item $$(\Gamma\cap \pi(\tilde \bL_S)) [\bL_S, \bL_S]= \pi(\tilde \bL_S) .$$
 there exists a one-parameter unipotent subgroup $U=\{u(t)\}$
of $\tilde \bL_S$ which acts ergodically on  $ \G\ba \G \pi({\tilde \bL}_S)$.
%\ee
 \end{Lem}

\begin{proof} Let $v\in S$ be strongly isotropic for $\bL$.
Denote by
 $\mathbf L(K_v)^+$ the subgroup
generated by all unipotent one-parameter
subgroups in $\mathbf L(K_v)$.
Then by \cite{BorT}, $$\pi(\tilde\bL(K_v))=\bL(K_v)^+.$$
%and $\tilde\bL(K_v)=\tilde\bL(K_v)^+$.

First, we show that $\bL(K_{v})^+$
acts ergodically on $\Gamma\ba \G \tilde \bL_S$.
Since $\tilde{\bL}$ satisfies the strong approximation property
with respect to $\{v\}$
and  $\pi^{-1}(\G)\cap  \tilde \bL_S$ is an $S$-arithmetic
subgroup of $\tilde\bL_S$, it follows that
the diagonal embedding of $\pi^{-1}(\G)\cap \tilde \bL_S$
is dense in $\prod_{v\in S\setminus \{v\}} \tilde\bL(K_v)$ by Theorem
\ref{sta}.
This implies that $(\G\cap \pi(\tilde \bL_S))\bL(K_{v})^+$ is dense
in $\pi(\tilde \bL_S)$.

Since $\tilde \bL(K_{v})$ is a normal subgroup
of $\tilde \bL_S$, this implies that
$\pi^{-1}(\G)\cap \tilde \bL_S $ acts ergodically on
$\tilde \bL(K_{v}) \ba \tilde \bL_S$.
By the duality, this implies that $\tilde \bL(K_{v})$
acts ergodically on $\pi^{-1}(\Gamma)\cap \tilde \bL_S \ba \tilde \bL_S$,
and hence on $\Gamma\ba\G \tilde \bL_S$.

Since every $K_{v}$-simple factor of $\bL$ is $K_{v}$-isotropic, there
exists a unipotent one-parameter subgroup $U$ of $\tilde\bL(K_v)$
such that $\tilde\bL(K_v)$ is the smallest normal subgroup containing
$U$. Now by the $S$-algebraic version of Mautner
phenomenon (Proposition \ref{Mautner2}), any $U$-invariant function in
$L^2(\G\ba \G \tilde \bL_S)$ is $\tilde\bL(K_v)$-invariant, and consequently
a constant function.
Hence the ergodicity of the $U$-action $\Gamma\ba
\G \tilde \bL_S$ follows.
\end{proof}

\noindent{\bf Proof of Theorem \ref{ssm} (2):} Fix $v\in S$ which
 is strongly isotropic for all $\bL_i$.
It follows from Lemma \ref{l:ergodic} that measures $\nu_i$ is ergodic
with respect to one-parameter unipotent subgroups $U_i':=g_i^{-1}U_i g_i$, where
 $U_i=\{u_i(t)\}\subset \bL_i(K_v)^+$ is as in Lemma \ref{l:ergodic}.
Hence, we may apply Theorem \ref{ms}(1) to conclude that
$\nu$ is an invariant measure supported on $\G\ba \G M g$
for some closed subgroup $M \in \mathcal H$ and $g\in \bGS$. In particular, $M$
is a finite index subgroup in $\bM_S$ where $\bM$ is the Mumford-Tate subgroup of
$M$ which is in class $\mathcal{F}$ (see Def. \ref{mumford}).
By a pointwise ergodicity theorem,
there exists
a sequence $z_i=g^{-1} \gamma_i h_i g_i \to e$
for some $\gamma_i\in \Gamma$ and $h_i\in \tilde \bL_{i}(K_v)$
and the trajectory $\G\ba \G g z_i U_i'$ is uniformly distributed with respect to
$g_i\nu_i$.
By
Theorem \ref{ms}(2), we have
 $$g_i^{-1} \tilde \bL_i g_i\subset z_i^{-1} (g^{-1}\bM g) z_i$$
for all large $i$. This implies that
$$\gamma_i \tilde \bL_i  \gamma_i^{-1} \subset \bM$$
as well as that $\gamma_i h_i g_i$ converges to $g$ as $i$ tends to infinity.
Now for (d), suppose that the centralizers of $\bL_i$ are $K$-anisotropic.
Since $\bM$ is reductive by Lemma \ref{aniso} and it belongs to class
$\mathcal F$ with respect to $S$, $\bM$ is semisimple.

%To verify (4), suppose the dimensions of
%$\bL$ and $\bM$ are equal. Then by (3),
%$\gamma_i\delta_i$ belongs to the normalizer
%of $\bL$ for all large $i$, and hence
%by passing to a subsequence, $\gamma_i \delta_i\in \bL$.
%By (3), this implies that $g_i$ converges on $\tilde L\ba \bG_S$.
%This contradicts to the assumption on the sequence $(\delta_i, g_i)$.
%\change{Minor modification

\subsection{Measures invariant under unipotent flows}
The crucial ingredient in our proof of Theorem \ref{ms} is a fundamental theorem
of Ratner \cite{R}
on the  classification
of the measures in $\mathcal{P}(X)$ which are ergodic with respect to
unipotent subgroups of $G$.
In the $S$-arithmetic case,  also see \cite{MT1}.

We will use the following more precise description
due to
Tomanov:
\begin{Thm} \cite[Theorem 2]{T}\label{tom}
Let $W$ be a subgroup of
$G$ generated by unipotent one-parameter subgroups.
\begin{enumerate}
\item For any $W$-invariant ergodic probability measure $\mu$
on $X$, there exist a subgroup
$L\in \mathcal{H}$ and $g\in G$ such that
$W\subset g^{-1}Lg$ and $\mu$ is the invariant measure supported
on $\G\ba \G Lg$.
\item For every $g\in G$,
there exists a closed subgroup $L\in\mathcal{H}$ such that
 $W\subset g^{-1}Lg$ and
  $$\overline{\G\ba\G g W}=\G\ba\G Lg .$$
\end{enumerate}
\end{Thm}
Although it is assumed in \cite{T} that
$S$ contains all archimedean valuations of $K$ and
$G=\bGS$, these assumptions are not used in the proof.
%\change{proof of Prop. \ref{mauc} changed}

%\todo{Mautner lemma in [MT] applies to almost algebraic groups ???}

%\change{I removed $H\cap .$. This definition is equivalent to the
% original, but simpler. Also, I removed $\mathbb{Q}_p$.}

%\change{Removed: Note that for a closed subgroup $P\in \H$,
%if $\bP$ is a $K$-subgroup in class $\Cal F$ appearing in
%the definition, $\bP=\MT(P)$, and hence $\bP$ is uniquely determined.}

%\begin{Lem}\label{zd}
%(remove?) If $\bL$ is a connected algebraic
%group over a local field $k$,
%any open subgroup of $\bL(k)$ is Zariski dense
%in $\bL$.
%\end{Lem}

\begin{Lem}\label{zdc}\label{zd}
 For $P, Q\in \mathcal H$,
we have $\MT(P)\subset\MT(Q)$ if and only if
$\G\ba \G P\subset\G\ba \G Q$.
\end{Lem}

\begin{proof} In the proof we shall use the notion of Lie algebra for groups which
are products of real and $p$-adic Lie groups. This is a natural
extension of the classical notion and has the same basic properties
(see, for instance, \cite{MT2}).
Suppose that $\MT(P)\subset\MT(Q)$.
Then $P\cap Q$ has finite index in $P$ and hence $P_u\subset Q_u$.
Since $P_u$ is normal in $P$ and it acts ergodically on $\G\ba \G P$,
it follows that $\overline{\G\ba \G P_u}= \G\ba \G P$.
Hence, $$\G\ba \G P\subset \G\ba \G Q .$$

Conversely, suppose that $\G\ba \G P\subset \G\ba \G Q$.
Since $P$ and $Q$ have finite indices in $\MT(P)_S$ and
$\MT(Q)_S$ respectively, it follows that
the Lie algebra of $\MT(P)_S$ is contained in
the Lie algebra of $\MT(Q)_S$.
Hence, $\MT(Q)_S$ contains an open subgroup of $\MT(P)_S$.
Since such groups are Zariski dense in $\MT(P)$,
we deduce that $\MT(P)<\MT(Q)$.
\end{proof}

%\change{In Lemma \ref{zdc}, = changed to $\subset$}

Let $W$ be a closed subgroup of $G$ generated by
one parameter unipotent subgroups in it.
For each $L\in \mathcal H$, define
\begin{align*}
\Cal N (L, W)&=\{g\in G: W\subset g^{-1} L g\},\\
\Cal S(L, W)&=\cup_{M\in \mathcal H, \MT(M)\subsetneq \MT(L) }
\Cal N(M, W),\\
\Cal T_L(W)&=\pi (\Cal N(L, W)-\Cal S(L,W))
\end{align*}
where $\pi:G\to \G\ba G$ denotes the canonical projection.

Note that for $L\in\mathcal{H}$, $L$ has finite index in $\MT(L)_S$ and
hence $L$ contains the closed subgroup of
$\MT(L)_{S}$ generated by all unipotent elements of $\MT(L)_S$.
Hence
$$\Cal N(L,W)=\{g\in G: W\subset g^{-1}\cdot \MT(L)_S\cdot g\} .$$

Note also that for any $P, Q\in \mathcal H$ with $\MT(P)=\MT(Q)$,
$$\Cal N(P,W)=\Cal N(Q,W);\quad \Cal S(P, W)=\Cal S(Q,W)
\quad \text{and
hence}\quad \Cal T_P(W)=\Cal T_Q(W) .$$

\begin{Lem}\label{ns}
For any $g\in \Cal N(L, W)\setminus \Cal S(L, W)$,
the closure of $\G\ba \G g W$ is equal to $\G\ba \G L g$.
\end{Lem}

\begin{proof}
By Theorem \ref{tom}, there exists $M\in \mathcal H$ such that
 $W\subset g^{-1}M g$ and $\overline{\G\ba \G g W}=\G\ba \G M g$.
Since $\G\ba \G g W\subset \G \ba \G Lg$
and  $\G \ba \G Lg$ is closed,
we have $$\G \ba \G M\subset \G\ba \G L .$$
Hence, by Lemma \ref{zd}, $\MT(M)\subset \MT(L)$ .
 Since $g\notin \Cal S(L,W)$,
$\MT(L)=\MT(M)$ and hence by Lemma \ref{zdc},
$$\G\ba \G L =\G\ba\G M .$$
This proves the lemma.
\end{proof}

Note that Lemma \ref{ns} implies that
\begin{equation}\label{thw}
\Cal T_L(W)=\pi (\Cal N(L, W))-\pi (\Cal S(L, W)).\end{equation}
\begin{Lem}\label{countable}
%\begin{enumerate}
%\item $\H$ is countable.
For $P,Q\in \mathcal H$,  the following are
equivalent:
\begin{itemize}
\item[(i)] $\Cal T_P(W)\cap \Cal T_Q(W)\neq\emptyset$;
\item[(ii)] $\MT(P)=\gamma \MT(Q)\gamma^{-1}$ for some $\gamma\in \G$;
\item[(iii)] $ \Cal T_P(W)= \Cal T_Q(W)$.
\end{itemize}
%\end{enumerate}
\end{Lem}
\begin{proof}
%Since there are only countably many
%possible subgroups
% $H\in \H$ with the same Mumford Tate
%group which is a $K$-subgroup,
%\sidenote{probably only finitely many
%with the same MT group}
%$\H$ is countable.
Suppose $g\in \Cal N(P, W)-\Cal S(P, W)$ and
$\gamma g\in  \Cal N(Q, W)-\Cal S(Q, W)$ for some $\gamma\in \G$.
Then by Lemma \ref{ns},
the closure of $\G \ba \G g W$
is equal to
$$\G P g=\G Q \gamma g=\G \gamma^{-1} Q \gamma g .$$
Hence by Lemma \ref{zdc},
 $$\MT(P)=\MT(\gamma Q \gamma^{-1}) =\gamma \MT(Q)\gamma^{-1}.$$
This shows (i) implies (ii).
If (ii) holds,
then $\Cal N(P, W)=\gamma \Cal N(Q,W)$
and $\Cal S(P, W)=\gamma \Cal S(Q,W)$.
Hence, (iii) follows.
The claim that (iii) implies (i) is obvious.
\end{proof}

%\change{$\mathcal{H}$ - removed. We only need countability of
%  $\mathcal{H}^*$ which is obvious}

%\change{We don't need (\ref{thw}) in the proof: multiplication by
%  $\gamma$ is a bijection}

Let $\Cal F^*$ be the $\G$-conjugacy class
of Mumford-Tate subgroups
of $L\in \mathcal H$. For each $[\bL]\in \Cal F^*$,
choose one subgroup $L\in \mathcal H$ with $\MT(L)=\bL$.
We collect them to a set $\mathcal H^*$.
Note that $\mathcal H^*$ is countable and
the sets $\Cal T_L(W)$, $H\in \mathcal H^*$, are disjoint
from each other.

\begin{Thm}\label{hstar} Let $\mu\in \Cal P(X)$ be a $W$-invariant measure.
For every $L\in \mathcal H$,
let $\mu_L$ denote the restriction of $\mu$ to $\mathcal T_L(W)$.
Then
\begin{enumerate}
\item $\mu=\sum_{L\in \mathcal H^*} \mu_L$.
\item Each $\mu_L$ is $W$-invariant. For any $W$-ergodic component $\nu\in \Cal P(X)$ of $\mu_L$,
there exists $g\in \mathcal N (L, W)$ such that
$\nu$ is the unique $g^{-1}Lg$-invariant measure on $\G \ba \G Lg$.
\end{enumerate}
\end{Thm}
\begin{proof}
We first disintegrate $\mu$ into $W$-ergodic components.
By Theorem \ref{tom}, each of them is of the form $\nu g$
where $L\in \mathcal H$, $g\in \mathcal N (L, W)-\mathcal S(L,W)$, and $\nu$ is the normalized
$L$-invariant measure on $\G\ba \G L$.
Now the claim follows from Lemma \ref{ns}, (\ref{thw}), and
Lemma \ref{countable}.
\end{proof}

%\change{{\bf H} is replaced with MT(H)}

%\change{{\bf N} is removed (it is not used anywhere)}

%\change{$V_S$ is replaced by $V_H$: S is fixed in the whole argument}
\subsection{Linearization}
Let $L\in \mathcal H$.
Let $\mathfrak g$ denote the Lie algebra of $\bG$ and $\mathfrak l$ the Lie subalgebra of
$\MT(L)$. For $d=\rm{dim}(\mathfrak l)$, we
consider the $K$-rational representation $$\land^d \text{Ad}: \bG\to
\GL(\mathbf V_L)\quad\hbox{where}\quad \mathbf V_{L}:=\land ^d\mathfrak g.
$$
We set $V_L=\prod_{v\in S} \mathbf V_L(K_v)$ and fix $p_L\in (\land^d
{\mathfrak l})(K)$, $p_L\neq 0$.

Consider the orbit map $\eta_L:G \to V_L$
given by
$$\eta_L((g_v)_{v\in S}):=(p_L g_v)_{v\in S}.
$$
Let
\begin{align*}
\Gamma_L &:=\{\gamma\in \Gamma: \gamma^{-1}\MT(L)\gamma=\MT(L)\},\\
\Gamma^0_L &:=\{\gamma\in \Gamma: \eta_L(\gamma)=p_L\}=\{\gamma\in\Gamma_L:\det(\text{Ad}(\gamma)|_{\mathfrak{l}})=1\}.
\end{align*}

%\change{definition of $\G_H^\circ$ added}

By Lemma \ref{zdc}, we have $\G\ba\G L= \G\ba\G L\gamma$ for $\gamma\in\G_L$.
This implies that $\gamma\in \Gamma_L$ preserves the volume and
\begin{equation}\label{eq:small}
\prod_{v\in S} \left|\det(\text{Ad}(\gamma)|_{\mathfrak{l}})\right|_v=1.
\end{equation}
Hence, $\eta_L(\Gamma_L)\subset\Cal O_S^\times \cdot p_L$
where $\Cal O_S^\times$ denotes the group of units in $\Cal O_S$.

Following Tomanov \cite[4.6]{T}, we define the notion of $S(v_0)$-small
subsets of $V_L$.
We fix $\delta>0$ such that for any $w\in S$,
any  $\alpha\in\Cal O^\times_S$ satisfying $\max_{v\in S\setminus\{w\}} |1-\alpha|_v<\delta$
is a root of unity in $K$.
\begin{Def} Let $v_0\in S$.
A subset $C=\prod_{v\in S} C_v\subset V_L$
is  {\it $S(v_0)$-small} if for any $v\in S\ba\{v_0\}$ and $\alpha\in
K_v^\times$, $\alpha C_v\cap C_v\neq \emptyset$ implies that $|1-\alpha|_v<\delta$.
\end{Def}
Then for $\alpha\in \Cal O_S^\times$ and $S(v_0)$-small subset $C$ of $V_L$,
\begin{equation*}
\alpha C\cap C\neq \emptyset\quad \Rightarrow \quad \alpha\in \mu_K
\end{equation*}
where $\mu_K$ is the set of roots of unity in $K$.

We set
$$\bar V_L= V_L/\{\alpha \in \mu_K: \alpha p_L\in \eta_L(\Gamma_L)\}.$$

Now $\bar \eta_L$ denotes the composition map
of $\eta_L$ with the quotient map $V_L\to\bar V_L$.

Since $\G$ is an $S$-arithmetic subgroup of
$G$ and $p_L$ is rational,
it is clear that $\bar\eta_L(\G)$ is a discrete subset in $\bar V_L$,
and the map
\begin{equation}\label{eq:proper}
\G_L^0\ba G \to \G\ba G\times \bar V_L:\;
\G_L^0 g \mapsto (\G g,\bar \eta_L(g))
\end{equation}
is proper (see \cite[4.7]{T}).

%\change{reference to Tomanov added}

Denote by $A_L$ the Zariski
closure of $\bar\eta_L(\mathcal N (L, W))$ in $\bar V_L$.
Then (see \cite[4.5]{T})
\begin{equation}\label{ah}
{\bar \eta}_L^{-1}(A_L)=\mathcal N (L, W) .
\end{equation}

%\change{reference to Tomanov added}

\begin{Prop}\label{sd} Let $D$ be a compact $S(v_0)$-small subset of
  $A_L$ for some $v_0\in S$.
Define
$$\mathcal S(D)=\{g\in {\bar \eta}_L^{-1}(D): \gamma g\in
{\bar \eta}_L^{-1}(D)\text{ for some $\gamma\in \G - \Gamma_L$}\} .$$
Then
\bi
\item[(1)] $\mathcal S(D)\subset \mathcal S(L,W)$;
\item[(2)] $\pi (\mathcal S(D))$ is closed in $X$;
\item[(3)] for any compact subset $B\subset X \setminus \pi(\mathcal S(D))$,
there exists a neighborhood $\Phi$ of $D$ in $\bar V_L$ such that
for each $y\in \pi({\bar \eta_L}^{-1}(\Phi))\cap B$, the set
$\bar \eta_L(\pi^{-1}(y))\cap \Phi$ consists of a single element.
\ei
\end{Prop}

%\change{$S(p_0)$-small added}

\begin{proof}
Suppose that $g \in \mathcal S(D)$. Then $\gamma g\in
{\bar \eta}_L^{-1}(D)$ for some $\gamma\in \G -
\Gamma_L$.
By (\ref{ah}), both $g$ and $\gamma g$ belong to $ \mathcal N (L,W)$.
Then
$$\overline{\G \ba \G g W}\subset \G\ba \G L \gamma g .$$

Suppose $g\notin \mathcal S(L, W)$. Then
by Lemma \ref{ns},
$$\overline{\G \ba \G g W}=\G\ba \G L g .$$
Hence by Lemma \ref{zdc},
$$\MT(L)\subset \gamma^{-1}\MT(L) \gamma .$$
Therefore, $\gamma\in \Gamma_L$, which gives a contradiction. This shows (1).

If (3) fails, then there
 exist $ g_i \in \pi^{-1}(B)$ and
$\gamma_i \in \Gamma$
with $\bar \eta_L(g_i)\ne
 \bar\eta_L(\gamma_i g_i)$ and
$\bar \eta_L(g_i),
 \bar\eta_L(\gamma_i g_i)$ converge to elements of $D$.
Since the map (\ref{eq:proper}) is proper,
by passing to a subsequence,
there exist $\delta_i, \delta_i'\in \G_L^0$ such that
$\delta_i g_i \to g$ and
$\delta_i' \gamma_ig_i\to g'$ for some $g, g'\in G$.
Hence by passing to a subsequence,
 $\delta_i'\gamma_i \delta_i^{-1}=g'g^{-1}$
for all large $i$.
Hence $\delta_0:=g'g^{-1}\in \G$.
Then
$\bar\eta_L(\G_L^0  g), \bar\eta_L(\G_L^0 \delta_0 g)\in D$.
Since  $\G_L^0  g\notin \mathcal S(D)$,
$\delta_0\in \Gamma_L$.
Hence $$
\bar\eta_L(g)\in D\cap \alpha D$$
for some $\alpha\in \Cal O^\times_S$.
 Since $D$ is $S(v_0)$-small,
it follows from (\ref{eq:small}) that $\alpha\in \mu_K$  and
hence $\bar\eta_L(\gamma_i)=\bar\eta_L(\delta_i)$.
This gives a contradiction.

Claim (2) can be proved similarly.

\end{proof}

By an interval $I$ in $K_v$ centered at $x_0\in K_v$,
we mean a subset of the form
$\{x\in K_v: |x-x_0|_v<T\}$
for some $T>0$.

We will need the following property of
polynomial maps in the proof of our main proposition \ref{3.4}.

\begin{Prop}\cite[4.2]{T}\label{cp}
Let $A_v$ be a Zariski closed subset of $K_v^m$,
 $C_v\subset A_v$  a compact subset, and $\e>0$.
Then there exists a compact neighborhood $D_v\subset A_v$
of $C_v$ such that for any neighborhood $\Phi_v$ of $D_v$ in $K_v^m$
there exists a neighborhood $\Psi_v\subset \Phi_v$ of $C_v$
 such that for any one parameter unipotent subgroup
$u(t)$ of $\bG(K_v)$, any bounded interval $I$ in $K_v$
 and any $w\in K_v^m$ such that $wu(t_0)\notin \Phi_v$ for some $t_0\in I$,
$$\theta_v(\{t\in I: wu(t)\in \Psi_v\})\leq \e \cdot
\theta_v(\{t\in I: wu(t)\in \Phi_v\}).
$$

\end{Prop}

We will also use the following simple lemma from \cite{T} to relate the behavior of unipotent one parameter subgroups
over $\c$ with those over $\br$.
\begin{Lem}\label{4.4} Let $K_v=\c$,
$I=\{t\in \c: |t|\le 1\}$, $\e>0$ and $A$ measurable subset of $I$ such that
for any $x\in I$.
$$\e \theta_0\{a\in \br: ax\in I\}\ge \theta_0\{a\in \br: ax\in I\cap A\} .$$
Then $\theta_v (A)\le \e \pi$
where $\theta_0$ is the Lebesgue  measure on $\br$.
\end{Lem}

The following proposition is a main tool in the proof of Theorem
\ref{ms}.

\begin{Prop}\label{3.4} Fix $v_0\in S$.
 Let $C\subset A_L$ be a compact subset and $\e>0$ be given.
Then there exists a closed subset $R$ of $X$
contained in $\pi(\mathcal S(L,W))$ such that for any compact subset
$B\subset X\setminus
R$, there exists a neighborhood $\Psi$ of $C$ in $V_L$
such that
for any one parameter unipotent subgroup
$u(t)$ of $\bG(K_{v_0})$ and any $x\in B$,
at least one of the following holds:
\begin{enumerate}
\item There exists $w\in \bar{\eta}_L(\pi^{-1}(x))\cap \overline{\Psi}$ such that
$$\{u(t)\}\subset \{g\in G: wg=w\}$$
\item For any sufficiently large bounded interval $I\subset K_{v_0}$
  centered at zero,
$$\theta_{v_0}(\{t\in I: xu(t) \in B\cap \pi(\bar\eta_L^{-1}(\Psi))\})
\leq \e \cdot \theta_{v_0}(I).$$
\end{enumerate}
\end{Prop}

%\change{S is changed to R}

\begin{proof}
Since $C$ can be covered by finitely many compact $S(v_0)$-small sets,
it suffices to prove the proposition for a $S(v_0)$-small subset
 $C=\prod_{v\in
  S} C_v$ with $C_v$ compact.
For $C_{v_0}$ and $\e >0$,
let $D_{v_0}$ be as in Proposition \ref{3.4} and $D_v=C_v$ for
 $v\in S\ba \{v_0\}$.
Then  the set $D:=\prod_{v\in S} D_v$ is also $S(v_0)$-small.
For $S(D)$ defined in Proposition \ref{sd},
set $R=\pi(S(D))$.
For a given $B$, let $\Phi$ be a neighborhood of $D$
as in Proposition \ref{sd}. Passing to a smaller neighborhood, we may assume
that $\Phi$ is of the form $\prod_{v\in S} \Phi_v$.
Let $\Psi_{v_0}\subset \Phi_{v_0}$ be a neighborhood of $C_{v_0}$
 so that the statement
of Proposition \ref{cp} holds.
We set $\Psi:=\prod_v \Psi_v$ where $\Psi_v=\Phi_v$ for $v\neq v_0$.

Let $\Omega:=B\cap \pi(\bar \eta_L^{-1}(\Psi) )$ and
 $J=\{t\in K_{v_0}: x u(t)\in {\Omega}\}$.

Assume that $v_0$ is non-archimedean.
For each $t\in J$, there exists a unique $w_t\in \eta_L(\pi^{-1}(x))$
such that $w_t u(t)\in \Phi$. By uniqueness, $w_t u(t)\in {\Psi}$.
Note that the map $t\mapsto w_t$ is a locally
constant. For each $t\in J$,
let $I(t)$ be the maximal interval containing $t$
such that $w_t u(I(t))\subset \Phi$.
By the nonarchimedean property of $K_{v_0}$,
the intervals $I(t)$ are either disjoint or equal.
Since $s\mapsto w_t u(s)$, $s\in I(t)$, is a polynomial map, it is either
constant or unbounded. Hence if some $I(t)$ is unbounded for $t\in J$,
$w_t u(K_{v_0})=w_t$ and hence the first case happens.
Now suppose that $I(t)$ is bounded for any $t\in J$.
Let $J(t)$ be the minimal interval containing
$I(t)$ such that $w_t u(J(t))\cap \Phi ^c\neq \emptyset$.
Note that $\theta_{v_0}(J(t))\leq q_0\cdot \theta_{v_0}(I(t))$
where $q_0$ is the cardinality of the residue field of $K_{v_0}$.
By Proposition \ref{cp},
\begin{align*}
\theta_{v_0}(\{s\in I(t):w_t  u(s)\in \Psi\})
&\leq\theta_{v_0}(\{s\in J(t): w_t u(s)\in \Psi\}) \\
&\leq \epsilon \theta_{v_0} (J(t))\leq \epsilon\cdot q_0 \cdot \theta_{v_0} (I(t)).
\end{align*}
Now for any interval $I$ centered at zero, we have
$$
\{s\in I: xu(s)\in \Omega \}= \bigcup_{t\in J} I(t)\cap I.
$$
If $I$ is sufficiently large, it follows from the nonarchimedean property of $v_0$ that
either $I\cap I(t)=\emptyset$ or $I(t)\subset I$. Hence
\begin{align*}
\theta_{v_0}(\{s\in I: xu(s)\in \Omega \})&= \sum_{I(t)\subset I} \theta_{v_0}(\{s\in I(t): xu(s)\in \Omega \})
\\  &\leq \epsilon\cdot q_0 \cdot  \sum_{I(t)\subset I}
\theta_{v_0} (I(t))\leq \epsilon\cdot q_0 \cdot  \theta_{v_0}(I) .\end{align*}
This proves the claim for $v_0$ non-archimedean.
The case when $K_{v_0}=\br$ is proved in \cite{MS}.
Consider the case of $K_{v_0}=\mathbb C$. By the restriction of scalars, we may consider
$\bG(K_{v_0})$ as a real Lie group and hence the statement holds for
any restriction $u_r$ of $u:\mathbb C \to \bG(K_{v_0})$ to a one dimensional
real subspace $r$ of $\mathbb C$.
Suppose (1) holds for some real subspace $r$, i.e., $u_r(\br)$
stabilizes a vector $w=p_L (\gamma g)$ with $\pi(g)=x$.
Then
 $$g u(r)g^{-1} \subset \gamma^{-1} \{y\in  N (L): \Ad(y)|_{\mathfrak l}=1
\} \gamma  .$$
Since the right hand side of the above is a $K$-subgroup,
it follows that $g u(K_{v_0})g^{-1}$ is also contained
in the same group and hence $u$ satisfies (1).
Therefore if (1) fails for $u$,
then (2) holds for $u_r$ for any one dimensional real
subspace $r\subset \mathbb C$ and for any interval of $r$.
By \eqref{4.4}, this implies (2).

\end{proof}

%\change{$\bar\Omega$ is replaced with $\Omega$ (I think it is only
%  important for archemidean case.)}

%\change{Interval I is centered at zero. Otherwise it might happen
%that I drifts to infinity and $I<I(t)$.}

\subsection{Proof of Theorem \ref{ms}}

%\begin{Lem} Let $L$ be the subgroup of $\bG(K_{w})$
%generated by $U_i$, $i\in \n$ and denote by $\mathbf L$
%the Zariski closure of $L$ in $\bG$.
%Then $$\bG(K_{w})\cap \mathbf L\subset L_u .$$
%\end{Lem}

%\begin{proof}[Proof of Theorem \ref{ms}]
Set $W:=\Lambda(\mu)_u$ and $U_i=\{u_i(t): t\in K_v\}$.
Then $\text{dim}(W)\ge 1$ by \cite[Lemma 2.2]{MS}  whose proof works
in the same way for $K_{v}$.
 By Theorem \ref{hstar}, there exists $L\in \mathcal H$
such that
$$\mu(\pi(\mathcal S(L,W))=0\quad\text{and}\quad \mu(\pi(\mathcal N (L,W))>0 .$$
%Since $L$ has finite index in $\MT(L)_S$, we may
%assume without loss of generality $L$ is a normal subgroup of
%$\MT(L)_S$ by replacing it with its finite index subgroup if necessary.
Therefore we can find
 a compact set $C_1\subset \mathcal N (L,W)\ba \mathcal S(L,W)$
such that $\mu(\pi(C_1))>0$.
Note that by (\ref{thw}), $\pi(C_1)\cap\pi(\mathcal S(L,W))=\emptyset$.
Let  $z_i\to e \in G$ be a sequence such that
$xz_i\in \text{supp}(\mu_i)$ and the trajectory
$\{xz_i u_i(t): t\in K_v\}$ is uniformly distributed with
respect to $\mu_i$ as $T\to \infty$ when the averages are taken over
the sets $I_T:=\{s\in K_{v}: |s|_{v}\leq T\}$.

By the pointwise ergodic theorem, such a sequence $\{z_i\}$
always exists.

Pick $y\in \text{supp}(\mu)\cap \pi(C_1)$.
Then there exists a sequence $y_i\in xz_i U_i$
which converges to $y$.
Let $h_i\to e$ be a sequence satisfying $y_i=yh_i$ for each $i$.
Set $$\mu_i'=\mu_i h_i\quad\text{ and } \quad
u_i'(t)=h_i u_i(t) h_i^{-1} .$$
Then $\mu_i'\to \mu$ as $i\to \infty$, $y\in\text{supp}(\mu_i')$
and $\{yu_i'(t)\}$ is uniformly distributed with respect to $\mu_i'$.

Let $R$ and $\Psi$ be as Proposition \ref{3.4} with
 respect to $C:=\bar \eta_L(C_1)$ and $\e:=\mu(\pi(C_1))/2$.
We can choose a compact neighborhood  $B$ of $\pi(C_1)$ such that
$B\cap R=\emptyset$.
Put $$\Omega:=\pi(\bar{\eta}_L^{-1}(\Psi))\cap B .$$
Since $\pi(C_1)\subset \Omega$,
we have $\mu_i'(\Omega)>\e$ for all sufficiently large $i$.
Hence for sufficiently large $T$ and $i$,
$$\theta_{v_0} (\{s\in I_T: yu_i'(t)\in \Omega\})>\e\cdot \theta_{v_0}(I_T)$$
By Proposition \ref{3.4},
there exists $g_0\in\pi^{-1}(y)$
so that $w=p_L( g_0) \in
\bar\eta_L(\pi^{-1}(y))\cap\overline{\Psi}$
and
$w u_i'(t)=w$ for all $t\in K_v$.
Consider the $K$-subgroup $\bM:=\text{Stab}_\bG(wg_0^{-1})$.
We observe that
\begin{equation}\label{em1}
g_0\{u_i'(t)\}g_0^{-1}\subset \bM_S\quad\hbox{and}\quad
\text{supp}(\mu_i')g_0^{-1}\subset \Gamma\ba\G \bM_S.
\end{equation}

We use induction on $\dim(\bG)$ to show that
 $$\text{(a)}\quad \text{supp}(\mu)=y\Lambda(\mu)\quad
 \quad\text{and}\quad\quad \text{(b)}\quad \Lambda (\mu_i ')\subset \Lambda(\mu) $$
for all sufficiently large $i$.

If $\dim(\bM)< \dim(\bG)$, since \eqref{em1} and $g_0^{-1}\mu_i'\to
g_0^{-1}\mu$, we can
 apply inductive hypothesis to the space  $\Gamma\ba\G
\bM_S$ and the measure $g_0^{-1}\mu$. This yields (a) and (b).

If $\dim(\bM)= \dim(\bG)$, then $\MT(L)$ is a
normal subgroup of $\bG$.

Since
$\mathcal N (L,W)=G$ and $\mu(\pi(\mathcal S(L,W)))=0$, we have
$\mu=\mu_L$.
By Theorem \ref{hstar},
 every $W$-ergodic component of $\mu$ is
$g^{-1}Lg$-invariant for some
$g\in G$.
Since $\MT(L)$ is a normal subgroup of $\bG$,
 $g^{-1}Lg$ is a subgroup of
$\MT(L)_S$ and
$[\MT(L)_S: L]=[\MT(L)_S: g^{-1}L g]$.
Since $\MT(L)_S$ has only finitely many closed subgroups
of bounded index, we obtain a finite index subgroup $L_0$ of $\MT(L)_S$
 such that $L_0$ is normal in $G$ and every $W$-ergodic
component of $\mu$, and hence $\mu$ itself, is $L_0$-invariant.

Denoting by $\rho: G \to L_0\ba G$ the quotient homomorphism,
 we set $\bar X=\rho(\Gamma)\ba (L_0\ba G)$
and obtain the push-forward map
$\bar \rho_*:\mathcal P(X) \to \mathcal P(\bar X)$ of measures.

Since $\text{dim}(\MT(L))\ge \dim W >1$,
we may  apply the induction to the measures $\bar\rho_*(\mu_i')$
and $\bar\rho_*(\mu)$
and obtain
$$\text{supp}(\bar\rho_*(\mu))=\bar y \Lambda(\bar\rho_*(\mu))$$
and for all large $i$,
$$\Lambda(\bar\rho_*(\mu_i))\subset \Lambda (\bar\rho_*(\mu)).$$

Since $\mu$ is $L_0$-invariant,
applying \cite{D} in the same way as in \cite{MS},
this implies $$\rho^{-1}(\Lambda (\bar \rho_*(\mu)))=\Lambda(\mu). $$
It is easy to deduce (a) and (b) now.
%Now we can apply inductive hypothesis to the group
%$\bG/\MT(L)$ to show (a) and (b), by repeating
%the same argument in the proof of Theorem 1.1 [MS].
%\change{ADD DETAILS}

We finally claim that (a) and (b) imply (1)--(3).
Since $\mu_i'$ are $\{u_i'(t)\}$-ergodic measures and $y\in \text{supp}(\mu')$,
by Theorem \ref{tom}, $\mu_i'$ is
a $\Lambda(\mu_i')$-invariant measure supported on $y\Lambda(\mu_i')$.
Hence, by (b),
$$
\text{supp}(\mu_i)=\text{supp}(\mu_i')h_i=y\Lambda(\mu_i')h_i
\subset y\Lambda(\mu)h_i =x\Lambda(\mu)h_i.$$
Since $$xz_i\in \text{supp}(\mu_i)\subset x\Lambda(\mu)h_i,$$
and $z_i, h_i\to e$,
it follows that $z_ih_i^{-1}\in \Lambda(\mu)$.
Therefore,
$$
\text{supp}(\mu_i)\subset x\Lambda(\mu)h_i=x\Lambda(\mu) z_i
\quad\hbox{and}\quad
\Lambda(\mu_i)=h_i^{-1}\Lambda(\mu_i')h_i\subset
z_i^{-1}\Lambda(\mu)z_i.
$$
This proves (2).

There exists $i_0$ such that for all $i\geq i_0$,
$z_i U_iz_i^{-1}\subset \Lambda(\mu)$.
Let $H$ be the subgroup of $G$ (in fact
of $\bG(K_{v})$)
generated by all $z_iU_i z_i^{-1}$,
$i\geq i_0$. By (2), $H\subset \Lambda(\mu)$
and hence by (a)
$$\overline{xH}\subset x\Lambda(\mu)=\hbox{supp}(\mu).$$
On the other hand, by Theorem \ref{tom},
$$
\overline{xH}=\G\ba\G L g
$$
for some $L\in\mathcal{H}$ such that $H\subset g^{-1}Lg$. Since $\mu_i\to\mu$,
it follows that $\G\ba\G L g=\hbox{supp}(\mu)$.

Since $\G \ba \G L g=x\Lambda(\mu)$,
$\mu$ is the unique invariant probability measure
supported on $\G\ba \G L g$, as required in (1).
Since $L\in \mathcal H$ and $H\subset g^{-1}L_u g$, by
the following proposition \ref{mauc}, $\mu$ is ergodic with respect to $H$.
This finishes the proof.

\begin{Prop}\label{mauc} \label{Mautner2}
Let $L$ be a closed subgroup of finite index
in $\bP_S$ for some
$\bP\in\mathcal{F}$, and let $H$ be a closed subgroup of $L$ generated by
unipotent one-parameter groups such that $\overline{\G\ba \G H}=\G \ba \G L$.
Then the translation action of $H$ on $\G \ba \G L$
is ergodic.
\end{Prop}
\begin{proof}
 By an $S$-algebraic version of the Mautner lemma (see below Prop. \ref{Mautner})
there exists a closed normal subgroup
$M\subset \bP_S$ containing $H$ such that the triple $(H, M, \bP_S)$ has the Mautner
property, that is, for any
continuous unitary representation of $\bP_S$, any $H$-invariant vector is also $M$-invariant.
Since $M\cap L$ is normal in $L$ and
$\overline{\G\ba \G (M\cap L)}=\G \ba \G L$, it follows
that $M\cap L$ acts ergodically on $\G \ba \G L$.
Applying the
Mautner property to the unitary representation $\hbox{Ind}^{\bP_S}_L
 L^2(\G \ba \G L)$, we deduce that $H$ acts ergodically on $\G \ba \G L$.
\end{proof}

We recall an $S$-arithmetic version of the Mautner lemma:
%Let $P_S$ be the product of
%$\qp$-points of connected $\qp$-algebraic subgroups of $\bG$ for $p\in S$.
\begin{Prop}\cite[Corollary 2.8]{MT2} \label{Mautner}
Let $L\subset \bG_S$ be a closed subgroup
generated by unipotent one-parameter subgroups in it.
Then there exists a closed normal subgroup
$M\subset \bG_S$ containing $L$ such that the triple $(L, M, \bG_S)$ has the Mautner
property, that is, for any
continuous unitary representation of $\bG_S$, any $L$-invariant vector is also $M$-invariant.
\end{Prop}

%\change{P=G so I removed P in the above???}

%\change{Some details in the proof are added}

%\change{A general comment is that this section is very similar to
%  [MS]. I have to emphisize more what are the new ideas}

%\change{Should we skip corollaries below? It identical to [MS].--
%this is S-arithmetic analogue of their theorem..I would like to keep???}

\section{Non-divergence of unipotent flows}\label{sec:dm}
\subsection{Statement of Main theorem}
Let $\bG$ be a connected semisimple algebraic $K$-group,
$S$ a finite set of normalized absolute values of $K$
 including all the archimedean $v$ such that $\bG(K_v)$
is non-compact
and $\G\subset G$ an $S$-arithmetic lattice.
Here we also assume that $\bG$ is $K$-isotropic, equivalently, that
$\G$ is a non-uniform lattice (otherwise, the main theorem of
this section holds trivially). Note this also implies
that $\bG(K_v)$ is non-compact for every valuation $v$ of $R$.
We generalize the main theorem of
Dani-Margulis in \cite{DM1} to an $S$-algebraic setting.
Some of our arguments follow closely those in \cite{EMS1}.

Let $\mathbf A$ be a maximal $K$-split torus of $\bG$ and
choose a system $\{\alpha_1, \ldots, \alpha_r\}$ of
simple $K$-roots for $(\bG, \mathbf A)$.
For each $i$, let $\mathbf P_i$ be the standard maximal parabolic subgroup
corresponding to $\{\alpha_1, \ldots, \alpha_r\}-\{\alpha_i\}$.

The subgroup $\mathbf P:=\cap_{1\le i\le r} \bP_i$ is a minimal $K$-parabolic
subgroup of $\bG$, and there exists a finite subset $F\subset \mathbf G(K)$ such that
$$\bG(K)=\Gamma  F \mathbf P(K) .$$

For $T>1$, we set
$$J_{T}:=\{x\in K_v: |x|_v<T\} .$$
\begin{Thm}\label{dmma} Let $\e>0$.
Then there exists a compact subset $C\subset \G\ba \bG_S$
such that for any unipotent one parameter subgroup $U=\{u(t)\}\subset \bG(K_v)$, and $g\in \bG_S$,
either one of the following holds:
\be
\item for all
 large $T>0$,
$$\theta_v\{t\in J_T: \G\ba \G g u_t\in C\}\ge (1-\e) \,\theta_v (J_T);$$

\item there exist $i$ and $\lambda\in \Gamma F$ such that
$$g U g^{-1}\subset \lambda \mathbf P_i \lambda^{-1} .$$
\ee
\end{Thm}

\subsection{Deduction of Theorem \ref{rc} (1) from Theorem \ref{dmma}:}\label{ded:ssm2}
Suppose not.
Let $C$ be a compact subset as in Theorem \ref{dmma}.
Then there exists $\e>0$
such that $$g_i\nu_i(C)<1-\e\quad\text{for all large $i$, }$$by passing to a subsequence.
Fix $v\in S$ that is strongly isotropic for all $\bL_i$.
Let $U_i=\{u_i(t)\}\subset \bL_i(K_v)^+$  be as in Lemma \ref{l:ergodic} and
let $R_i$ denote a subset of full
measure in $\pi(\tilde \bL_{i,S})$ such that for every $h\in R_i$,
the orbit $\G\ba \G   h U_i $ is
uniformly distributed on $\Gamma\ba \Gamma  \tilde \bL_{i,S}  $.
Hence for each $i$, there exists $T_i$ such that
$$\theta_v\{t\in J_T: \G\ba \G  h u_i(t) g_i\in  C\} \le (1-\e/2 )\theta_v(J_T)$$
for all $T>T_i$.

Applying $U=g_i^{-1} U_i g_i$ and $g=hg_i$ to
 Theorem \ref{dmma}, there exist $j_i$ and $\lambda_i\in \Gamma F$
such that
$$hU_ih^{-1}\subset \lambda_i \bP_{j_i}\lambda_i^{-1}$$
for all $h\in R_i$, where $\bP_{j_i}$ is a proper parabolic $K$-subgroup of $\bG$.

Since the set $\{h\in \pi(\tilde \bL_{i,S}): hU_ih^{-1}\subset
\lambda_i \bP_{j_i}\lambda_i^{-1}\}$ is a product of analytic
manifolds over local fields which contains a subset of full measure
in $\pi(\tilde \bL_{i,S})$, it follows that this set is $\pi(\tilde
\bL_{i,S})$ itself. Since $U_i$ is not contained in any proper
normal subgroup of $\bL_i (K_v)^+$, we have $$\bL_i \subset
\lambda_i \bP_{j_i}\lambda_i^{-1} .$$ This is a contradiction to the
assumption by Lemma \ref{aniso}.

\subsection{}

For each $i$, let $\mathbf U_i$ denote
 the unipotent radical of $\mathbf P_i$.
Denote by $\mathfrak u_i$ the Lie algebra of $\mathbf U_i$
and by $\mathfrak g$ the Lie algebra of $\bG$.
For each $v\in S$, we fix a norm $\|\cdot \|_v$ on the $K_v$-vector space
$\land^{{\dim \mathfrak{u}_i}} \mathfrak g(K_v)$
and choose a non-zero vector $w_i$ of $ \land^{\dim \mathfrak{u}_i}
\mathfrak u_i (K)$ with $\|w_i\|_v=1$ for all $v\in S$.
Define $\Delta_i: \bGS \to \br^*$ by
$$\Delta_i((g_v)):=\prod_{v\in S}\|w_i g_v\|_v .$$

%We refer to [DM] for some of the basic facts
%we are using in this section without reference.

%\todo{check if this F works??}

Fix $v\in S$.   Let $\Cal P_d$ denote the family of
all polynomial maps $K_v\to \bG(K_v)$ (resp. $\br \to \bG(\c)$) of
degree at most $d$ if $K_v\ne \c$ (resp. $K_v=\c$).

For $T>1$ we set if $K_v \ne \c$
$$I_{T}:=\{x\in K_v: |x|_v<T\} ,$$
and if $K_v=\c$,
$$I_{T}:=\{x\in \br: |x|<T\} $$
where $|\cdot |$ is the usual absolute value of a real number.
We keep this definition of $I_T$ for the rest of this section.
We will deduce Theorem \ref{dmma} from the following:
\begin{Thm}\label{nd} Fix $\alpha,\e>0$ and $v\in S$.
Then there exists a compact subset $C\subset \G\ba\bG_S$
such that for any $u\in \mathcal{P}_d$ and any $T>0$,
either one of the following holds:
\be
\item $\theta_v (\{s\in I_T: \Gamma \ba \Gamma u(s)\in C\}) \geq (1-\e) \theta_v(I_T)$

\item there exist $i\in \{1, \ldots, r\}$ and $\lambda\in \Gamma F$
such that $$\Delta_i(\lambda^{-1} u(s))\leq \alpha \quad\hbox{for all $s\in I_T$,}$$

\ee
\end{Thm}

\noindent{\bf Deduction of Theorem \ref{dmma} from Theorem \ref{nd}}
First consider the case of $K_v\ne \c$.
There is $d>0$ such that for any $g\in \bG_S$ and for any $u$ one parameter
unipotent subgroup of $\bG(K_v)$,
the maps $t\mapsto g u(t)$ belongs to $\mathcal P_d$.
Hence if the first case of Theorem \ref{dmma}
fails, then there exists $1\le i\le r$, $T_m \to \infty$,
$0<\alpha_m<1$, $\alpha_m \to 0$, $\lambda_m\in \Gamma F$ such that
$$\Delta_i(\lambda^{-1}_m g u(s))<\alpha_m$$
for all $s\in I_{T_m}$.

Since this implies
$\Delta_i(\lambda_m^{-1}g)< 1$
and for a given
$\theta>0$,
the number
of the elements $\lambda\in \Gamma F$, modulo the stabilizer
of $w_i$,
such that  $\Delta_i(\lambda^{-1}g)<\theta$
is finite, we can assume, by passing to a subsequence,
that there exists $\lambda\in \G F$ such that for each $m$,
$$\Delta_i(\lambda^{-1} g u(s))<\alpha_m$$
for all $s\in I_{T_m}$.

Since any orbit of a unipotent one parameter subgroup is unbounded
except for a fixed point,
it follows that $\lambda^{-1} g u(s)g^{-1} \lambda $ fixes $w_i$ for all $s\in K_v$.
Therefore $$\lambda^{-1} g U g^{-1} \lambda \subset \mathbf P_i .$$

Now consider the case when $K_v=\c$,
Suppose (1) fails for some $g\in \bG_S$ and for some
unipotent one parameter subgroup $u:\c \to \bG(\c)$.
By \eqref{4.4}, we have
a
one dimensional
real subspace $r=\br x \subset \c$, $x\in \c$, such that
$$|\{s\in [-T_m, T_m]: \G\ba \G g u(sx)\in C\}|< 2(1-\e) T_m$$
for some $T_m\to \infty$.
By Theorem \ref{nd},
for any $\alpha>0$,
there are
 $i$ and $\lambda\in  \G F$ such that
for all $s\in I_{T_m}$,
$$\Delta_i(\lambda^{-1} u_r(s))\le \alpha$$
 where $u_r(s)=u(sx)$.

By the same argument as in the above case,
we deduce that
 for some $1\le i\le r$ and $\lambda\in \Gamma F$,
we have
$$g u(r)g^{-1}\subset \lambda \mathbf P_i \lambda^{-1}.$$
Since $\mathbf P_i$ is an algebraic
$K$-subgroup, it follows that
$$g U g^{-1}\subset \lambda \mathbf P_i \lambda^{-1}.$$
This finishes the proof.

In order to prove Theorem \ref{nd}, we use the following:
\begin{Thm} \label{dm2}
Let $\alpha>0$ be given.
There exists a compact subset
$C\subset \G\ba \bG_S$ such that for any $u\in\mathcal{P}_d$ and $T>0$, one of the following holds:
\be
 \item there exist $i\in \{1, \ldots, r\}$ and $\lambda\in \Gamma F$
such that $$
\Delta_i(\lambda^{-1} u(s)) \leq \alpha \quad\hbox{for all $s\in I_T$},$$
\item $\Gamma \ba \Gamma u(I_T)\cap C\neq \emptyset$.
\ee
\end{Thm}

Theorem \ref{dm2} implies Theorem \ref{nd} in view of the following theorem,
proved by Kleinbock and Tomanov \cite[Theorem 9.1]{KT}:
\begin{Thm} \label{kt} For a given compact subset
$C\subset \G\ba\bG_S$ and $\e>0$
there exists a compact subset $C'\subset \G\ba\bG_S$
such that for any $u\in \mathcal{P}_d$, any $y\in \G\ba \bG_S$ and $T>0$ such that
$yu(I_T)\cap C\neq \emptyset$,
$$\theta_v(\{s\in I_T: yu(s)\in C'\})\geq (1-\e) \theta_v(I_T).$$
\end{Thm}

The rest of this section is devoted to a proof of
Theorem \ref{dm2}. We start by constructing
certain compact subsets in $X$ which will serve as $C$
in the theorem \ref{dm2}.

We denote by $S_\infty$ the set of all archimedean absolute values
and $S_f:=S-S_\infty$.
For a $K$-subgroup $\mathbf M$ of $\mathbf G$,
and $S_0\subset S$,
we use the notation
 $\mathbf M_{S_0}=\prod_{v\in S_0}
 \mathbf M(K_v)$, $\mathbf M_\infty=\mathbf M_{S_\infty}$,
and $\mathbf M_v=\mathbf M(K_v)$.
For simplicity, we write $M$ for
$\mathbf M_{S}$ in this section.

We often write an element of $g\in M$
as $(g_\infty, g_f)$
where $g_\infty\in \bM_{\infty }$
and $g_f\in  \bM_{S_f}$.

\subsection{Description of compact subsets in $X$}
For each $i=1, \ldots, r$,
we set
$$\mathbf Q_i=\{x\in \mathbf P_i: \alpha_i(x)=1\}$$
and
$$\mathbf A_i:=\{x\in \mathbf A: \alpha_j(x)=1\quad \forall j\neq i\} .$$
For a subset $I\subset\{1,\ldots, r\}$,
we define
$$\mathbf P_I:=\cap_{i\in I} \mathbf P_i,
\;\;\mathbf Q_I:=\cap_{i\in I} \mathbf Q_i,\;\; \mathbf A_I:=\prod_{i\in I}\mathbf A_i .$$
Let $\mathbf U_I$ be the unipotent radical of $\mathbf P_I$ and
$\mathbf H_I$ the centralizer of $\mathbf A_I$ in $\mathbf Q_I$.
We have Langlands decomposition:
$$
\bP_I=\mathbf A_I \bQ_I=  \mathbf A_I \bH_I {\bf U}_I.
$$

There is $m_i\in \n$ such that for $x\in \mathbf P_i$,
$$\operatorname{det}(\Ad x)|_{\mathfrak u_i}=\alpha_i^{m_i}(x).$$

For each non-archimedean $v\in S$,
we set $$ A_v^0 =\{x\in \mathbf A_v: \alpha_i(x)
\in q_v^{\z}\quad\forall i=1, \ldots, r\}$$
where $q_v$ is
the cardinality of the residue field of $K_v$.
Since $\mathbf A$ is $K$-split,
$A_v^0 \subset \mathbf A(K)$.

For archimedean $v$, we set
$$A_v^0 =\{x\in \mathbf A_v: \alpha_i(x) > 0
\quad\forall i=1, \ldots, r\} .$$
For $v\in S$, let $W_v$ be a maximal compact
subgroup of $\bG_v$
such that $\bG_v=W_v \bQ_v$
for any parabolic $K$-subgroup $\bQ$ containing $\bP$ and
 $\mathbf A_v\subset W_v A_v^0$.

We set $W=\prod_{v\in S}W_v$,
 $W_f=\prod_{v\in S_f}W_v$ and $W_\infty=\prod_{v\in S_\infty}W_v$.
Without loss of generality, we may assume that each norm
$\|\cdot\|_v$ is $W_v$-invariant.

For a subset $I\subset\{1,
\ldots, r\}$,
we set
$$A_{I, v}^0 :=\mathbf A_{I,v}\cap A_v^0
\quad\text{and}\quad  A_{I,\infty}^0=\prod_{v\in S_\infty} A_{I, v}^0 .$$

%\change{A Lemma about $\Delta_i$ is moved to later (it does fit well
%  with the rest of the material).}
%For a non-archimedean $v\in S$,
%we use the notation $\Cal O[v]$ to denote the ring of
%$v$-integers in $K$,
%that is, $\Cal O[v]$ consists of $x\in K$
%such that $x\in \Cal O_w$ for all non-archimedean valutation
%$w\ne v$ of $K$.

\begin{Lem}\label{sp} There exists a finite subset
$Y\subset \mathbf A_I(K)$
such that
$$
\prod_{v\in S_f} A_{I,v}^0 \subset (\mathbf A_I \cap\G)
 Y.$$
\end{Lem}

%\change{Notation for Y changed; more precise formulation (I think it
%  is actually used in the proof later.}

\begin{proof}
Since $\G$ is commensurable with $\bG(\mathcal O_S)$,
$\G$ contains a finite index subgroup of $\mathbf A_I(\mathcal O_S)$.
Now the claim follows easily from the fact that
the map $f: \mathbf A_I\to \bG_m^{l}$, $l=|I|$, given by
$x\mapsto (\alpha_1(x), \ldots, \alpha_{l}(x))$,
is a $K$-rational isomorphism, where $\bG_m$
denotes the one-dimensional multiplicative group.
%Since $\bG_m(\mathcal [v])=\{\pm q_v^{\z}\}$,
%a finite index subgroup of $\mathbf A(v)$ is contained in
%$\mathbf A_I(\mathcal O [v])$.
%Hence
%for $v\in S_f$,
% $A_I(v)\subset (\mathbf A_I(\mathcal O[v])\cap\G) Y_v$ for some finite subset $Y_v\subset \mathbf A_I(K)$.
%It suffices to set $Y=\cup_{v\in S_f} Y_v$.
\end{proof}

\begin{Lem}\label{dm1.5}
Given $I\subset \{1, \ldots, r\}$, $j\in \{1, \ldots, r\}-I$,
and $0< a\leq b$, there exists a compact subset $M_0$ of $Q_{I}$
such that
$$
\{g\in
Q_{I}:\, \Delta_j(g)\in [a, b]\}\subset (\mathbf A_j\cap \G)  Q_{I\cup\{j\}}M_0.$$
\end{Lem}

\begin{proof}
Since $\mathbf A_{I,v}\subset W_vA_{I,v}^0$ for each $v\in S$,
we can show in a similar way as in the proof of
\cite[Lemma 1.5]{DM1} that for any $j\notin I$,
$$
Q_I= (\prod_{v\in S} A_{j,v}^0 )Q_{I\cup\{j\}} (W\cap  H_{I}).
$$
Hence any $(g_v)\in Q_I$ is
of the form $g_v=a_v q_v w_v$
with $a_v\in A_{j,v}^0 $, $q_v\in \mathbf Q_{I\cup\{j\},v}$, and $w_v\in W_v$,
and $$\|w_j g_v\|_v=|\alpha_j(a_v)|_v^{m_j}.$$

Suppose
 $g=(g_v)\in Q_I$ satisfies
$a<\Delta_j(g)<b$, i.e.,
$$\Delta_j(g)=\prod_{v\in S}|\alpha_{j}(a_v)|_v^{m_j}\in [a, b].$$

It follows from Lemma \ref{sp} that $d_0:=\prod_{v\in S_f} a_v\in
(\mathbf A_j\cap \G) Y$ where $Y$ is a finite subset of $\mathbf A_j(K)$.
If we set $d_v=a_v d_0^{-1}$ for $v\in S_\infty$
 and $d_v=\prod_{w\in S_f\ba\{v\}}
a_w^{-1}$ for $v\in S_f$,
then $d_0 d_v= a_v$ and
 $d_v\in W_v$  for $v\in S_f$, and $$
\prod_{v\in S_\infty}|\alpha_j(d_v)|_v^{m_j}
=\prod_{v\in S}|\alpha_{j}(a_v)|_v^{m_j}\in [a, b]
$$
This implies that there exists a compact set $M_\infty\subset
\mathbf A_{j,\infty}$,
 depending only on
$[a,b]$, such that
$$(d_v)_{v\in S}\in M_\infty \times (\prod_{v\in S_f}M_v)
$$
where $M_v:=\{a\in \mathbf A_{j,v}:\|a\|_v=1\}$.

Therefore for $M:=M_\infty \times \prod_{v\in S_f} M_v$,
$$Q_I\subset (\mathbf A_j\cap \G) Y M Q_{I\cup\{j\}} (W\cap H_I).$$
Since $\mathbf A_j$ normalizes $\bQ_{I\cup\{j\}}$,
it follows that
$$Q_I\subset (\mathbf A_j\cap \G) Q_{I\cup\{j\}}  M_0 $$
for some compact subset $M_0$ of $Q_I$.
\end{proof}

For $I\subset \{1, \ldots, r\}$,
we define a finite subset $F_I\subset  \mathbf Q_I(K)$
such that
$$\mathbf Q_I(K)=(\G\cap \mathbf Q_I)F_I(\mathbf P\cap \mathbf Q_I)(K) .$$
Since $\mathbf A_I$ normalizes $\bQ_I$, there exists a finite
subset $\tilde F_I\subset  \mathbf Q_I(K)$ such that
\begin{equation}\label{si}
F_I^{-1}(\bQ_I\cap \G)(\mathbf A_I\cap \G)\subset
F_I^{-1}(\mathbf A_I\cap \G)(\bQ_I\cap \G)\subset
(\mathbf A_I\cap \G)\tilde F_I^{-1}(\bQ_I\cap \G).
\end{equation}
We put $$\Lambda(I):=\tilde F_I^{-1} (Q_I\cap \G)\subset \mathbf Q_I(K).$$
Note that $\mathbf P_\emptyset=\mathbf Q_\emptyset=\mathbf G$, $\mathbf A_\emptyset=
\mathbf A$, $F_\emptyset=F=\tilde F_\emptyset$,
and $\Lambda(\emptyset)=F^{-1}\G$.

%\todo{It is not clear why $\Lambda(\emptyset)$ can be taken like this. Check?}

%\todo{Additional condition on $F_I$ is imposed later.}

\begin{Lem}\label{l3.6}
For $j\in \{1, \ldots, r\}$ and $I\subset \{1, \ldots, r\}-\{j\}$,
there is a finite subset $E\subset \mathbf P(K)$  such that
$\Lambda(I\cup\{j\})\Lambda(I)\subset E \Lambda(I) $.
\end{Lem}
\begin{proof}
Same as Lemma 3.6 in \cite{EMS1}.
\end{proof}

Denote by $\Cal T$ the set of all
$l$-ordered tuples of integers $1\le i_1, \cdots, i_l \le r$
for $1\le l\le r$.
 For $I=(i_1, \cdots, i_l)\in\Cal T$,
there exists a finite subset $L(I)\subset \mathbf G(K)$ such that
$$\Lambda(\{i_1, \ldots, i_l\})\cdots \Lambda(\{i_1\})\Lambda(\emptyset)=
L(I)\G .$$
We set $L(\emptyset)=\{e\}$.

An $l$-tuple $((i_1, \lambda_1), \ldots, (i_l, \lambda_l))$
is called an {\it admissible} sequence of length $l$ if
$i_1, \ldots, i_l\subset \{1, \cdots, r\}$ are distinct and
 $\lambda_1, \ldots, \lambda_l\in
\mathbf G(K)$ satisfy
$\lambda_j \lambda_{j-1}^{-1}\in \Lambda(\{i_1,\ldots, i_{j-1}\})$ for
all $j=1,\ldots, l$ (here we set $\lambda_0=e$).
For an admissible sequence $\xi$ of length $l$,
we denote by $\Cal C(\xi)$ the set of all pairs $(i, \lambda)$ where $1\leq i\leq r$ and $\lambda\in \mathbf G(K)$ for which
there exists an admissible sequence $\eta$ of length $l+1$ extending $\xi$ and containing $(i, \lambda)$ as the last term.
The support of $\xi$, denoted by
$\operatorname{supp}(\xi)$, is defined to be the emptyset if $l=0$; and otherwise
the set $\{(i_1, \lambda_1), \ldots, (i_l, \lambda_l)\}$
if $\xi=((i_1, \lambda_1), \ldots, (i_l, \lambda_l))$.

For any $0<a<b$, $\alpha>0$ and any admissible
 sequence $\xi$,
we define
\begin{multline}
\mathcal W_{\alpha,a,b}(\xi)=\{g\in G:
\Delta_j( \lambda g ) \geq \alpha, \forall (j, \lambda)\in \Cal C(\xi)
\\
\text{ and }a\leq \Delta_i( \lambda g )\leq b, \forall (i, \lambda)\in\text{supp}(\xi)
  \} .\end{multline}

The same proof of  \cite[Prop. 1.8]{DM1} shows:
\begin{Lem} \label{l:dm}
For any
 admissible sequence $\xi=\{(i_1,\lambda_1), \ldots, (i_l, \lambda_l)\}$
 of length $l\geq 1$, we have
$$\mathcal W_{\alpha,a,b}(\xi)=\mathcal W_{\alpha, a,b}(I, \lambda_l)$$
where $I=\{i_1, \ldots, i_l\}$ and
\begin{multline}
W_{\alpha, a,b}(I, \lambda)
:=\{g\in G:
\Delta_j( \theta \lambda g)\geq\alpha, \forall j\notin I, \forall
\theta\in \Lambda(I) \\
\text{ and }a\leq \Delta_i(\lambda g)\leq b, \forall i\in I \} .\end{multline}
\end{Lem}

Note that $\lambda_l$ arises in the above way if and only if
$\lambda_l\in  L(I)\G$.

For any subset $I\subset \{1, \ldots, r\}$,
note that $W_\infty\cap \bH_{I,\infty}$ is a maximal compact subgroup
of $\bH_{I,\infty}$. Set $J:=\{1, \ldots, r\}\setminus I$.
 By reduction theory,
there exist a compact subset $C_I\subset  \mathbf U_{J,\infty}\cap  \bH_{I,\infty}$,
 a finite subset $E_I\subset  \mathbf H_I(K)$, and $t_I>0$ such that
$$H_I =(\G\cap H_I)E_I \left(C_I\Omega_I(W_\infty \cap
  \mathbf H_{I,\infty})\times
(W_f\cap   \mathbf H_{I,S_f})\right)$$
where $$\Omega_I=\{(s_v)_{v\in S_\infty}: s_v\in
 A_{J,v}^0 ,\; 0< \alpha_j(s_v)\leq t_I,
\quad\forall j\in J,\; \forall v\in S_\infty\}.
$$
We enlarge the finite subset $F_I$, chosen above,  so that
$$(\Gamma\cap \bH_I)E_I(\G\cap \mathbf U_I)\subset (\G\cap \bQ_I)F_I .$$

We have $U_I=(\Gamma\cap U_I) D_I'$
for some $D_I'=D_I\times (\mathbf U_{I,S_f}\cap W_f)$
with $D_I\subset \mathbf U_{I, \infty}$.
Then for $C'_I=C_I\Omega_I(W_\infty \cap
  \mathbf H_{I,\infty})\times
(W_f\cap   \mathbf H_{I,S_f})$,
\begin{align}\label{eq:decomp}
Q_I &= U_I H_I
= U_I (\Gamma\cap H_I)E_I
\\
&=(\Gamma\cap H_I)E_I U_I C' =(\Gamma\cap H_I)E_I (\Gamma\cap U_I)D_I' C'_I \notag \\
&=(\G\cap \bQ_I) F_I  (\Psi_I \Omega_I(W_\infty \cap
 \bQ_{I,\infty})\times
(W_f\cap   \bQ_{I,S_f}))\notag
\end{align}
where $\Psi_I$ is a compact subset of $(\mathbf Q_I\cap \mathbf Q_J)_{\infty}$.

In the proof of the next proposition, we use the following lemma,
which follows from continuity of the norms:

\begin{Lem}\label{dm1.4} Let
 $1\leq i\leq r$ and $C$ be a compact subset of $G$.
Then for some $c>0$,
$$\Delta_i(gx)\geq c\cdot
 \Delta_i(g)\quad\text{ for all $x\in C$ and $g\in G$} .$$
\end{Lem}
%\begin{proof} Without loss of generality, we may assume $C=\prod_pC_p$.
%The proof of Lemma 1.4 in [DM] shows
%that  there exist $c_p>0$ such that
%$$\|w_ig_vx_p\|_p\geq c_p\|w_ig_v\|_p$$for
%all $v\in S$ and $x_p\in C_p$. Hence the claim works
%for $c=\prod c_p$.
%\end{proof}

For $g=(g_v)_{v\in S_\infty}\in \bG_{\infty}$, set
$$d_i(g):=\prod_{v\in S_\infty}\|w_i g_v\|_v .$$

%and a subset of $G$ as $C_\infty\times C_f$
%where $C_\infty\subset \mathbf G(\br)$
%and $C_f\subset \mathbf G_{S_f}$.

\begin{Prop} \label{p:compact}
For any admissible sequence $\xi$ of length $0\le l\le r$
 and positive $a<b$ and $\alpha>0$,
 the set $\G\ba \G \mathcal W_{\alpha,a, b}(\xi)$ is relatively compact.
\end{Prop}

\begin{proof}
For simplicity, write $\mathcal W=\mathcal W_{\alpha,a, b}(\xi)$.

Let $\xi$ be the empty sequence.
Then
$$\mathcal W=\{g\in G:
\Delta_j(  \lambda g) \geq\alpha,\quad \forall j,\; \forall \lambda\in \Lambda(\emptyset)\}.$$
Every $g\in \mathcal W$ has a decomposition $g=(\lambda,\lambda)(\psi \omega k_\infty , k_f)$, $\psi\in \Psi_\emptyset$, $w\in \Omega_\emptyset$,
$k_\infty\in W_\infty$ and $k_f\in W_f$ as
in (\ref{eq:decomp}) where $\lambda\in \G \tilde F_\emptyset
=\Lambda(\emptyset)^{-1}$. Hence,
$$\Delta_j(\psi \omega k_\infty  , k_f)= d_j( \psi \omega k_\infty  )=d_j(\omega) \geq c\alpha $$
where $c>0$ is a constant depending on $\Psi_\emptyset$.
Since $d_j(\omega)=\prod_{v\in S_\infty}|\alpha_j(\omega)|_v^{m_j}$,
we have
$$\mathcal W\subset \G \tilde F_\emptyset ( \Psi_\emptyset
\tilde \Omega_{\emptyset} W_\infty  \times W_f)$$
where $$\tilde \Omega_\emptyset=\{(\omega_v) \in A_{\emptyset, \infty}^0:\,
 (c\alpha)^{1/m_j} \leq \prod_{v\in S_\infty}
|\alpha_j(\omega_v)|_v\leq t_\emptyset,
\quad \forall j\} .$$
This shows that $\G\ba\G \mathcal W$ is relatively compact in this case.

Now let $\xi=((i_1, \lambda_1), \ldots, (i_l, \lambda_l))$
 be an admissible sequence of length $l\geq 1$ and $I(j)=\{i_1,\ldots,i_j\}$.
We claim that there exist compact subsets $M_1, \ldots, M_l$
such that for any $j=1, \ldots, l$ and $g\in \mathcal W$,
$$
\lambda_j \mathcal W \subset (S_{i_{j}}\cap \G) Q_{I(j)} M_j^{-1}.$$

We prove the claim by induction. For $j=1$, we can take $M_1=M_0^{-1}$
where $M_0$ is as in Lemma \ref{dm1.5} with $I=\emptyset$ and $j=i_1$.
Suppose that the sets $M_1, \ldots, M_j$ have been found.
By Lemma \ref{dm1.4},
there is $c\in (0, 1) $ such that
$$\Delta_{i_{j+1}}(hx)\geq c\cdot \Delta_{i_{j+1}}(h)$$ for all $x\in M_j\cup M_j^{-1}$
and $h\in G
$.
By Lemma \ref{dm1.5}, there exists a compact set $M_0$ such that
\begin{equation}\label{eq:step}
\{g\in Q_{I(j)}: \Delta_{i_{j+1}}(g)\in [c a,c^{-1}b]\}\subset (A_{i_{j+1}}\cap \Gamma)Q_{I(j+1)}M_0.
\end{equation}
Let $M_{j+1}=M_jM_0^{-1}$.
For $g\in \mathcal W$, there exists $m_j\in M_j$ such that $\lambda_jg m_j\in
(A_{i_{j}}\cap\G) Q_{I(j)}$.
Since $\lambda_{j+1}\lambda_j^{-1}\in Q_{I(j)}$
and $A_{i_{j}}\cap \G$ normalizes $Q_{I(j)}$,
we have $$\lambda_{j+1}g m_j\in(A_{i_{j}}\cap \G) Q_{I(j)} .$$
Hence for some $\gamma_{j}\in A_{i_{j}}\cap \Gamma$,
$ \gamma_{j}\lambda_{j+1}gm_j \in Q_{I(j)}$.
Since $\alpha_{i_{j+1}}(\gamma_j)=1$,
$$\Delta_{i_{j+1}}(\gamma_{j}\lambda_{j+1}gm_j)=
 \Delta_{i_{j+1}}(\lambda_{j+1}gm_j),$$
and
$$c a\leq c  \Delta_{i_{j+1}}(  \lambda_{j+1} g) \leq
\Delta_{i_{j+1}}( \lambda_{j+1} gm_j)\leq c^{-1} \Delta_{i_{j+1}}( \lambda_{j+1}g )
\leq c^{-1}b.$$
By (\ref{eq:step}), there exists $m_0\in M_0$ such that
$$\gamma_j \lambda_{j+1} gm_{j}\in (A_{i_{j+1}}\cap \G)  Q_{I(j+1)}m_0.$$
So for $m_{j+1}=m_jm_0^{-1}$,
we have $$ \lambda_{j+1} g m_{j+1} \in(A_{i_{j+1}}\cap \G) Q_{I_{j+1}}$$
proving the claim.

By the above claim,
\begin{equation}\label{eq:claim0}
\lambda_l \mathcal W \subset (A_{i_l}\cap \G) Q_{I}M_l^{-1}
\end{equation}
where $I:=\{i_1, \ldots, i_l\}$.
If $I=\{1, \ldots, r\}$,
then $\G\cap Q_I\ba Q_I$ is compact. Hence
$\G \ba \G\lambda_r^{-1} (A_{i_r}\cap \G) Q_I M_r$ is compact,
which implies that $\G\ba \G \mathcal W$ is relatively compact.

Now suppose $I$ is a proper subset. Then by (\ref{eq:decomp}) and
(\ref{eq:claim0}), for $g\in \mathcal W$,
$$\delta \gamma\lambda_l g m=(\psi \omega k_\infty , k_f)\in
 \Psi_{I}\Omega_{I}W_\infty\times
W_f$$
for some $\delta\in
F_I^{-1}(Q_I\cap \G)$, $\gamma\in A_{i_l}\cap \G$, and $m\in M_l$.
Hence, for every $j\notin I$,
$$|\alpha_j(\omega)|_\infty^{m_j}=d_j(\psi \omega k_\infty)=
\Delta_j(\delta \gamma \lambda_l g m) .$$
By (\ref{si}), $\delta \gamma=
\gamma'\theta $ where $\gamma'\in  A_I\cap \G$ and $\theta\in \Lambda(I)$.
Since $A_I$ acts trivially on the vectors $w_j$, $j\notin I$, we have
$$
\Delta_j(\delta \gamma\lambda_l g m) =
\Delta_j (\theta \lambda_l g m) .$$
%Since $\theta \lambda_l\in \Cal C(\xi)$,
By Lemma \ref{l:dm}, we have $\Delta_j (\theta \lambda_l g) \geq \alpha.$
Hence, by Lemma \ref{dm1.4}, there exists $\beta>0$, depending only on
$\alpha$ and $M_l$, such that
$\Delta_j(\theta \lambda_l g  m)\geq \beta$.
This shows that $\prod_{v\in S_\infty}|\alpha_j(\omega)|_v^{m_j}\geq \beta$ for
$j\notin I$.
Therefore,
if we set
$$
\tilde\Omega_I=\{\omega\in A_J:  \beta^{1/m_j} \leq \prod_{v\in S_\infty}
|\alpha_j(\omega)|_v \leq t_I,\;
j\in J\}
$$
where $J$ is the complement of $I$, then
 $$\mathcal W\subset \lambda_l^{-1} \G
 F_I (\Psi_I \tilde \Omega_I W_\infty  \times W_f)M_l^{-1},
$$ and the
later set
is compact modulo $\G$.
\end{proof}

\subsection{Proof of Theorem \ref{dm2}}
Fix $v\in S$ and a vector space $K_v^N$.
We define $\mathcal P_d^*$ is
the set of polynomial maps $K_v \to K_v^N$
(res. $\br \to \c^N$)
of degree less than $d$ if $K_v \ne \c$,
(resp. $K_v=\c$).
%Fixing $\bG\subset M_N$ and a basis for $M_N$,
We write $f\in \Cal P_d^*$
as $(f_1, \cdots, f_{N})$.
We set
$$\|f(x)\|_v=\max_{i}|f_i(x)|_v .$$
Recall that by an interval of a non-archimedean local field,
we mean a subset of $K_v$ of the form
$I=\{t\in K_v: |t-t_0|_v< \delta\}$.
There is the unique $k$ such that $q_v^{k} < \delta\le q_v^{k+1}$.
Then $2q_v^k$ is called the diameter of $I$.
In the case when $K_v=\c$,
as $\mathcal P_d^*$ consists of polynomial maps
defined in $\br$, the intervals are understood as subsets of $\br$
and the meaning of diameter is then clear.

\begin{Lem}\label{c2.18}
Given $M>1$, there exists $\eta\in (0,1)$ such that
for any $f\in  \Cal P_d^*$ and any interval $I$,
there exists a subinterval $I_0\subset I$ with $\hbox{\rm diam}(I_0)\geq
\eta\cdot \hbox{\rm diam}(I)$ satisfying
$$
\sup_{I} \|f\|_v\leq M\cdot \inf_{I_0} \|f\|_v
$$
\end{Lem}
\begin{proof}
For the archimedean version of this lemma, see \cite[Corollary 2.18]{EMS1}.
Let $v$ be non-archimedean.
Since $I$ can be expressed as a disjoint union $\cup J$ of intervals
so that on each interval $J$, there is $i$
such that $\|f(x)\|_v=|f_i(x)|_v$ for all $x\in J$.
Therefore it suffices to prove the above claim for $N=1$.

There exists $t_0\in I$ such that $\sup_I |f|_v=|f(t_0)|_v$.
It follows from the Lagrange interpolation formula that there
exists $M_i>0$, depending on $I$, such that
$$
\sup_I |f^{(i)}|_v\leq M_i \cdot \sup_I |f|_v\quad\hbox{for all $f\in\mathcal{P}_d^*$}
$$
where $f^{(i)}$ is the $i$-th derivative of $f$.
Let $\delta$ denote the diameter of $I$. Let $n\in \n$ be big enough so that
$$M^{-1}<1-\sum _{i=1}^d q_v^{-ni}\delta^i\frac{M_i}{i!}$$
and $ I_0:=\{t:|t-t_0|_v\le  q_v^{-n}\delta\}$ is contained in $I$.
 Then
using the Taylor formula, we deduce that for every
$t\in I_0$, $$
|f(t)|_v\geq |f(t_0)|_v -\left(\sum_{i=1}^d
(q_v^{-n}\delta)^i \frac{M_i}{i!}\right) \sup_I |f|_v=M^{-1} \sup_I |f|_v.
$$

Hence $\sup_I|f|_v\le M \inf_{I_0} |f|_v$ and
the diameter of $I_0$ is $ 2 q_v^{-n}\delta $. Hence this proves the claim.
\end{proof}

\begin{Lem}\label{c2.17} Given $\eta\in (0,1)$, there exists $M> 1 $ such that
for any interval $I$ and any subinterval $I_0\subset I$
with $\hbox{\rm diam}(I_0)\geq \eta\cdot \hbox{\rm diam}(I)$,
\begin{equation}\label{eq:well}
\sup_{I}\|f\|_v\leq M \cdot \sup_{I_0} \|f\|_v\quad\hbox{for all $f\in
  \Cal P_d^*$.}
\end{equation}
\end{Lem}

\begin{proof} For the archimedean case, this is proved in \cite[Coro. 2.17]{EMS1}.
We give a proof in the non-archimedean case.
By Lemmas 2.1 and 2.4 in \cite{KT},
%Note that all functions in $\Cal P_d^*$ are $(C, \alpha)$-good
%for uniform $C>0$ and $0<\alpha<1$ (see [KT] for the definition of
%$(C, \alpha)$-good as well as for the claim).
%Hence
there exist $C, \alpha >0$ such that
for any $\e>0$ and any interval $I$
\begin{equation}\label{eq:well2}
\theta_v\{x\in I: \|f(x)\|_v<\e\} \leq C \left(\frac{\e}{\sup_I
    \|f\|_v}\right)^{\alpha} \theta_v(I).
\end{equation}
Choose $n\in\mathbb{N}$
and $M>1$ satisfying $\eta>q_v^{-n}$ and $M^\alpha> C q_v^n.$
Then for any subinterval $I_0\subset I$
with $\hbox{\rm diam}(I_0)\geq \eta\cdot \hbox{\rm diam}(I)$,
 $\theta_v(I_0)> q_v^{-n} \theta_v(I)$. Applying (\ref{eq:well2}) with
$\e=M^{-1}\cdot \sup_{I}\|f\|_v$, we deduce that
$$
\theta_v \left\{x\in I: M\cdot \|f(x)\|_v<\sup_{I}\|f\|_v\right\} \leq
 q_v^{-n} \theta_v (I).
$$
Therefore there exists $x\in I_0$ such that
$$
M\cdot \|f(x)\|_v\geq \sup_{I}\|f\|_v.
$$
This proves the lemma.
\end{proof}

\begin{Prop}\label{p3.8}
There exists $M>1$ such that
for any $\alpha>0$, any interval $I$,
and  a subfamily $\Cal F\subset \Cal P_d^*$ satisfying:
\begin{enumerate}
\item[(i)]  For any $t_0\in I$, $\#\{\phi_f(t):=\|f(t)\|_v: f\in \mathcal{F},\, \|f(t_0)\|_v<\alpha\} <\infty$,
\item[(ii)] For any $f\in \Cal F$, $\sup_{t\in I} \phi_f(t)\geq\alpha,$
\end{enumerate}
one of the followings holds:
\bi
\item [(a)] There is $t_0\in I$
such that
$$\phi_{f}(t_0)\geq  \alpha\quad\text{for all $f\in \Cal F$}.$$
\item [(b)]
There exist an interval $I_0\subset B$ and $f_0\in \Cal F$ such that
$$
\phi_{f_0}(I_0)\subset [\alpha/M,\alpha M]\quad\hbox{and}\quad
\sup_{I_0}\phi_f\geq \alpha/M\quad\hbox{for all $f\in \Cal F$}.
$$
\ei
\end{Prop}

%\change{condition as in Eskin-Mozes-Shah added}

\begin{proof}
Pick $t_0\in I$ and suppose that (a) fails, that is,
$$
\mathcal{F}_1=\{\phi_f: f\in \mathcal{F},\; \|f(t_0)\|_v<\alpha\}\neq \emptyset.
$$
By (i), the set $\mathcal{F}_1$ is finite. By Lemma \ref{c2.17},
there exists $M_1>1$ such that for every $\phi_f \in\mathcal{F}_1$ and $k\in \mathbb{Z}$,
$$
\sup_{|t-t_0|_v\leq q_v^{k+1}} \|f(t)\|_v
 \leq M_1\cdot \sup_{|t-t_0|_v\leq q_v^{k}} \|f(t)\|_v.
$$
We set $E=\{t:\, |t-t_0|_v\leq q_v^{k}\}$, where $k$ is the smallest
integer such that $\sup_{E} \|f\|_v\geq \alpha$ for all
$t\mapsto \|f(t)\|_v\in\mathcal{F}_1$. Such $k$ exists by (ii).
Then there is $\phi_{f_0}\in\mathcal{F}_1$ such that $\sup_{E} \|f_0\|_v\leq
\alpha M_1$.
By Lemma \ref{c2.18}, there exists a subinterval $I_0\subset E$ such
that $\hbox{diam}(I_0)\geq \eta\cdot \hbox{diam}(I)$ and
$$
\inf_{I_0} \|f_0\|_v\geq \alpha/M_1.
$$
By Lemma \ref{c2.17}, there exists $M_2>1$ such that
$$
\sup_{I_0} \|f\|_v\geq \alpha/M_2 \quad\hbox{for all $f\in \Cal F$}.
$$
This proves the proposition.
\end{proof}

%\begin{Thm}\label{dm} Let $\alpha>0$ be given and $v\in S$.
%There exists a compact subset $C\subset X$ such that
%for any unipotent one-parameter subgroup $U=\{u(t): t\in K_v \}\subset
% \mathbf G({K_v})$, $x\in X$ and $T>0$, one of the following holds:
%\begin{itemize}
% \item[(1)] there exists $i\in \{1, \ldots, r\}$
%such that $$\min_{\lambda\in \G F}
%$$\Delta_i(\lambda^{-1} u(s)) < \alpha \quad\hbox{for all $s\in I_T$},$$
%\item[(2)] $xu(I_T)\cap C\neq \emptyset$.
%\end{itemize}
%\end{Thm}

\begin{proof}[Proof of Theorem \ref{dm2}]
Suppose that condition (1) in the theorem fails.
We will show that for some $I\in \Cal T$, $\lambda\in L(I)\Gamma$,
$\alpha_I>0$, and $0<a_I<b_I$ depending only on $\alpha$,
\begin{equation}\label{eq:cmp}
xu(I_T)\cap \G\ba \G
W_{\alpha_I, a_I, b_I}(I, \lambda) \neq \emptyset.
\end{equation}
By Proposition \ref{p:compact}, this implies the theorem.

We construct inductively increasing sequence of tuples
$J\in\mathcal{T}$, elements $\lambda\in L(I)\G$, constants
$0<a_I<b_I$, $\alpha_I>0$, and intervals $B\subset I_T$
satisfying the following properties:
\bi
\item[(A)] $\Delta_i(\lambda u(B))\subset [a_I, b_I]$ for all $i\in I$.
\item[(B)] $\sup_{B}\phi \geq \alpha_I$ for all $\phi\in \Cal F(I,
  \lambda)$
where $\Cal F(I, \lambda)$ is the family
of functions $K_v\to \br^+$ of the form
$$\phi (t)=\Delta_j(\theta \lambda u(t))$$
where $\theta\in \Lambda(I)$, $j\notin I$
and $u\in \mathcal P_d$ \ei

Note that for some fixed constant $d_j>0$,
 $$\phi(t)=
(\prod_{w \in S\setminus{v}} \|w_j \theta\lambda\|_w ) \cdot \|f(t)\|_v$$
and $f(t):=w_j \theta \lambda  u(t)\in \mathcal P_{d_j}^*$
where $\mathcal P_{d_j}^*$ are polynomial maps into
the vector space $\wedge^{\text{dim}\mathfrak u_i}\mathfrak g (K_v)$
of degree at most $d_j$.

We start with $I=\emptyset$, $\lambda=e$, $\alpha_\emptyset=\alpha$,
$B=I_T$ which satisfy (A) and (B) because (1) fails.

Property (B) implies that $\Cal F(I, \lambda)$ satisfies condition
(ii) of Proposition \ref{p3.8}. We claim that $\Cal F(I, \lambda)$
satisfies condition (i) as well, that is, there are only finitely many $\phi\in \Cal F(I, \lambda)$
such that
$\phi(t)<\alpha$ for a fixed $\alpha$ and $t$.
Fix $\beta>0$ and any rational vector $w$ with co-prime entries in $\mathcal O$.
For any $\gamma\in \Gamma$,
$\gamma w =\alpha w'$
where $\alpha\in \mathcal O_S^*$ (here
$\mathcal O_S^*$ denotes the unit group) and the entries of $w'$ are relatively prime
to each other in $\mathcal O$.
Since $\prod_{v\in S}|\alpha|_v=1$ for $\alpha\in \mathcal O_S^*$,
the claim follows from the fact that $\tilde F_I$ is finite and that
there are only finitely many vectors $w'$ with coefficients in $\mathcal O$ whose entries
are relative prime to each other and $\prod_{v\in S}\|w'\|_v<\beta$.

By Proposition \ref{p3.8}, one of the following holds:
\bi
\item[(a)]
For some $t_0\in B$,
$$\Delta_j(\theta\lambda u(t_0))\geq \alpha_I$$
for all $\theta\in \Lambda(I)$, $j\notin I$, and
 $u\in \mathcal P_{d}$. In this case, using (A),
we have $\G u(t_0)\in W_{\alpha_I, a_I, b_I}(I)$
and hence we stop the process.
\item[(b)]
There exist $j_0\notin I$, $\theta_0\in \Lambda(I)$ and an interval
$B_0\subset B$
such that
$$\Delta_{j_0}(\theta_0 \lambda u(B_0))\subset [\alpha_I/M, \alpha_I M]$$
and for all $\theta\in \Lambda(I)$ and $j\notin I$,
$$\sup_{t\in B_0}\Delta_j(\theta\lambda u(t))\geq \alpha_I/M .$$
\ei

In case (b), we set $I_1=I\cup\{j_0\}$ and $\lambda_1=\theta_0\lambda$.
Then since $\Delta_i(\theta_0g)=\Delta_i(g)$ for all $i\in I$ and $g\in \bG_S$,
condition (A) is satisfied for suitable $0<a_{I_1}<b_{I_1}$ and  $B_0$.
By Lemma \ref{l3.6}, there is a finite subset $E\subset \bP(K)$
such that for any  $\theta\in \Lambda(I\cup\{j_0\})$,
there exists $\theta'\in \Lambda(I)$
and $x\in E$ such that $\theta\theta_0=x\theta'$,
Hence this implies for any $j\notin I_1$,
\begin{align*}
\sup_{t\in B_0}\Delta_j(\theta\lambda_1 u(t))
=\sup_{t\in B_0}\Delta_j( x \theta '\lambda u(t))=
\sup_{t\in B_0}\Delta_j(x)\Delta_j(\theta '\lambda u(t))\geq
 \beta\alpha_I/M
\end{align*}
where $\beta=\min_{E}\Delta_j>0$ depends only on $I$ and $j_0$.
Hence Condition (B) is satisfied for the family $\Cal F(I_1, \lambda_1)$,
 $ \alpha_{I_1}=\beta \alpha_I/M$ and $B_0$. This completes the
 description of the inductive step.
Since the cardinality of $I$ increases, this process must stop after
finitely many steps, and we deduce that (\ref{eq:cmp}) holds.
\end{proof}

%\change{Minor modification in the proof. Verification of (i) from
%  Prop. \ref{p3.8} added.}

%\change{minor modifications in the proof}

\newpage
\renewcommand{\thesection}{A}
\renewcommand{\thetheorem} {{\bf A.\arabic{theorem}}}

\section{Appendix:
Symmetric homogeneous spaces over number fields with finitely many
orbits (by Mikhail Borovoi)} \label{s:app}

Let $G$ be a connected linear algebraic group over a  field $\kk$ of
characteristic 0. Let $H\subset G$ be a connected closed $K$-subgroup. Let
$X=H\backslash G$ be the corresponding homogeneous space. The group
$G(\kk)$ acts on $X(\kk)$ on the right. We consider the set of
orbits $X(\kk)/ G(\kk)$.

We fix an algebraic closure $\kbar$ of $\kk$ and write
$\Gbar=G\times_\kk \kbar$, $\Hbar=H\times_\kk \kbar$. We say that
$(G,H)$ is a \emph{symmetric pair} if $G$ is semisimple and $\Hbar$
is the subgroup of invariants $\Gbar^\theta$ of some involutive
automorphism $\theta$ of $\Gbar$. In this case we say also that $H$
is a \emph{symmetric subgroup of} $G$ and that $X$ is a
\emph{symmetric space of} $G$.

Let $K$ be a number field. In this Appendix we give a list of all
symmetric pairs $(G,H)$ over $K$ with adjoint absolutely simple $G$
and semisimple $H$, such that the set of orbits $X(\kk)/ G(\kk)$ is
finite (Theorem \ref{thm:main-finite-adjoint}). 
We show that the assumption that
$X(\kk)/ G(\kk)$ is finite is equivalent to the assumption that
$G(\kv)$ acts on $X(\kv)$ transitively for almost all places $v$ of
$\kk$.

The plan of the Appendix is as follows. In Section \ref{sec:1} we
consider a connected $\kk$-group $G$ and a connected closed $\kk$-subgroup
$H\subset G$ over a number field $\kk$. We prove that the set of
$\kk$-orbits  $X(\kk)/G(\kk)$ is finite if and only if the set of
adelic orbits $X(\AA)/G(\AA)$ is finite (here $\AA$ is the ad\`ele
ring of $\kk$). We show that the set $X(\AA)/G(\AA)$ is finite if
and only if $\# X(\kv)/G(\kv)=1$ for almost all $v$. We give a
criterion  when $\# X(\kv)/G(\kv)=1$ for almost all $v$ in terms of
the induced homomorphism $\pi_1(\Hbar)\to\pi_1(\Gbar)$, where
$\pi_1$ is the {algebraic fundamental group} introduced in
\cite[Sect. 1]{Bor98}. These results constitute Theorem
\ref{thm:main-f-m-orbits}.

In Section \ref{sec:2} we give corollaries of Theorem
\ref{thm:main-f-m-orbits}. We show that the finiteness of
$X(\kk)/G(\kk)$ is related to the following condition: the
homomorphism $\pi_1(\Hbar)\to\pi_1(\Gbar)$ is injective.

In Section \ref{sec:3} we assume that $\kk$ is algebraically closed
and that both $G$ and $H$ are semisimple. We write $G\sc$ and $H\sc$
for the universal coverings. We show that the homomorphism
$\pi_1(H)\to\pi_1(G)$ is injective if and only if the subgroup
$H':=\im[H\sc\to G\sc]\subset G\sc$ is simply connected.

In Section \ref{sec:4} we again assume that $\kk$ is algebraically
closed. We give a list of all symmetric pairs $(G,H)$ over $\kk$
with  simply connected absolutely almost simple $G$ and semisimple
$H$ such that $H$ is simply connected (Theorem
\ref{cor:Kac-sc-table}).

In Section \ref{sec:5} $\kk$ is a number field and $(G,H)$ is a
symmetric pair over $\kk$, such that  $G$ is an absolutely almost
simple $\kk$-group and $H$ is semisimple $\kk$-subgroup. We consider
two cases: either $G$ is simply connected or $G$ is adjoint.
We give a list of all such symmetric pairs $(G,H)$
with finite  $X(\kk)/G(\kk)$ (Theorems \ref{thm:symmetric-sc} and
\ref{thm:main-finite-adjoint}). We show that for such $(G,H)$ with
finite set of $\kk$-orbits $X(\kk)/G(\kk)$, this set of $\kk$-orbits
is related to the set of ``real'' orbits (Theorem
\ref{thm:cor-orbits}). In particular, if $\kk=\QQ$, then any
$G(\RR)$-orbit in $X(\RR)$ contains exactly one orbit of $G(\QQ)$ in
$X(\QQ)$.

In Section \ref{sec:6} (Addendum) we give examples of homogeneous
spaces $X=H\backslash G$ (symmetric or not, with $G$ absolutely
almost simple or not), satisfying assumptions (i--iii) of Theorem
\ref{maint} but not covered by Theorems \ref{thm:symmetric-sc} and
\ref{thm:main-finite-adjoint}.

The author is very grateful to \`E.B.~Vinberg and  A.G.~Elashvili
for their invaluable help in proving Theorem \ref{cor:Kac-sc-table}.
\bigskip

\subsection{Orbits over a number field and over adeles}\label{sec:1}

\begin{subsec}{} \label{subsec:Notation}
Let $\kk$ be a number field, and let $\kbar$ be a fixed algebraic
closure of $\kk$. Let $G$ be a connected linear $\kk$-group.
Let $H\subset G$ be a connected closed $\kk$-subgroup. Set $X=H\backslash
G$, it is a right homogeneous space of $G$.
We would like to investigate, when the set of orbits $X(\kk)/G(\kk)$
of $G(\kk)$ in $X(\kk)$ is finite.

We write $i\colon H\into G$ for the inclusion map. We consider the
induced morphism of $\Gal(\kbar/\kk)$-modules
$$
i_*\colon \pi_1(\Hbar)\to\pi_1(\Gbar),
$$
where $\pi_1$ is the algebraic fundamental group introduced in
\cite[Sect. 1]{Bor98}, see also \cite[\S6]{CT}.
Note that if we choose an embedding $\kbar\into \CC$ 
and an isomorphism $\pi_1^{\text{top}}(\CC^*)\isoto\ZZ$,
then we obtain a canonical isomorphism of abelian groups 
$\pi_1(\Gbar)\simeq \pi_1^{\text{top}}(G(\CC))$, cf. \cite[Prop.~1.11]{Bor98},
where $\pi_1^{\text{top}}$ denotes  the usual topological fundamental group.
In particular, $G$ is simply connected if and only if $\pi_1(\Gbar)=0$.

Let $\gg$ denote the image of $\Gal(\kbar/\kk)$ in
$\Aut\;\pi_1(\Hbar)\times\Aut\;\pi_1(\Gbar)$; it is a finite group.
Let $L\subset\kbar$ be the field  corresponding to the subgroup
$\ker[\Gal(\kbar/\kk)\to\gg]$ of $\Gal(\kbar/\kk)$, so $L/\kk$ is
a finite Galois extension with Galois group $\gg$. For any place $v$
of $\kk$, let $\gg_v\subset\gg$ denote a decomposition group of $v$
(defined up to conjugacy). For almost all $v$ the group $\gg_v$ is
cyclic.

Let $\hh\subset\gg$ be a subgroup. We shall consider the group of
coinvariants $\pi_1(\Hbar)_\hh$ and the subgroup of torsion elements
$(\pi_1(\Hbar)_\hh)\tors$. We shall also consider the induced map
$$
i_*\colon (\pi_1(\Hbar)_\hh)\tors\to (\pi_1(\Gbar)_\hh)\tors\;.
$$

We write $\sV$ for the set of all places of $\kk$. We write $\sV_f$
(resp. $\Vinf$) for the set of all finite (resp. infinite) places of
$\kk$. We write $\kk_v$ for the completion of $\kk$ at $v\in\sV$,
and $\AA$ for the ad\`ele ring of $\kk$.
\end{subsec}

\begin{theorem}\label{thm:main-f-m-orbits}
Let $G$ be a connected linear algebraic group over a number field
$\kk$. Let $H\subset G$ be a connected closed $\kk$-subgroup. Set
$X=H\backslash G$. Then the following four conditions are
equivalent:

(i) The set of $\kk$-orbits $X(\kk)/G(\kk)$ is finite.

(ii) The set of adelic orbits $X(\AA)/G(\AA)$ is finite.

(iii) We have  $\# X(\kv)/G(\kv)=1$ for almost all places $v$ of
$\kk$.

(iv) For any {\bf cyclic} subgroup $\hh\subset\gg$ the map
$$
(\pi_1(\Hbar)_\hh)\tors\to(\pi_1(\Gbar)_\hh)\tors
$$
is injective.
\end{theorem}

\begin{proof}
Write
\begin{align*}
\ker(\kk,H\to G)&=\ker[H^1(\kk,H)\to H^1(\kk,G)],\\
\ker(\kv,H\to G)&=\ker[H^1(\kv,H)\to H^1(\kv,G)].
\end{align*}
We have canonical bijections
\begin{align*}
X(\kk)/G(\kk)&\isoto\ker(\kk,H\to G),\\
X(\kv)/G(\kv)&\isoto\ker(\kv,H\to G),
\end{align*}
see \cite[Ch.~I \S5.4, Cor.~1 of Prop.~36]{Serre}.

In \cite[Sections 2,3]{Bor98} we defined, for any connected linear group
$H$ over a field $\kk$ of characteristic 0, an abelian group
$H^1_\ab(\kk, H)$ and an abelianization map of pointed sets
$$
\ab^1\colon H^1(\kk, H)\to H^1_\ab(\kk,H)
$$
(see also \cite[Prop. 8.3]{CT} in any characteristic). Both
$H^1_\ab(\kk,H)$ and $\ab^1$ are functorial in $H$.

Now let $\kk$ be a number field. Set $\Gamma=\Gal(\kbar/\kk)$,
$\Gamma_v=\Gal(\kbar_v/\kk_v)$. We regard $\Gamma_v$ as a subgroup
of $\Gamma$.

For $v\in\sV_f$ we defined in \cite[Prop.~4.1(i)]{Bor98} a canonical
isomorphism $\lambda_v\colon H^1_\ab(\kk_v,H)\isoto
(\pi_1(\Hbar)_{\Gamma_v})\tors$. Here we set
$$
\lambda'_v=\lambda_v\colon H^1_\ab(\kk_v,H)\isoto
(\pi_1(\Hbar)_{\Gamma_v})\tors\;.
$$
For $v\in\Vinf$ we defined in \cite[Prop.~4.2]{Bor98} a canonical
isomorphism
$$\lambda_v\colon H^1_\ab(\kk_v,H)\isoto H^{-1}(\Gamma_v,\pi_1(\Hbar)),$$
where $H^{-1}$ denotes the Tate cohomology 
(note that $\Gamma_v$ is finite for $v\in\Vinf$).
Here we define a homomorphism $\lambda'_v$ as the composition
$$
\lambda'_v\colon H^1_\ab(\kk_v,H)\labelto{\lambda_v}
H^{-1}(\Gamma_v,\pi_1(\Hbar)) \into(\pi_1(\Hbar)_{\Gamma_v})\tors\;.
$$
For any $v\in\sV$ we define the Kottwitz map $\beta_v$ as the
composition
$$
\beta_v\colon H^1(\kk_v,H)\labelto{\ab^1}
H^1_\ab(\kk_v,H)\labelto{\lambda'_v}
(\pi_1(\Hbar)_{\Gamma_v})\tors\;.
$$
This map $\beta_v$ is functorial in $H$. Note that for  $v\in\sV_f$
the  maps $\beta_v$  and $\ab^1\colon H^1(\kk_v,H)\to
H^1_\ab(\kk_v,H)$ are bijections, cf. \cite[Cor.~5.4.1]{Bor98}.
 Thus for $v\in\sV_f$ we have a
canonical and functorial in $H$ bijection $H^1(\kk_v,H)\isoto
(\pi_1(\Hbar)_{\Gamma_v})\tors$.

For any $v\in\sV$ we define a map $\mu_v$ as the composition
\begin{equation}\label{eq:mu}
\mu_v\colon H^1_\ab(\kk_v,H)\labelto{\lambda'_v}
(\pi_1(\Hbar)_{\Gamma_v})\tors \labelto{\textup{cor}_v}
(\pi_1(\Hbar)_{\Gamma})\tors\;,
\end{equation}
where $\textup{cor}_v$ is the obvious map.

We prove that (ii)$\Leftrightarrow$(iii). Since $H$ is connected,
using Lang's theorem and Hensel's lemma, we can prove easily that
\begin{equation}\label{eq:adelic-orbits}
X(\AA)/G(\AA)=\bigoplus_v X(\kk_v)/G(\kk_v).
\end{equation}
Here $\bigoplus$ means that we take the families of local orbits
$(o_v\in X(\kk_v)/G(\kk_v))_{v\in\sV}$ with $o_v=x_0\cdot G(\kv)$
for almost all $v$, where $x_0\in X(\kk)$ is the image of the
neutral element $e\in G(\kk)$. For any place $v$ of $\kk$ the set
$X(\kv)/G(\kv)$ is finite (because $H^1(\kv,H)$ is finite, see
\cite[Ch.~III \S4.4, Thm.~5 and Ch.~III \S4.5, Thm.~6]{Serre}). It
follows that $X(\AA)/G(\AA)$ is finite if and only if
$\#X(\kv)/G(\kv)=1$ for almost all $v$. Thus
(ii)$\Leftrightarrow$(iii).

We prove that (iv)$\Rightarrow$(iii). For almost all $v$ the group
$\gv$ is cyclic, hence by assumption (iv) we have for such $v$
$$
\ker[(\pi_1(\Hbar)_\gv)\tors\to(\pi_1(\Gbar)_\gv)\tors]=0.
$$
But for $v\in\sV_f$ we have canonical bijections
\begin{multline*}
X(\kk_v)/G(\kk_v)\isoto \ker[H^1(\kk_v,H)\to H^1(\kk_v,G)]
\\
\isoto \ker[(\pi_1(\Hbar)_\gv)\tors\to(\pi_1(\Gbar)_\gv)\tors].
\end{multline*}
Thus for almost all $v$ we have $\#(X(\kv)/G(\kv))=1$. This proves
that (iv)$\Rightarrow$(iii).
\bigskip

We prove that (iii)\implies(iv). Indeed, assume that (iv) does not
hold, i.e. there exists a cyclic subgroup $\hh\subset\gg$ such that
$$
\ker[(\pi_1(\Hbar)_\hh)\tors\to(\pi_1(\Gbar)_\hh)\tors]\neq 0.
$$
Then by Chebotarev's density theorem there exist infinitely many
finite places $v$ of $K$ such that $\gv$ is conjugate to $\hh$. For
all these places $v$ we have
$$
\ker[(\pi_1(\Hbar)_\gv)\tors\to(\pi_1(\Gbar)_\gv)\tors]\neq 0,
$$
hence $\#(X(\kv)/G(\kv))> 1$, which contradicts to (iii). Thus
(iii)\implies(iv).
\bigskip

We prove that (ii)\implies(i). Indeed, by Borel's theorem \cite[Thm.~6.8]{B} the kernel
$$\ker[H^1(\kk,H)\to\prod_v H^1(\kk_v,H)]$$ is finite. 
It follows that all the fibers of the localization map
$$
X(\kk)/G(\kk)\to X(\AA)/G(\AA)
$$
are finite, Hence if the set $X(\AA)/G(\AA)$ is finite, then
$X(\kk)/G(\kk)$ is finite as well. Thus (ii)\implies(i).
\bigskip

All what is left to prove is that (i)\implies(ii), i.e that if the
set of $\kk$-orbits $X(\kk)/G(\kk)$ is finite, then the set of
adelic orbits $X(\AA)/G(\AA)$ is finite. For this end we consider
the group
$$
\ker_\ab(\kk,H\to G):=\ker[H^1_\ab(\kk,H)\to H^1_\ab(\kk,G)].
$$
\smallskip

Consider the following condition:
\medskip

\emph{(v) The group $\ker_\ab(\kk,H\to G)$ is finite.}
\medskip

We shall prove that (i)\implies(v) and (v)\implies(ii). This will
show that (i)\implies(ii).
\bigskip

We prove that (i)\implies(v).
Write
\begin{align*}
H^1(\kinf,H)&=\prod_{v\in\Vinf} H^1(\kv,H),\\
H^1_\ab(\kinf,H)&=\prod_{v\in\Vinf} H^1_\ab(\kv,H).
\end{align*}
Similarly we define
\begin{align*}
\ker(\kinf,H\to G)=\ker[H^1(\kinf,H)\to H^1&(\kinf, G)]=
\prod_{v\in\Vinf}\ker(\kk_v,H\to G),\\
\ker_\ab(\kinf,H\to G)=\ker[H^1_\ab(\kinf,H)\to H^1_\ab&(\kinf, G)]\\
&=\prod_{v\in\Vinf}\ker_\ab(\kk_v,H\to G).
\end{align*}

Set
$$
\kabf=\ker[\loc_\infty\colon\ker_\ab(\kk,H\to
G)\to\ker_\ab(\kinf,H\to G)].
$$
Since for $v\in\Vinf$ we have $H^1_\ab(\kk_v,H)\simeq
H^{-1}(\Gamma_v,\pi_1(\Hbar))\subset(\pi_1(\Hbar)_{\Gamma_v})\tors$\,,
we see that $H^1_\ab(\kv, H)$ is finite for every $v\in\Vinf$, and
therefore $\ker_\ab(\kinf,H\to G)$ is a finite group. It follows
that $\kabf$ is a subgroup of finite index in $\ker_\ab(\kk, H\to
G)$.

Consider the maps
\begin{align*}
\ab^1\colon &H^1(\kk,H)\to H^1_\ab(\kk,H),\\
\loc_\infty\colon &H^1(\kk,H)\to H^1(\kinf,H).
\end{align*}
By \cite[Thm. 5.12]{Bor98} these maps induce a canonical bijection
$$
H^1(\kk,H)\isoto H^1_\ab(\kk,H)\fibreprod_{H^1_\ab(\kinf,H)}
H^1(\kinf,H)
$$
(with a fiber product in the right hand side). This bijection is
functorial in $H$, hence we obtain a bijection
\begin{equation}\label{eq:fibre-product}
\ker(\kk,H\to G)\isoto \ker_\ab(\kk,H\to
G)\fibreprod_{\ker_\ab(\kinf,H\to G)} \ker(\kinf,H\to G).
\end{equation}

We define a map
$$
\kabf\to \ker_\ab(\kk,H\to G)\times \ker(\kinf,H\to G)
$$
by $x\mapsto (x,1)$. Since $\loc_\infty(x)=1$ for $x\in
\kabf\subset\ker_{\ab}(K,H\to G)$, we obtain from
\eqref{eq:fibre-product} an induced map $\kabf\to\ker(K,H\to G)$,
which is a section of the map
$$
\ab\colon \ker(\kk,H\to G)\to\ker_\ab(\kk,H\to G)
$$
over $\kabf$. Thus the group $\kabf$ embeds as a subset into the set
$\ker(\kk, H\to G)$.

By  assumption (i) $X(\kk)/G(\kk)$ is a finite set. Since we have a
canonical bijection
$$
X(\kk)/G(\kk)\simeq \ker(\kk,H\to G),
$$
we see that $\ker(\kk,H\to G)$ is finite. Since $\kabf$ embeds into
$\ker(\kk, H\to G)$, we see that $\kabf$ is finite. Since $\kabf$ is
a subgroup of finite index of $\ker_\ab(\kk,H\to G)$, we conclude
that $\ker_\ab(\kk,H\to G)$ is finite. Thus (i)\implies(v).
\bigskip

We prove that (v)\implies(ii). Here we use the abelian group
structure in $\ker_\ab(\kk, H\to G)$. We write $\bigoplus_v$ for
$\bigoplus_{v\in \sV}$.

We define a map $\mu\colon \bigoplus_v H^1_\ab(\kv,H)\to
(\pi_1(\Hbar)_\Gamma)\tors$ as the sum of the local maps $\mu_v$
defined in \eqref{eq:mu}. Namely, if $\xi_\AA=(\xi_v)\in \bigoplus_v
H^1_\ab(\kv,H)$, we set $\mu(\xi_\AA)=\sum_v \mu_v(\xi_v)$. The
sequence
\begin{equation}\label{eq:exact-sequence}
H^1_\ab(k,H)\labelto{\loc}\bigoplus_v H^1_\ab(\kv,H)\labelto{\mu}
(\pi_1(\Hbar)_\Gamma)\tors
\end{equation}
is exact, see \cite[Proof of Thm. 5.16]{Bor98}.

Exact sequence \eqref{eq:exact-sequence} is functorial in $H$, hence
the embedding $H\into G$ gives rise to a commutative diagram with
exact rows
\begin{equation}\label{eq:diagram}
\xymatrix{
H^1_\ab(\kk,H)\ar[d]\ar[r]^-\loc &\bigoplus_v H^1_\ab(\kv,H)\ar[d]\ar[r]^-{\mu} &(\pi_1(\Hbar)_\gg)\tors\ar[d]\\
H^1_\ab(\kk,G)\ar[r]^-\loc       &\bigoplus_v
H^1_\ab(\kv,G)\ar[r]^-{\mu}       &(\pi_1(\Gbar)_\gg)\tors }
\end{equation}
This diagram induces a homomorphism
$$
\varkappa\colon \bigoplus_v\ker_\ab(\kv,H\to G)\to
\ker[(\pi_1(\Hbar)_\gg)\tors\to(\pi_1(\Gbar)_\gg)\tors].
$$
The group $\ker[(\pi_1(\Hbar)_\gg)\tors\to(\pi_1(\Gbar)_\gg)\tors]$
is clearly finite. Set $k_0=\ker\;\varkappa$.

We define
$$
\Sha^1_\ab(\kk,H):=\ker\left[H^1_\ab(\kk,H)\to\prod_v
H^1_\ab(\kv,H)\right].
$$
We construct a homomorphism
$$
\psi\colon k_0\to\Sha^1_\ab(\kk,G)/i_*(\Sha^1_\ab(\kk,H))
$$
as follows. Let
$$
\xi_\AA\in k_0\subset \bigoplus_v\ker_\ab(\kv,H\to G)\subset
\bigoplus_v H^1_\ab(\kk_v,H).
$$
Since the top row of diagram \eqref{eq:diagram} is exact, we see
that $\xi_\AA$ comes from some $\xi\in H^1_\ab(\kk,H)$, and this
$\xi$ is defined up to addition of $\xi'\in \Sha^1_\ab(\kk,H)$. It
is clear from the diagram that the image of $\xi$ in
$H^1_\ab(\kk,G)$ is contained in $\Sha^1_\ab(\kk,G)$. Thus we obtain
a map $\psi\colon k_0\to\Sha^1_\ab(\kk,G)/i_*(\Sha^1_\ab(\kk,H))$.
It is easy to see that $\psi$ is a homomorphism. By Lemma
\ref{lem:Sha-finite}  below, the group $\Sha^1_\ab(\kk,G)$ is
finite. Hence the group $\Sha^1_\ab(\kk,G)/i_*(\Sha^1_\ab(\kk,H))$
is finite. Set $k_{00}=\ker\;\psi$. Using diagram chasing, we see
easily that $k_{00}$ is the image of $\ker_\ab(\kk,H\to G)$ in
$k_0$.

By assumption (v) the group $\ker_\ab(\kk,H\to G)$ is finite, hence
its image $k_{00}$ is finite. Since we have a homomorphism of
abelian groups $\psi$ from $k_0$ to the finite group
$\Sha^1_\ab(\kk,G)/i_*(\Sha^1_\ab(\kk,H))$ with finite kernel
$k_{00}$, we see that $k_0$ is finite. Since we have a homomorphism
of abelian groups $\varkappa$ from $\bigoplus_v\ker_\ab(\kv,H\to G)$
to the finite group
$\ker[(\pi_1(\Hbar)_\gg)\tors\to(\pi_1(\Gbar)_\gg)\tors]$ with
finite kernel $k_0$, we see that $\bigoplus_v\ker_\ab(\kv,H\to G)$
is finite. Since for all $v\in\sV_f$ we have  bijections
$$
\ab\colon \ker(\kv,H\to G)\isoto\ker_\ab(\kv,H\to G),
$$
we see that the set $\bigoplus_v\ker(\kv,H\to G)$ is finite. This
means that the set $X(\AA)/G(\AA)$ is finite. Thus (v)\implies(ii).

This completes the proof of Theorem \ref{thm:main-f-m-orbits} modulo
Lemma \ref{lem:Sha-finite}.
\end{proof}

\begin{lemma}\label{lem:Sha-finite}
Let $G$ be a connected linear algebraic group over a number field
$\kk$. Then the abelian group $\Sha_\ab^1(\kk,G)$ is finite.
\end{lemma}

\begin{proof} We give two proofs.

First proof: by \cite[Thm.~5.12]{Bor98} we have a canonical
bijection $\Sha^1(\kk,G)\to\Sha^1_\ab(\kk,G)$, and by Borel's
theorem, see \cite[Thm.~6.8]{B}, $\Sha^1(\kk,G)$ is
finite. Thus  $\Sha^1_\ab(\kk,G)$ is finite.

Second proof: We may and shall assume that $G$ is reductive. Let
$$
1\to S\to G'\to G\to 1
$$
be a flasque resolution of $G$, see \cite[\S3]{CT}. Here $G'$ is a
quasi-trivial reductive group and $S$ is a torus. Let
$P=(G')^{\textup{tor}}$, the biggest quotient torus of $G'$. Since
$G'$ is a quasi-trivial group, $P$ is a quasi-trivial torus. For any
field $F\supset K$ we have a canonical isomorphism
$$
H^1_\ab(F,G)\isoto\ker[H^2(F,S)\to H^2(F,P)],
$$
see \cite[App.~A]{CT}. Since $P$ is quasi-trivial, we have
$\Sha^2(\kk,P)=0$, and therefore
$$
\Sha^1_\ab(\kk,G)\simeq\Sha^2(\kk,S).
$$
It is known that the group $\Sha^2(\kk,S)$ is finite for any
$\kk$-group of multiplicative type $S$, see \cite[Ch.~I, Thm.
4.20(a)]{Milne}. Thus $\Sha_\ab^1(\kk,G)$ is finite. 
\end{proof}

%%%%%%%%%%%%%%%%%%%%%%%%%%%%%%%%%%%%%%%%%%%%%%%%%%%%%%%%%%%%%%

\subsection{Corollaries of Theorem \ref{thm:main-f-m-orbits}}\label{sec:2}

\begin{corollary}\label{cor:sc-finite}
Let $\kk$, $G$, $H$, and $X$ be as in \ref{subsec:Notation}. If
$\pi_1(\Hbar)=0$, then the set of orbits $X(\kk)/G(\kk)$ is finite.
\end{corollary}
\begin{proof}
Indeed, then $(\pi_1(\Hbar)_\hh)\tors=0$, hence the map
$(\pi_1(\Hbar)_\hh)\tors\to (\pi_1(\Gbar)_\hh)\tors$ is injective
(for any $\hh$). By Theorem \ref{thm:main-f-m-orbits} the set
$X(\kk)/G(\kk)$ is finite.
\end{proof}

\begin{corollary}\label{cor:pi1-isomorphism}
Let $\kk$, $G$, $H$, and $X$ be as in \ref{subsec:Notation}. If the
map $\pi_1(\Hbar)\to \pi_1(\Gbar)$ is an isomorphism, then the set
of orbits $X(\kk)/G(\kk)$ is finite.
\end{corollary}
\begin{proof}
Indeed, then the map $(\pi_1(\Hbar)_\hh)\tors\to
(\pi_1(\Gbar)_\hh)\tors$ is  an isomorphism, hence injective (for
any $\hh$). By Theorem \ref{thm:main-f-m-orbits} the set
$X(\kk)/G(\kk)$ is finite.
\end{proof}

\begin{corollary}\label{cor:1-app}
Let $\kk,\ G,\ H,\ X$ be as in \ref{subsec:Notation}. Assume the set
$X(\kk)/G(\kk)$ is finite. Then the induced homomorphism
$i_*\colon\pi_1(\Hbar)\tors\to\pi_1(\Gbar)\tors$ is injective.
\end{corollary}

\begin{proof}
Since the set $X(\kk)/G(\kk)$ is finite, by Theorem
\ref{thm:main-f-m-orbits} the map
$$
i_*\colon (\pi_1(\Hbar)_\hh)\tors\to (\pi_1(\Gbar)_\hh)\tors
$$
is injective for \emph{any} cyclic subgroup  $\hh\subset\gg$, in
particular for $\hh=\{1\}$. Thus the map $\pi_1(\Hbar)\tors\to
\pi_1(\Gbar)\tors$ is injective.
\end{proof}

\begin{corollary}\label{cor:1-ss-app}
Let $\kk,\ G,\ H,\ X$ be as in \ref{subsec:Notation}. Assume the set
$X(\kk)/G(\kk)$ is finite. If  $H$ has no nontrivial $\kbar$-characters (e.g.
semisimple), then the homomorphism
$i_*\colon\pi_1(\Hbar)\to\pi_1(\Gbar)$ is injective.
\end{corollary}
\begin{proof}
Indeed, since $H$ has no nontrivial $\kbar$-characters, we see that
$\pi_1(\Hbar)$ is finite, hence $\pi_1(\Hbar)\tors=\pi_1(\Hbar)$,
and we apply Corollary \ref{cor:1-app}.
\end{proof}

\begin{corollary}\label{cor:sc}
Let $\kk,\ G,\ H,\ X$ be as in \ref{subsec:Notation}. Assume that
both $G$ and $H$ have no nontrivial $\kbar$-characters (e.g. they both are
semisimple) and assume that $\pi_1(\Gbar)=0$. Then $X(\kk)/G(\kk)$
is finite if and only if $\pi_1(\Hbar)=0$.
\end{corollary}
\begin{proof}
If $\pi(\Hbar)=0$, then by Corollary \ref{cor:sc-finite}
$X(\kk)/G(\kk)$ is finite. Conversely, assume that $X(\kk)/G(\kk)$
is finite. By Corollary \ref{cor:1-ss-app} the homomorphism
$\pi_1(\Hbar)\to\pi_1(\Gbar)$ is injective, but $\pi_1(\Gbar)=0$,
hence $\pi_1(\Hbar)=0$.
\end{proof}

\begin{subsec}{}\label{subsec:semisimple}
Let $G$ be a connected semisimple $\kk$-group. We  say that $G$ is
an inner form if $G$ is an inner form of a $\kk$-split group. If $G$
is an inner form, then the Galois group  $\Gal(\kbar/\kk)$ acts on
$\pi_1(\Gbar)$ trivially. Indeed, for a $\kk$-split group $G$ this
follows from the definition of $\pi_1(\Gbar)$, and an inner twisting
does not change the Galois module $\pi_1(\Gbar)$.
\end{subsec}

\begin{corollary}\label{cor:2-app}
Let $\kk,\ G,\ H$ and  $X$ be as in \ref{subsec:Notation}. Assume
that the Galois group $\Gal(\kbar/\kk)$ acts on $\pi_1(\Gbar)$
trivially. If the homomorphism
$i_*\colon\pi_1(\Hbar)\to\pi_1(\Gbar)$ is injective, then the set
$X(\kk)/G(\kk)$ is finite.
\end{corollary}
\begin{proof}
Since $\pi_1(\Hbar)$ injects into $\pi_1(\Gbar)$, we see that
$\Gal(\kbar/\kk)$ acts also  on $\pi_1(\Hbar)$ trivially. Thus
$\geff=\{\id\}$, hence  the only cyclic subgroup $\hh\subset\geff$
is $\hh=\{\id\}$. We see that the homomorphism
$$
i_*\colon \pi_1(\Hbar)_\hh\to \pi_1(\Gbar)_\hh
$$
is injective, hence the homomorphism
$$
i_*\colon (\pi_1(\Hbar)_\hh)\tors\to (\pi_1(\Gbar)_\hh)\tors
$$
is injective, and the corollary follows from Theorem
\ref{thm:main-f-m-orbits}.
\end{proof}
%%%%%%%%%%%%%%%%%%%%%%%%%%%%%%%%%%%%%%%%%%%%%%%%%%%%%%%%%%%%%%%%%%%%

\subsection{Semisimple groups}\label{sec:3}

In this section $\kk$ is an \emph{algebraically closed field} of
characteristic 0. We consider pairs $(G,H)$, where $H$ is a
connected semisimple $\kk$-subgroup of a connected semisimple
$\kk$-group $G$. We find conditions under which the map
$\pi_1(H)\to\pi_1(G)$ is injective.

\begin{subsec}\label{subsec:ss-groups}
Let $H\subset G$ be connected semisimple $\kk$-groups. In this case
the both $\pi_1(H)$ and $\pi_1(G)$ are finite. Let $i\colon H\to G$
be the inclusion homomorphism. Consider  the map $i\sc\colon H\sc\to
G\sc$, where $G\sc$ is the universal covering of $G$. Set
$H'=i\sc(H\sc)\subset G\sc$. Let $T_H\subset H$ be a maximal torus,
and let $T_G\subset G$ be a maximal torus containing $T_H$. Let
$T_{H\sc}\subset H\sc$, $T_{H'}\subset H'$, and $T_{G\sc}\subset
G\sc$ be the  maximal tori corresponding to $T_H$ and $T_G$. For a
$\kk$-torus $T$ let $\xx(T)$ denote the cocharacter group of $T$,
i.e. $\xx(T)=\Hom(\GG_{m,\kk},T)$. We have  canonical homomorphisms
$T_{H\sc}\to T_{H'}\to  T_{G\sc}$ and the induced homomorphisms
$\xx(T_{H\sc})\to\xx(T_{H'})\to \xx(T_{G\sc})$.
\end{subsec}

\begin{lemma}\label{lem:tori}
Let $H\subset G$ be connected semisimple $\kk$-groups. With the
notation of Subsection \ref{subsec:ss-groups} we have canonical
isomorphisms
$$
\pi_1(H')\simeq\ker[\pi_1(H)\to\pi_1(G)]
\simeq\coker[\xx(T_{H\sc})\to \xx(T_{G\sc})]\tors.
$$
where $\tors$ denotes the torsion subgroup (of the cokernel).
\end{lemma}

\begin{proof}
Consider the following commutative diagram with exact rows:
$$
\xymatrix{
        &0\ar[d]                   &0\ar[d]              &k\ar[d]  \\
0\ar[r] &\xx(T_{H\sc})\ar[d]\ar[r] &\xx(T_H)\ar[d]\ar[r] &\pi_1(H)\ar[d]\ar[r] &0 \\
0\ar[r] &\xx(T_{G\sc})\ar[d]\ar[r] &\xx(T_G)\ar[d]\ar[r] &\pi_1(G)\ar[r]       &0 \\
        &C\sc                      &C
}
$$
where $k$ is the kernel and $C\sc$ and $C$ are the cokernels of the
corresponding homomorphisms. By the snake lemma we have an exact
sequence
\begin{equation}\label{eq:snake}
0\to k\to C\sc\to C.
\end{equation}
Since $\pi_1(H)$ is finite, clearly $k$ is finite. Since $T_H$
embeds into $T_G$, the group $C$ has no nonzero torsion.
From exact sequence \eqref{eq:snake} we obtain an isomorphism
$k\isoto (C\sc)\tors$,
i.e. an isomorphism
$$
\ker[\pi_1(H)\to\pi_1(G)] \simeq \coker[\xx(T_{H\sc})\to
\xx(T_{G\sc})]\tors,
$$

Since the injective homomorphism $\xx(T_{H\sc})\to \xx(T_{G\sc})$
factorizes as a composition of injective homomorphisms
$\xx(T_{H\sc})\to\xx(T_{H'})\to \xx(T_{G\sc})$, we obtain a short
exact sequence
\begin{align*}
0\to \coker[\xx(T_{H\sc})\to\xx(T_{H'})]
 \to \coker[\xx(T_{H\sc})\to \xx(T_{G\sc})]
\\ \to \coker[\xx(T_{H'})\to \xx(T_{G\sc})]\to 0.
\end{align*}
Since $T_{H'}$ embeds into $T_{G\sc}$, we have
$\coker[\xx(T_{H'})\to \xx(T_{G\sc})]\tors=0$, and therefore we
obtain an isomorphism
$$
\coker[\xx(T_{H\sc})\to\xx(T_{H'})]\tors\simeq
  \coker[\xx(T_{H\sc})\to \xx(T_{G\sc})]\tors.
$$
But $\coker[\xx(T_{H\sc})\to\xx(T_{H'})]\tors=
\coker[\xx(T_{H\sc})\to\xx(T_{H'})]=\pi_1(H')$. Thus we obtain an
isomorphism
$$
\pi_1(H')\simeq \coker[\xx(T_{H\sc})\to \xx(T_{G\sc})]\tors.
$$
This completes the proof of Lemma \ref{lem:tori}.
\end{proof}

\begin{corollary}\label{cor:H'}
With the assumptions and notation of Lemma \ref{lem:tori}, the
following assertions are equivalent:

(i) $H'$ is simply connected;

(ii) The homomorphism $\pi_1(H)\to\pi_1(G)$ is injective;

(iii) The group $\coker[\xx(T_{H\sc})\to \xx(T_{G\sc})]$ has no
torsion.
\end{corollary}

%%%%%%%%%%%%%%%%%%%%%%%%%%%%%%%%%%%%%%%%%%%%%%%%%%%%%%%%%%%%%%%%%%%%%%%%%%%%%%%

\subsection{Symmetric pairs over an algebraically closed field}\label{sec:4}

In this section we assume that $\kk$ is  an \emph{algebraically
closed} field of characteristic 0. We consider symmetric pairs
$(G,H)$ over $\kk$, where $G$ is a connected, simply connected, almost simple
$\kk$-group, and $H\subset G$ is a connected, semisimple, closed subgroup. Recall that
``symmetric pair'' means that $H$ is the group of invariants $G^\theta$
for some involutive automorphism $\theta$ of $G$.
Symmetric pairs  $(G,H)$ (or $(G,\theta)$) were classified by E.
Cartan. We shall use the unified description of symmetric pairs due
to V. Kac, see \cite{Helgason01} and \cite{OV}. A symmetric pair
$(G,H)$ with semisimple $H$ corresponds to an affine Dynkin diagram
$D$ and a vertex $s$ of $D$,
see \cite[Table 7]{OV}.
We give a list
of all symmetric pairs $(G,H)$ from \cite[Table 7]{OV}  with simply connected almost simple
$G$  for which $H$ is  simply connected.
We give only the isomorphism classes of $G$ and $H$.
For a description of the embedding $H\into G$ see \cite[X.2.3]{Helgason01} for $G$ classical,
or \cite[Table 7]{OV} for all cases (in terms of affine Dynkin diagrams).

\begin{theorem}\label{cor:Kac-sc-table}
Let $(G,H)$ be a symmetric pair over 
an algebraically closed field $\kk$ of characteristic 0, 
where $G$ be a  connected, {\bf simply connected}, almost simple $\kk$-group,
and $H\subset G$ a connected, semisimple, closed subgroup. Then $H$ is
simply connected for the symmetric pairs $(G,H)$ in the list below,
and $\pi_1(H)=\ZZ/2\ZZ$ for the symmetric pairs $(G,H)$ not in the list.
\smallskip

\noi$(\textbf{\bf A II})$ $G=\SL_{2n}$, $H=\Sp_n$ $(n\ge 3)$.

\noi$(\textbf{\bf C II})$ $G=\Sp_{p+q}$, $H=\Sp_p\times \Sp_q$
$(1\le p\le q)$.

\noi$(\textbf{\bf BD I}(2l,1))$ $G=\Spin_{2l+1}$, $H=\Spin_{2l}$
$(l\ge 3)$.

\noi$(\textbf{\bf BD I}(2l-1,1))$  $G=\Spin_{2l}$, $H=\Spin_{2l-1}$
$(l\ge 3)$.

\noi$(\textbf{\bf E IV})$ $G=E_6$, $H=F_4$.

\noi$(\textbf{\bf F II})$ $G=F_4$, $H=\Spin_9$.

\bigskip
\end{theorem}

\begin{proof}
We consider two cases.
\medskip

\noindent (i) $\theta$ is an inner automorphism.

In this case the affine Dynkin diagram $D$
which is used in the description of the symmetric pair $(G,H)$, cf. \cite[Table 7]{OV},
is the \emph{extended Dynkin diagram} of $G$. Let $T_H$ be
a maximal torus of $H$, and let $T_G$ be a maximal torus of $G$
containing $T_H$. Let $T_{H\sc}$ be the corresponding maximal torus
of the universal covering $H\sc$ of $H$.
\newcommand{\XX}{{\textsf{\upshape X}^*}}
Let $\XX(T_G):=\Hom(T_G,\GG_m)$ be the character group of $T_G$. Set
$V=\XX(T_G)\otimes_\ZZ \RR$. Let $R=R(G,T_G)\subset\XX(T_G)\subset
V$ be the root system of $G$. Let
$\alpha_0,\alpha_1,\dots,\alpha_l\in R$ be the  roots corresponding
to the vertices of $D$, where $\alpha_0$ is the lowest root. Then
\begin{equation}\label{eq:main}
\sum_{i=0}^l a_i\alpha_i=0,
\end{equation}
 where $a_i\in\ZZ$, $a_0=1$.
The distinguished vertex $s$ of $D$ corresponds to some root $\alpha_k$,
and $a_k=2$ (see \cite[Ch.~X \S 5]{Helgason01}, \cite[Ch.~5 \S1.5,
Problem 38]{OV}).

Let $R^\vee\subset\xx(T_G)$ denote the dual root system. For every
$\alpha\in R$ let $\alpha^\vee\in R^\vee$ be the corresponding
coroot. The  coroots $\alpha_1^\vee,\dots,\alpha_l^\vee$ constitute
a basis of $\xx(T_G)$, hence
$\alpha_0^\vee,\alpha_1^\vee,\dots,\alpha_l^\vee$ generate
$\xx(T_G)$. The diagram $D-\{s\}$ is the Dynkin diagram of $H$, and
the coroots $\alpha_i^\vee,\ (i\neq k)$ constitute a basis of
$\xx(T_{H\sc})$.

Let $W=W(R)$ denote the Weyl group. Choose a $W$-invariant scalar
product $(\ ,\,)$ in $V$. We can embed $R^\vee$ into $V$ by
$$
\alpha^\vee=\frac{2\alpha}{(\alpha,\alpha)}.
$$

We consider 4 subcases.

\medskip

(i)(a) Suppose that $D$ has no multiple edges. Then all the roots
$\beta\in R$ are of the same length, and we can normalize the scalar
product such that $(\beta,\beta)=2$, hence $\beta^\vee=\beta$, for
all $\beta\in R$.
Now it follows from \eqref{eq:main} that
$$
\sum_{i=0}^l a_i\alpha_i^\vee=0.
$$
Recall that $a_i\in\ZZ$, $a_0=1$, and $a_k=2$. We see that
$\alpha_k^\vee\in\frac{1}{2}\xx(T_{H\sc})$, but
$\alpha_k^\vee\notin\xx(T_{H\sc})$. Thus
$\coker[\xx(T_{H\sc})\to\xx(T_G)]=\ZZ/2\ZZ$, and by Lemma
\ref{lem:tori} $\pi_1(H)=\ZZ/2\ZZ$.
\medskip

(i)(b) Suppose that $D$ has double edges and $\alpha_k$ is a
\emph{short} root. We can normalize the scalar product such that for
any long root $\beta$ we have $(\beta, \beta)=2$, hence
$\beta^\vee=\beta$. Then for any short root $\gamma$ we have
$(\gamma,\gamma)=1$, hence $\gamma^\vee=2\gamma$.
Now it follows from \eqref{eq:main} that
\begin{equation}\label{eq:a-prime}
\sum_{i=0}^l a'_i\alpha_i^\vee=0,
\end{equation}
where $a'_i=a_i$ when $\alpha_i$ is long, and $a'_i=a_i/2$ when
$\alpha_i$ is short.

Since $\alpha_0$ is the lowest root, it is long. Since $a_0=1$, we
obtain that $a'_0=1$. Thus \eqref{eq:a-prime} gives
\begin{equation*}
\alpha_0^\vee=\sum_{i=1}^l -a'_i\alpha_i^\vee.
\end{equation*}
Since $\alpha_0^\vee\in R^\vee$ and $(\alpha_i^\vee)_{i=1,\dots,l}$
is a basis of $R^\vee$, we see that  $a'_i\in\ZZ$ for all $i$. Since
$a_k=2$ and $\alpha_k$ is short, we see that $a'_k=1$. Thus
$\alpha_k^\vee$ is a linear combination of $\alpha_i^\vee\ (i\neq
k)$ with  coefficients in $\ZZ$. We see that the map
$\xx(T_{H\sc})\to\xx(T_G)$ is surjective, and by Corollary
\ref{cor:H'} the group $H$ is simply connected.
\medskip

(i)(c) Suppose that $D$ has double edges and $\alpha_k$ is a
\emph{long} root. As in (i)(b),  from \eqref{eq:main} we obtain the
relation \eqref{eq:a-prime}
with $a'_i=a_i$ when $\alpha_i$ is long, and $a'_i=a_i/2$ when
$\alpha_i$ is short. Again $a'_0=1$ and $a'_i\in\ZZ$ for all $i$.
Since  $a_k=2$ and now $\alpha_k$ is long, we see that $a'_k=2$.
As in (i)(a), we see that
$\alpha_k^\vee\in\frac{1}{2}\xx(T_{H\sc})$, but
$\alpha_k^\vee\notin\xx(T_{H\sc})$. Thus
$\coker[\xx(T_{H\sc})\to\xx(T_G)]=\ZZ/2\ZZ$, and by Lemma
\ref{lem:tori}   $\pi_1(H)=\ZZ/2\ZZ$.
\medskip

(i)(d) Suppose that $D$ has a triple edge (type $G_2$). We have
$k=1$,
\begin{equation}\label{eq:G2}
\alpha_0+2\alpha_1+3\alpha_2=0,
\end{equation}
(see \cite[Table 7.I]{OV}), where
$$
(\alpha_0,\alpha_0)=3,\ (\alpha_1,\alpha_1)=3, \text{ and }
(\alpha_2,\alpha_2)=1.$$ Then
$$
\alpha_0^\vee=\frac{2}{3}\alpha_0,\
\alpha_1^\vee=\frac{2}{3}\alpha_1,\ \alpha_2^\vee=2\alpha_2.
$$
From \eqref{eq:G2} we obtain
$$
\alpha_0^\vee+2\alpha_1^\vee+\alpha_2^\vee=0.
$$
Similarly to (i)(a),  we see that
 $\alpha_1^\vee\in\frac{1}{2}\xx(T_{H\sc})$,
but $\alpha_1^\vee\notin\xx(T_{H\sc})$. It follows that
$\coker[\xx(T_{H\sc})\to\xx(T_G)]=\ZZ/2\ZZ$, and therefore
$\pi_1(H)=\ZZ/2\ZZ$.
\medskip

We obtain the following list of pairs $(G,H)$ with simply connected
$H$ in the case when $\theta$ is inner:

\noi(\textbf{C II}) $G=\Sp_{p+q}$, $H=\Sp_p\times \Sp_q$ $(1\le p\le
q)$.

\noi(\textbf{BD I}($2l$,1)) $G=\Spin_{2l+1}$, $H=\Spin_{2l}$ $(l\ge
3)$.

\noi(\textbf{F II}) $G=F_4$, $H=\Spin_9$.
\bigskip

\noindent Case (ii): $\theta$ is an outer automorphism. We use
case-by-case consideration.

When $G$ is a classical group and $\theta$ is outer, we see from
\cite[Table 7.III]{OV} that $H$ is simply connected only in the
following cases:

\noi(\textbf{A II}) $G=\SL_{2n}$, $H=\Sp_n$ $(n\ge 3)$.

\noi(\textbf{BD I}($2l-1$,1))  $G=\Spin_{2l}$, $H=\Spin_{2l-1}$
$(l\ge 3)$.

These are exactly the cases when $D$ has a double edge and
$\alpha_k$ is a short root.

\medskip

We list all the other classical cases with $\theta$ outer:

\noi(\textbf{A I}) $G=\SL_{n}$, $H=\SO_{n}$ $(n\ge 3,\ n\neq 4)$.

\noi(\textbf{BD I}$(2p+1, 2q+1)$) $G=\Spin_{2p+2q+2}$,
    $H=(\Spin_{2p+1}\times\Spin_{2q+1})/\mu_2$  $(1\le p\le q)$.

In  these cases $\pi_1(H)=\ZZ/2\ZZ$.
\medskip

We must treat the case $D=E_6^{(2)}$, see \cite[Table 7.III]{OV}
(this diagram has a double edge). Then either $H=F_4$ or $H=C_4$.

When $H=F_4$, clearly $H$ is simply connected. In this case
$\alpha_k$ is a short root.

When $H$ is of the type $C_4$, the restriction of the adjoint
representation of $G$ to $H$ is the direct sum
$\Lie(G)=\Lie(H)\oplus\pp$ of two irreducible representations. Here
the  the representation of $H=C_4$ in $\pp$ is a subrepresentation
of the  representation of $C_4$ in $\bigwedge^4R$, where $R$ is the
standard 8-dimensional representation of $C_4$. We see that the
central element $-1\in H\sc(K)= \Sp_4(K)$ acts trivially on
$\Lie(H)$ and on $\pp$, hence the image of $-1$ in the adjoint group
$G\ad$ is 1. Since $\ker[G\to G\ad]$ is of order 3, we see that the
image of this element in  $G$ is 1. Thus $\pi_1(H)=\ZZ/2\ZZ$. In
this case $\alpha_k$ is a long root.

We obtain the following list of symmetric pairs $(G,H)$ with simply
connected $H$ in the case $E_6$:

\noi(\textbf{E IV}) $G=E_6$, $H=F_4$.

This completes the proof of Theorem \ref{cor:Kac-sc-table}.
\end{proof}

\begin{corollary}\label{thm:Kac-sc}
Let $G$ be a connected, simply connected, almost  simple $\kk$-group over an
algebraically closed field  $\kk$ of characteristic 0, and let
$H\subset G$ be a connected, symmetric, semisimple $\kk$-subgroup. Assume that
the symmetric pair $(G,H)$ corresponds to $(D,s)$ as above. If $D$
has a double edge and $s$ corresponds to a short root, then $H$ is
simply connected; otherwise $\pi_1(H)=\ZZ/2\ZZ$.
\end{corollary}

%%%%%%%%%%%%%%%%%%%%%%%%%%%%%%%%%%%%%%%%%%%%%%%%%%%%%%%%%%%%%%%%%%%%%%

\subsection{Symmetric pairs over a number field}\label{sec:5}
Clearly $(G,H)$ is a symmetric pair if and only if $(\Gbar,\Hbar)$ is a symmetric pair.

First we consider symmetric homogeneous spaces $X=H\backslash G$
with $G$ \emph{simply connected}.
It turns out that the answer to the question whether $X(\kk)$ has finitely many $G(\kk)$-orbits
depends only on the isomorphism class of the pair $(\Gbar,\Hbar))$.

\begin{theorem}\label{thm:symmetric-sc}
Let $\kk$ be a number field,
and $(G,H)$ a symmetric pair over $\kk$ with  
connected, {\bf simply connected}, absolutely almost simple $G$
and connected semisimple $H$.
The symmetric homogeneous space $X=H\backslash G$ over $\kk$
has finitely many $G(\kk)$-orbits if and only if 
the pair $(\Gbar,\Hbar)$ is in the list of Theorem  \ref{cor:Kac-sc-table}, 
or, which is the same, if and only if the symmetric pair $(G,H)$ is in the list below:
\smallskip

\noi$(\textbf{\bf A II})$ $G$ is a $\kk$-form of $\SL_{2n}$, $H$ is
a $\kk$-form of $\Sp_n$ $(n\ge 3)$.

\noi$(\textbf{\bf C II})$ $G$ is a $\kk$-form of $\Sp_{p+q}$, $H$ is
a $\kk$-form of $\Sp_p\times \Sp_q$ $(1\le p\le q)$.

\noi$(\textbf{\bf BD I}(2l,1))$ $G$ is a $\kk$-form of
$\Spin_{2l+1}$, $H$ is a $\kk$-form of $\Spin_{2l}$ $(l\ge 3)$.

\noi$(\textbf{\bf BD I}(2l-1,1))$  $G$ is a $\kk$-form of
$\Spin_{2l}$, $H$ is a $\kk$-form of $\Spin_{2l-1}$ $(l\ge 3)$.

\noi$(\textbf{\bf E IV})$ $G$ is a $\kk$-form of the simply connected group of type $E_6$, $H$ is a $\kk$-form of $F_4$.

\noi$(\textbf{\bf F II})$ $G$ is a $\kk$-form of $F_4$, $H$ is a
$\kk$-form of $\Spin_9$.
\end{theorem}

\begin{proof}
We have $\pi_1(\Gbar)=0$. By Corollary \ref{cor:sc} the set of
orbits $X(\kk)/G(\kk)$ is finite if and only if $\pi_1(\Hbar)=0$,
i.e $H$ is simply connected. The symmetric pairs $(\Gbar,\Hbar)$
over $\kbar$ with simply connected $\Hbar$ were listed in Theorem
\ref{cor:Kac-sc-table}. The list of Theorem \ref{thm:symmetric-sc}
is exactly the list of Theorem \ref{cor:Kac-sc-table}.
\end{proof}

Now we consider symmetric homogeneous spaces $X=H\backslash G$ with
$G$ \emph{adjoint}.
In this case the answer to the question whether $X(\kk)$ has finitely many $G(\kk)$-orbits
depends  on the isomorphism class of the pair $(\Gbar,\Hbar))$ 
and may also depend on whether $G$ is an inner or outer form of a $\kk$-split group.

\begin{theorem}\label{thm:main-finite-adjoint}
Let $\kk$ be a number field. 
Let $(G,H)$ be a symmetric pair over $\kk$
with connected, {\bf adjoint}, absolutely simple  $G$
and connected semisimple $H$.
The symmetric homogeneous space $X=H\backslash G$ over $\kk$ 
has finitely many $G(\kk)$-orbits 
if and only if the symmetric pair $(G,H)$ is in the following list:
\smallskip

\noi$(\textbf{\bf A II})$ $G$ is a $\kk$-form of $\PSL_{2n}$, $H$ is a
$\kk$-form of $\PSp_n$ $(n\ge 3)$, where either $n$ is odd or  $G$ is an
{\bf inner} form of a $\kk$-split group.

\noi$(\textbf{\bf C II})$ $G$ is a $\kk$-form of $\PSp_{p+q}$, $H$ is a
$\kk$-form of $(\Sp_p\times \Sp_q)/\mu_2$ $(1\le p\le q)$.

\noi$(\textbf{\bf BD I}(2l,1))$ $G$ is a $\kk$-form of $\SO_{2l+1}$, $H$
is a $\kk$-form of $\SO_{2l}$ $(l\ge 3)$.

\noi$(\textbf{\bf BD I}(2l-1,1))$  $G$ is a $\kk$-form of
$\PSO_{2l}$ which is an {\bf inner} form of a $\kk$-split group, $H$ is a $\kk$-form of $\SO_{2l-1}$ $(l\ge 3)$.

\noi$(\textbf{\bf E IV})$ $G$ is a $\kk$-form of the adjoint group of type $E_6$, $H$ is a
$\kk$-form of $F_4$.

\noi$(\textbf{\bf F II})$ $G$ is a $\kk$-form of $F_4$, $H$ is a $\kk$-form of
$\Spin_9$.
\end{theorem}

\begin{proof}
First assume that $X(\kk)/G(\kk)$ is finite. Since $H$ is
semisimple, by Corollary \ref{cor:1-ss-app} the homomorphism
$\pi_1(\Hbar)\to\pi_1(\Gbar)$ is injective. Let $H'$ denote the
image of $i\sc\colon H\sc\to G\sc$. By Corollary \ref{cor:H'}  $H'$
is simply connected. Thus $(\Gbar\sc,\Hbar')$ is a symmetric pair
with simply connected groups
 $\Gbar\sc$ and $\Hbar'$.
Such pairs were listed in Theorem \ref{cor:Kac-sc-table}. Thus we
obtain that the pair  $(\Gbar, \Hbar')$ is from the list of Theorem
\ref{cor:Kac-sc-table}, hence $(G,H)$ is from the following list:

\noi$(\textbf{\bf A II})$ $G$ is a $\kk$-form of $\PSL_{2n}$, $H$ is a
$\kk$-form of $\PSp_n$ $(n\ge 3)$

\noi$(\textbf{\bf C II})$ $G$ is a $\kk$-form of $\PSp_{p+q}$, $H$ is a
$\kk$-form of $(\Sp_p\times \Sp_q)/\mu_2$ $(1\le p\le q)$.

\noi$(\textbf{\bf BD I}(2l,1))$ $G$ is a $\kk$-form of $\SO_{2l+1}$, $H$
is a $\kk$-form of $\SO_{2l}$ $(l\ge 3)$.

\noi$(\textbf{\bf BD I}(2l-1,1))$  $G$ is a $\kk$-form of $\PSO_{2l}$, $H$
is a $\kk$-form of $\SO_{2l-1}$ $(l\ge 3)$.

\noi$(\textbf{\bf E IV})$ $G$ is a $\kk$-form of the adjoint group of type $E_6$, $H$ is a
$\kk$-form of $F_4$.

\noi$(\textbf{\bf F II})$ $G$ is a $\kk$-form of $F_4$, $H$ is a $\kk$-form of
$\Spin_9$.
\bigskip

Conversely, let us check, for which  $(G,H)$ from this list the set
of orbits $X(\kk)/G(\kk)$ is finite.

If $G$ is an \emph{inner} form, then $\Gal(\kbar/\kk)$ acts on
$\pi_1(\Gbar)$ trivially, see Subsection \ref{subsec:semisimple},
and by Corollary \ref{cor:2-app} the set of orbits $X(\kk)/G(\kk)$
is finite. Thus in the cases $(\textbf{\bf C II})$, $(\textbf{\bf BD
I}(2l,1))$, and $(\textbf{\bf F II})$ the set $X(\kk)/G(\kk)$ is
finite, because any form of $G$ is inner in these cases.

In the case $(\textbf{\bf E IV})$ we have $\pi_1(\Hbar)=0$, and by
Corollary \ref{cor:sc-finite} the set $X(\kk)/G(\kk)$ is finite
(when $G$ is an inner form or an outer form).

What is left is to consider the cases $(\textbf{\bf A II})$ and
$(\textbf{\bf BD I}(2l-1,1))$ with outer forms of $G$.

\medskip

We consider the case (\textbf{A II}). Then $\Gbar=\PSL_{2n}$,
$\pi_1(\Gbar)=\ZZ/2n\ZZ$, $\Hbar=\PSp_n$, $\pi_1(\Hbar)=\ZZ/2\ZZ$.
The embedding $\pi_1(\Hbar)\into\pi_1(\Gbar)$ is given by
$$
\overline{1}\mapsto\overline{n},\text{\quad where }
\overline{1}=1+2\ZZ\in\ZZ/2\ZZ,\ \overline{n}=n+2n\ZZ\in \ZZ/2n\ZZ.
$$
Since $G$ is an outer form, we have $\gg=\{1,\sigma\}$, where the
nontrivial element $\sigma$ of $\gg$ is of order 2 and acts on
$\pi_1(\Gbar)=\ZZ/2n\ZZ$ by ${}^\sigma x=-x$. We see that ${}^\sigma
x-x=-2x$. Thus the kernel of the canonical map
$\pi_1(\Gbar)\to\pi_1(\Gbar)_\gg$ is the subset
$$
\{\overline{2k}\subset \ZZ/2n\ZZ\ |\ k\in\ZZ\}.
$$
We see that the element $\overline{n}$ lies in this kernel if and
only if $n$ is even.

If $n$ is even, then the map $\pi_1(\Hbar)_\gg\to \pi_1(\Gbar)_\gg$
is the zero map. In other words, for $\hh=\gg$ the map
$\pi_1(\Hbar)_\hh\to \pi_1(\Gbar)_\hh$ is not injective. By Theorem
\ref{thm:main-f-m-orbits} the set $X(\kk)/G(\kk)$ is infinite.

If $n$ is odd, then the map $\pi_1(\Hbar)_\gg\to \pi_1(\Gbar)_\gg$
is injective. In other words, for $\hh=\gg$ the map
$\pi_1(\Hbar)_\hh\to \pi_1(\Gbar)_\hh$ is injective. On the other
hand, the map $\pi_1(\Hbar)\to \pi_1(\Gbar)$ is injective (for any
$n$). In other words, for $\hh=\{1\}$ the map $\pi_1(\Hbar)_\hh\to
\pi_1(\Gbar)_\hh$ is injective as well. By Theorem
\ref{thm:main-f-m-orbits} the set $X(\kk)/G(\kk)$ is finite.

\medskip

We consider the case  $(\textbf{\bf BD I}(2l-1,1))$ when $G$ is an
outer form. We show that $X(\kk)/G(\kk)$ is infinite in this case.

In this case $G$ is a form of $D_l$ and $H$ is a form of $B_{l-1}$.
We have $\pi_1(\Hbar)=\ZZ/2\ZZ$, and $\pi_1(\Gbar)$ is $\ZZ/4\ZZ$
when $l$ is odd and $\ZZ/2\ZZ\times \ZZ/2\ZZ$ when $l$ is even. The
group $\pi_1(\Hbar)$ embeds into $\pi_1(\Gbar)$, and the image is a
$\Gal(\kbar/\kk)$-invariant subgroup of order 2.

We observe that in the case $l=4$, $G$ does not come from triality.
Indeed, if $G$ comes from triality, then $\Gal(\kbar/\kk)$ acts
transitively on the set of nonzero elements of $\pi_1(\Gbar)$, and
therefore $\pi_1(\Gbar)$ cannot have a $\Gal(\kbar/\kk)$-invariant
subgroup of order 2.

We see that for any $l$, the group $\gg$ is of order 2. We write
$\gg=\{1,\sigma\}$.

Assume that $l$ is odd. Then $\pi_1(\Gbar)=\ZZ/4\ZZ$, and $\sigma$
acts on $\pi_1(\Gbar)$ by ${}^\sigma x=-x$. Arguing as in the case
(\textbf{A II}) with even $n$, we see that $X(\kk)/G(\kk)$ is
infinite.

Assume that $l$ is even. Denote the elements of
$\pi_1(\Gbar)=\ZZ/2\ZZ\times \ZZ/2\ZZ$ by $0,a,b,c$. We may assume
that $\sigma$ permutes $a$ and $b$ and fixes $c$. Then clearly the
image of $\pi_1(\Hbar)$ is $\{0,c\}$. Since $a-{}^\sigma a=a+b=c$,
we see that the map $\pi_1(\Hbar)_\gg\to \pi_1(\Gbar)_\gg$ is the
zero map, hence it is not injective. By Theorem
\ref{thm:main-f-m-orbits} the set of orbits $X(\kk)/G(\kk)$ is
infinite.

This completes the proof of Theorem \ref{thm:main-finite-adjoint}.
\end{proof}

\begin{theorem}\label{thm:cor-orbits}
Let  $X=H\backslash G$, where $(G,H)$ is a symmetric pair as in
Theorem \ref{thm:symmetric-sc} or as in Theorem
\ref{thm:main-finite-adjoint}. Then:

(i) For any finite place $v$ of $\kk$, the group $G(\kk_v)$ acts on
$X(\kk_v)$ transitively.

(ii) Write $\kk_\infty=\prod_{v\in\sV_\infty}\kk_v$ (so that
$X(\kk_\infty)=\prod_{v\in\sV_\infty}X(\kk_v)$). Then every obit of
$G(\kk_\infty)$ in $X(\kk_\infty)$ contains exactly one orbit of
$G(\kk)$ in $X(\kk)$. In particular, any two $\kk$-points in the
same connected component of $X(\kk_\infty)$ are $G(\kk)$-conjugate.
\end{theorem}

\begin{proof}
(i) Let $\gg$ denote the image of $\Gal(\kbar/\kk)$ in $\Aut\
\pi_1(\Hbar)\times\Aut\ \pi_1(\Gbar)$. Since $\pi_1(\Hbar)$ embeds
into $\pi_1(\Gbar)$, we can say that  $\gg$ is the image of
$\Gal(\kbar/\kk)$ in $\Aut\ \pi_1(\Gbar)$. We have seen in the proof
of Theorem \ref{thm:main-finite-adjoint} that $G$ does not come from
triality. Thus either $\gg=1$ or $\gg=\ZZ/2\ZZ$. We see that $\gg$
is cyclic. It follows that all the decomposition groups $\gg_v$ are
cyclic. Condition (iv) of Theorem \ref{thm:main-f-m-orbits} shows
now that $\ker[\pi_1(\Hbar)_{\gg_v}\to\pi_1(\Gbar)_{\gg_v}]=0$ for
any $v$ (because $\gv$ is cyclic for any $v$). It follows that
$\ker[H^1(\kk_v,H)\to H^1(\kk_v,G)]=1$ for $v\in\sV_f$, hence there
is only one orbit of $G(\kk_v)$ in $X(\kk_v)$ for such $v$, which
proves (i).

(ii) $G$ is an absolutely almost simple $\kk$-group. By
\cite[Cor.~5.4]{Sansuc} $G$ satisfies the Hasse principle and has
the weak approximation property. Since $H$ is a connected
$\kk$-subgroup of $G$, by \cite[Cor.~1.7]{Bor99} 
$X$ has the real approximation property, 
i.e.  any orbit of $G(\kinf)$ in $X(\kinf)$  contains a
$\kk$-point.

Now let $x,y\in X(\kk)$ lie in the same $G(\kinf)$-orbit. We wish to
prove that they lie in the same $G(\kk)$-orbit. Our homogeneous
space $X=H\backslash G$ has a distinguished $\kk$-point $x_0$, the
image of the unit element $e\in G(\kk)$. The stabilizer of $x_0$ in
$G$ is $H$. Let $H_y$ denote the stabilizer of $y$ in $G$. Clearly
the pair $(G,H_y)$ is a symmetric pair satisfying the hypotheses of
Theorem \ref{thm:main-finite-adjoint} (or Theorem
\ref{thm:symmetric-sc}). We may and shall assume that $y=x_0$.

So let $x\in X(\kk)$ lie in the $G(\kinf)$-orbit of $x_0$. We wish
to prove that $x$ lies in the $G(\kk)$-orbit of $x_0$. Let $c(x)$
denote the class of the $G(\kk)$-orbit of $x$ in $\ker(\kk,H\to
G):=\ker[H^1(\kk,H)\to H^1(\kk,G)]$ (we use the notation of the
proof of Theorem \ref{thm:main-f-m-orbits}). For any place $v$ of
$\kk$ let
$$
\loc_v\colon\ker(\kk,H\to G)\to\ker(\kk_v,H\to G)
$$
be the localization map. By (i) for any finite place $v$ of $\kk$ we
have $\ker(\kk_v,H\to G)=1$, hence $\loc_v(c(x))=1$. Since $x$ lies
in the $G(\kinf)$-orbit of $x_0$, we have $\loc_v(c(x))=1$ for all
infinite places $v$ of $\kk$.

Set $B=\ker[H\sc\to H]$. By \cite[Cor.~4.4]{Sansuc} there is a
canonical bijection $\Sha^1(\kk,H)\isoto \Sha^2(\kk,B)$. From the
lists of Theorems \ref{thm:symmetric-sc} and
\ref{thm:main-finite-adjoint} we see that in our case either $B=0$
or $B=\ZZ/2\ZZ$. Since  in both cases $\Sha^2(\kk,B)=0$, we conclude
that $\Sha^1(\kk,H)=1$. This means that $\ker\left[\loc\colon
H^1(\kk,H)\to\prod_v H^1(k_v,H)\right]=1$. We have seen that
$\loc(c(x))=1$. Hence $c(x)=1$. This means that $x$ lies in the
$G(\kk)$-orbit of $x_0$. This completes the proof of Theorem
\ref{thm:cor-orbits}.
\end{proof}
%%%%%%%%%%%%%%%%%%%%%%%%%%%%%%%%%%%%%%%%%%%%%%%%%%%%%%%%%%%%%%%%%%%%

\subsection{Addendum: Further examples}\label{sec:6}
In this addendum we give examples of homogeneous spaces satisfying
assumptions (i--iii) of Theorem \ref{maint} but not covered by
Theorems \ref{thm:symmetric-sc} and \ref{thm:main-finite-adjoint}.

\begin{subsec}\label{subsec:not-absolutely-simple}
\emph{Example with $G$ not absolutely simple.} Let $K$ be a number
field,
  $K'/K$ a quadratic extension, $D/K'$ a central simple algebra
of dimension $r^2$ with an involution of second kind $\sigma$ (i.e.
$\sigma$ induces the nontrivial automorphism $\sigma_0$ of $K'$ over
$K$). Let $m$ be a natural number and let $\Phi$ be a
$\sigma$-Hermitian form on $D^m$. Set
$$
G=\PSL_D(D^m),\quad H=\PSU(D^m,\Phi),
$$
where we regard $G$ and $H$ as $K$-groups. Then $G$ is adjoint and
$H$ is a symmetric subgroup of $G$. An easy calculation
shows that $\pi_1(\Gbar)=\Zz/n\Zz\oplus\Zz/n\Zz$ and
$\pi_1(\Hbar)=\Zz/n\Zz$, where $n=mr$. The group $\Gal(\kbar/K')$
acts trivially on $\pi_1(\Gbar)$ and $\pi_1(\Hbar)$. The
non-identity element $\sigma_0\in\Gal(K'/K)$ acts on $\pi_1(\Hbar)$
by multiplication by $-1$. Thus $\geff=\Gal(K'/K)$ and
$\pi_1(\Hbar)_\geff=(\Zz/n\Zz)/2(\Zz/n\Zz)$.

Now assume that $n$ is odd (i.e. both $r$ and $m$ are odd). Then\\
$(\Zz/n\Zz)/2(\Zz/n\Zz)=0$, hence $\pi_1(\Hbar)_\geff=0$ and the
homomorphism
$$
\pi_1(\Hbar)_\geff\to\pi_1(\Gbar)_\geff
$$
is injective. Since $H\sc$ embeds into $G\sc$, by Corollary
\ref{cor:H'} the homomorphism $\pi_1(\Hbar)\to\pi_1(\Gbar)$ is also
injective. By Theorem \ref{thm:main-f-m-orbits} the set
$X(\kk)/G(\kk)$ is finite and for almost all $v$ the group
$G(\kk_v)$ acts transitively on $X(\kk_v)$.

Similar examples can be constructed for $G$ and $H$ of type $E_6$
(then $\pi_1(\Hbar)=\Zz/3\Zz$) and for $G$ and $H$ of types $E_8,\
F_4$ and $G_2$.
\end{subsec}

\begin{subsec}\label{subsec:spherical}
\emph{Examples with spherical non-symmetric $H$.} We are interested
in examples of pairs $(G,H)$ over a number field $\kk$
such that $(\Gbar,\Hbar)$ is a non-symmetric spherical pair,
the group $G$ is connected, adjoint, and absolutely simple, the subgroup $H$ is a
connected semisimple subgroup of $G$, and $H$ is a
maximal connected subgroup of $G$.
Spherical pairs $(\Gbar, \Hbar)$ are listed in  \cite[Table 1]{Kr}.
See \cite[Ch.~I, \S 3, Table 1]{Vinberg} for a list of non-symmetric spherical pairs.
From the latter table we  see that there are only
two non-symmetric spherical pair $(\Gbar,\Hbar)$ over $\kbar$ 
such that $\Gbar$ is connected, adjoint, and simple,  $\Hbar$ is a
connected semisimple subgroup of $\Gbar$, and $\Hbar$ is a
maximal connected subgroup of $\Gbar$: they are $(\SO_7, G_2)$ and $(G_2,\SL_3)$.
For the embeddings $\Hbar\into \Gbar$ in these examples 
we refer to the cited tables.
Thus over $\kk$ we obtain pairs $(G,H)$ where $(\Gbar,\Hbar)$ is a spherical non-symmetric pair and 

\par (a) $G$ is a form of $\SO_7$, $H$ is a form of $G_2$; or
\par (b) $G$ is a form of $G_2$, $H$ is a form of $\SL_3$.

In both cases $\pi_1(\Hbar)=0$. By Corollary
\ref{cor:sc-finite} the set  $X(\kk)/G(\kk)$ is finite and for
almost all $v$ the group $G(\kk_v)$ acts transitively on $X(\kk_v)$.
\end{subsec}

\begin{subsec}\label{subsec:non-spherical}
\emph{Examples with non-spherical $H$.} There are lots of pairs
$(G,H)$ satisfying assumptions (i--iii) of Theorem \ref{maint}.
Indeed, let $H$ be a connected, simply connected, absolutely almost simple
$K$-group. Let $\rho \colon H\to \GL(V)$ be an absolutely
irreducible faithful representation of $H$ in a vector space $V$
defined over $K$ (if $H$ is split, then any irreducible
representation of $H$ defined over $\kbar$ can be defined over $K$).
If $\rho$ does not admit a nondegenerate invariant bilinear form, we
set $G=\GL(V)$. If $\rho$ admits a nondegenerate symmetric invariant
bilinear form $F_s$, we set $G=\SO(V,F_s)$. If $\rho$ admits a
nondegenerate alternating invariant bilinear form $F_a$, we set
$G=\Sp(V,F_a)$.

Since $\rho$ is faithful, we may regard $H$ as a subgroup of $G$.
For almost all $\rho$ the subgroup $\Hbar$ is a maximal connected
$\kbar$-subgroup of $\Gbar$, with a small number of exceptions, see
\cite[Thm.~1.5]{Dynkin} (see also \cite[Thm.~6.3.3]{OV94}). Note
that if $\Hbar$ is a maximal connected subgroup of $\Gbar$, then
clearly $H$ is a maximal connected $K$-subgroup of $G$. Set
$X=H\backslash G$. Since $H$ is simply connected, by Corollary
\ref{cor:sc-finite} the set $X(K)\backslash G(K)$ is finite and
$G(K_v)$ acts transitively on $X(K_v)$ for almost all $v$. Thus for
all faithful absolutely irreducible representations $\rho$ such that
$H$ is a maximal connected subgroup of $G$, we get  pairs $(G,H)$
satisfying assumptions (i--iii) of Theorem \ref{maint}.

\end{subsec}

\end{document}